\documentclass[preprint,3p,times]{elsarticle}
\usepackage[toc,page,title,titletoc,header]{appendix}
\usepackage{amsmath}
\usepackage{amssymb}
\usepackage{amsthm}
\usepackage{enumerate}
\usepackage{enumitem}
\usepackage{mathrsfs}
\usepackage{graphicx}
\usepackage{epsfig}
\usepackage{tikz,pgfplots,tikz-3dplot}
\usepackage{epstopdf}
\usepackage{microtype}
\usetikzlibrary{fit}
\biboptions{sort&compress}

\newtheorem{thm}{Theorem}[section]

\newtheorem{rem}{Remark}
\renewcommand{}

\newcommand{\nn}{\nonumber}

\usepackage{cases}

\def\epsilon{\varepsilon} 
\newcommand{\mat}[1]{\boldsymbol{#1}}

\allowdisplaybreaks
\begin{document}
\begin{frontmatter}
\title{%\texttt{\jobname} \textbf{DRAFT} \quad {\rm \mydate} \\
Efficient energy-stable parametric finite element methods for surface diffusion flow and applications in solid-state dewetting}

\author[1]{Meng Li}
\address[1]{School of Mathematics and Statistics, Zhengzhou University,
Zhengzhou 450001, China}
\ead{limeng@zzu.edu.cn}
\author[1]{Yihang Guo}
\author[1]{Jingjiang Bi}
%\author[1]{Quan Zhao}
\begin{abstract}
    Currently existing energy-stable parametric finite element methods for surface diffusion flow and other flows are usually limited to first-order accuracy in time. Designing a high-order algorithm for geometric flows that can also be theoretically proven to be energy-stable poses a significant challenge. Motivated by the new scalar auxiliary variable approach \cite{huang2020highly}, we propose novel energy-stable parametric finite element approximations for isotropic/anisotropic surface diffusion flows, achieving both first-order and second-order accuracy in time. Additionally, we apply the algorithms to simulate the solid-state dewetting of thin films. 
Finally, extensive numerical experiments validate the accuracy, energy stability, and efficiency of our developed numerical methods.   The designed algorithms in this work exhibit strong versatility, as they can be readily extended to other high-order time discretization methods (e.g., BDFk schemes). Meanwhile, the algorithms achieve remarkable computational efficiency and maintain excellent mesh quality. More importantly, the algorithm can be theoretically proven to possess unconditional energy stability, with the energy nearly equal to the original energy. 
\end{abstract}
\end{frontmatter}
\section{Introduction}\label{sec:intro}\numberwithin{figure}{section}
\numberwithin{equation}{section}%(按章节编号)
Surface diffusion (SDF) involves the movement and migration of surface atoms, atomic clusters, and molecules on material surfaces and interfaces in solids. This phenomenon is widely studied in materials and surface science \cite{oura2013}, and it is crucial for various processes such as thin film growth, catalysis, epitaxial growth, and the formation of surface phases \cite{chang90}.
SDF can be categorized based on the orientation of the surface lattice, leading to either isotropic or anisotropic SDF. Anisotropic SDF, in particular, has extensive applications in materials science and solid-state physics, including the crystal growth of nanomaterials \cite{Cahn91, gomer1990}, morphology development in alloys, and solid-state dewetting (SSD) \cite{Jiran90, jiang2012phase, Jiang19a}.

The SSD process is a significant application of SDF occurring in solid-solid-vapor systems. In these systems, the solid film adhering to the surface is often unstable or metastable in its as-deposited state, leading to complex morphological evolution driven by surface tension and capillarity effects, including edge retraction \cite{Wong00,Dornel06,hyun2013quantitative}, faceting \cite{Jiran90,Jiran92,ye2010mechanisms} and fingering instabilities \cite{Kan05,Ye10b,Ye11a,Ye11b}. 
This phenomenon, commonly observed in various thin film/substrate systems, characterized by the maintenance of the thin film in a solid state during the process \cite{thompson12solid,Zucker13,Rabkin14}, is known as SSD. Recently, SSD has found extensive applications in modern technology. 
For example, SSD of thin films in micro-/nanodevices can lead to the surface instabilities of well-prepared patterned structures; however, they can be leveraged for generating well-defined patterns of nanoscale particle arrays. These arrays are subsequently applied in sensors \cite{Mizsei93}, optical and magnetic devices \cite{Armelao06}, as well as catalysts for the growth of carbon and semiconductor nanowire \cite{Schmidt09}.
\begin{figure}[h]
    \centering
    \includegraphics[width=0.5\linewidth]{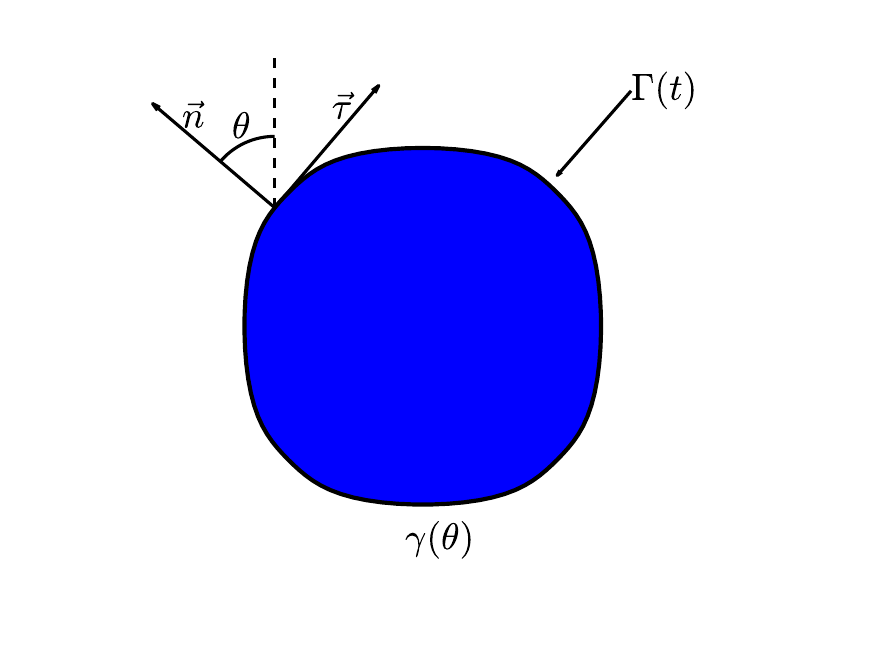}
    \caption{An illustration of SDF on a closed curve $\Gamma(t)$ with anisotropic surface energy density in two dimensions.}
    \label{fig:illustration of surface diffusion}
\end{figure}

 As depicted in Figure \ref{fig:illustration of surface diffusion}, $\Gamma(t)=(x(s,t),y(s,t))^{T}$ denotes a closed curve in two-dimensional space, $s$ is the arc length parametrization of $\Gamma(t)$, $\Vec{n}= (-\sin\theta,\cos\theta)^T$ represents the unit outward normal vector to the curve with $\theta \in [-\pi,\pi]$ being the angle between $\Vec{n}$ and $y$-axis, $\Vec{\tau}$ denotes the unit tangent vector, and $\gamma(\theta)$ represents the surface energy density function. From  \cite{Cahn94,fonseca2015}, anisotropic SDF is governed by the following partial differential equation:
\begin{equation}\label{eqn:SDF_eqn}
    \partial_t \Vec{X} = \partial_{ss}\mu\,\Vec{n},
\end{equation}
where $\mu:=\mu(s,t)$ denotes the chemical potential (or weighted mean curvature), defined by
\begin{equation}\label{eqn:mu_def}
    \mu= \left[\gamma(\theta)+{\gamma}''(\theta)\right]\kappa,
\end{equation}
with $\kappa$ representing the curvature of $\Gamma(t)$, defined by $\kappa = -\partial_{ss}\Vec{X}\cdot \Vec{n}$. The dynamic evolution of SSD is governed by SDF and contact line migration. A detailed discussion on the boundary control conditions of SSD is presented in Section \ref{sec:SSD}. 

There exists a substantial body of numerical methods for geometric evolution equations, extensively detailed in references such as \cite{Dziuk90, ColemanFM96, Wong00, Bansch04, Bansch05, hausser2007discrete, Barrett07, Barrett08JCP, Kemmochi17, Kovacs2019, Kovacs2020, Jiang2021, Zhao2021, BGNZ21volume, BJY21symmetrized, hu22evolving}. 
For comprehensive discussions, we refer to the review articles \cite{Deckelnick05, Barrett20}. Over the past 30 years, numerous numerical approximations have been proposed for the SDF, including the marker-particle method \cite{Wong00, du2010tangent}, the finite element method (FEM) via graph representation \cite{Bansch04, Deckelnick05, Deckelnick05fully}, the discontinuous Galerkin FEM \cite{xu2009local}, and the parametric FEM (PFEM) \cite{Barrett07, Barrett08Ani, Barrett19, hausser2007discrete, BGNZ23, Bao17, Li2020energy}.
Among these methods, the numerical method proposed by Barrett, Garcke, and Nürnberg (BGN) \cite{Barrett07, Barrett08JCP} is recognized as one of the most efficient and accurate methods for the SDF. The BGN scheme offers significant advantages, including the capability to demonstrate the energy stability of the scheme and the ability to maintain high mesh quality. It has undergone extensive development and has been successfully applied to solve the SDF with anisotropic surface energy \cite{Barrett07Ani, Barrett08Ani, Li2020energy, bao2023symmetrized1, bao2024structure}. Furthermore, mesh quality can also be sustained by introducing an artificial tangential velocity \cite{mikula04, hu22evolving, duan24}.
Moreover, the BGN method has been significantly extended to effectively address SSD problems, encompassing both isotropic \cite{Zhao20, Zhao2020p} and anisotropic surface energy \cite{Wang15, Bao17, Jiang19a, JiangZB20}.
However, we note that the temporal discretization methods in the aforementioned PFEMs are all first-order schemes. 
Recently, Jiang et al. \cite{jiang24,jiang24stable} attempted to develop high-order BGN methods based on leapfrog and BDFk schemes, but no theoretical results have been obtained.
Constructing a high-order temporal discretized PFEM with energy-stable property remains a significant challenge. 

The challenge brings to mind two very popular methods: the invariant energy quadratization (IEQ) method \cite{chen2019efficient,yang2016linear,chen2019fast} and the scalar auxiliary variable (SAV) method \cite{Shen18,huang2020highly,huang2022new}. Both methods all introduce an auxiliary variable, but the SAV method relaxes the lower bound of the nonlinear free energy potential compared to the IEQ method. 
The SAV method has experienced its rapid evolution process. Compared to the original SAV method proposed in \cite{Shen18,shen2019new}, Huang et al. in \cite{huang2022new} presented a new SAV method that can reduce computational costs by approximately half and maintain unconditional dissipation of the modified energy under high-order BDFk schemes. Moreover, the new SAV method offers several other advantages. On the one hand, it provides rough error estimates for BDFk ($1\le k \le5$) schemes. On the other hand, it can be widely applied to dissipative systems \cite{huang2023stability,yang2024linear,li2023error}. Although the SAV method has been extensively developed, there is currently no literature on solving geometric flow problems using the SAV method. 
This is the primary focus of our article. In addition to effectively constructing the first-order time-accurate energy-stable parametric finite element approximation, we also develop the first-order time-accurate approximately area-conservation scheme, and the high-order time-accurate energy-stable algorithm. Compared to the works presented in \cite{jiang24,jiang24stable}, this high-order time scheme can achieve certain theoretical results and also accommodate flows with anisotropic surface energy. As an extension, these energy-stable methods are also utilized to simulate anisotropic SSD problems.

The rest of the paper is organized as follows. In Section \ref{sec:isotropic}, we introduce novel energy-stable parametric finite element approximations for isotropic SDF, including BDF1-SAV, BDF1-CSAV and BDF2-SAV methods. In Section \ref{sec:anisotropic}, we further extend these numerical schemes to anisotropic SDF. In Section \ref{sec:SSD}, we discuss boundary conditions for SSD and extend the three energy-stable schemes to address SSD problems. We present extensive experiments to validate the efficiency and accuracy of the proposed schemes in Section \ref{sec:experiment}. Finally, we draw some conclusions in Section \ref{sec:conclusion}.
%% 插图1(表面扩散插图-引言)
%\begin{figure}[h]
%    \centering
%\includegraphics[width=0.5\textwidth]{surface diffusion illustration .png}
%    \caption{An illustration about surface diffusion of a closed curve with anisotropic surface energy $\gamma(\theta)$ in two dimensions, while $\theta$ is the angle between the outward unit normal vector $\Vec{n}$ and the y-axis. }
  %  \label{fig:illustration of surface diffusion}
%\end{figure}
   %% 插图1(表面扩散插图-引言)
   %%第二章(各向同性曲率流)

\section{For isotropic SDF}\label{sec:isotropic}
\numberwithin{equation}{section}%(按章节编号)
 In this section, motivated by \cite{huang2022new}, we propose three types of energy-stable PFEMs for the isotropic SDF, which maintains first-order or second-order accuracy in time. 
\subsection{The new formulation }
We firstly consider the isotropic SDF, which can be regarded as the $H^{-1}$-gradient flow of the length functional $L(t)$: 
   %%2.1(各向同性曲率流新公式)
\begin{subequations}\label{eqn:is_model1}
  \begin{align}
     & \partial_t\Vec{X}\cdot \vec n=\partial_{ss}\kappa,\label{eqn:isdf1}\\
     & \kappa\Vec{n}=-\partial_{ss}\Vec{X},\label{eqn:isdf2}
  \end{align}
\end{subequations}
where $0<s<L(t)$ is the arc length parameter  with $L(t):=\int_{\Gamma(t)} 1 \,ds$ being the perimeter of $\Gamma(t)$, and $\kappa:=\kappa(s, t)$ is the curvature of the interface curve. 
  %%2.1(各向同性曲率流新公式)
We introduce a time independent variable $\rho\in\mathbb{I}=[0,1]$. Then, the arc length parameter $s$ can be computed by $s(\rho,t)=\int_0^\rho |\partial_q \Vec{X}|\, dq$. We parameterize the evolution closed curves $\Gamma(t)$ as:
\begin{equation*}
    \Vec{X}(\rho,t):=\left(x(\rho,t),y(\rho,t)\right)^T: \mathbb{I}\times[0,T] \to \mathbb{R}^2.  
\end{equation*}
According to the parameterization above, we can equivalently represent $ds=\partial_{\rho}sd\rho$, where $\partial_{\rho}s=|\partial_{\rho}\Vec{X}|$.
Then we define the functional space with respect to the closed curve $\Gamma(t)$ as
\begin{equation}\label{eqn:fs}
    L^2(\mathbb{I}):=\left\{u:\mathbb{I}\to \mathbb{R} \bigg| \int_{\Gamma(t) }|u(s)|^2 \,ds=\int_{\mathbb{I}}|u(s(\rho,t))|^2\partial_\rho s \, d\rho < \infty\right\},
\end{equation}
equipped with the $L^{2}$-inner product 
\begin{equation}\label{eqn:innerp}
    (u,v)_{\Gamma(t)}:=\int_{\Gamma(t)} u(s)v(s)\, ds=\int_{\mathbb{I}} u(s(\rho,t)) v(s(\rho,t)) \partial_\rho s \,d\rho,\quad \forall u,v\in L^2 (\mathbb{I}).
\end{equation}
 The definition \eqref{eqn:innerp} can be directly extended to $[L^2(\mathbb{I})]^2$.

For $t\ge 0$, $A(t)$ represents the area of the domain enclosed by $\Gamma(t)$, and $L(t)$ is the perimeter of $\Gamma(t)$. Then, using Reynolds transport theorem, we can obtain  
\begin{subequations} \label{eqn:is_str}
    \begin{align} 
        &\label{eqn:is_stra}\frac{d}{dt}A(t)=\int_{\Gamma(t)} \partial_t\Vec{X}\cdot \Vec{n}\,ds=\int_{\Gamma(t)} \partial_{ss} \kappa\, ds \equiv 0, \quad \forall t \ge 0,\\ & \label{eqn:is_strb}\frac{d}{dt}L(t)=\int_{\Gamma(t)} (\partial_t \Vec{X} \cdot \Vec{n})\kappa \,ds=-\int_{\Gamma(t)} (\partial_s \kappa)^2\,ds \le 0, \quad \forall t\ge 0,
    \end{align}
\end{subequations}
 i.e., area conservation and energy (perimeter) dissipation.

We next introduce the following new time-dependent variables: 
\begin{align}
    R(t)=L(t)\quad \text{and}\quad \xi(t)=\frac{R(t)}{L(t)}\equiv 1.
\end{align}
Then, from \eqref{eqn:is_strb}, we have 
\begin{align}
  \frac{d}{dt}R(t)= \frac{d}{dt}L(t)= -\xi(t) \int_{\Gamma(t)} (\partial_s \kappa)^2\,ds.
\end{align}
Therefore, we can reformulate the system \eqref{eqn:is_model1} into the following expanded form:
\begin{subequations} \label{eqn:is_model2}
\begin{align}
   & \label{/eqn:is_sdf2a}\partial_t \Vec{X} \cdot \Vec{n}=\partial_{ss}\kappa,\\
       \label{/eqn:is_sdf2b}& \kappa \Vec{n}=-\partial_{ss} \Vec{X},\\ \label{/eqn:is_sdf2c} 
       & \frac{d}{dt}R(t)=-\xi(t)\int_{\Gamma(t)} (\partial_s \kappa)^2 \,ds.
\end{align} 
\end{subequations}

Define the Sobolev spaces as 
\begin{align*}
&H^1(\mathbb{I}):=\left \{ u:\mathbb{I}\to\mathbb{R}~|\hspace{0.5em}u\in L^2(\mathbb{I})\nn\
\text{and}\ \partial_\rho u \in L^2(\mathbb{I}) \right \}\quad\text{and}\quad \mathbb{X}:=H^1(\mathbb{I})\times H^1(\mathbb{I}).
\end{align*}
Then the weak formulation of the system \eqref{eqn:is_model2} is given as follows: 
Given the initial closed curve $\Gamma(0)=\Vec{X}(\mathbb{I},0)\in \mathbb{X}$, to find the evolution curve $\Gamma(t)=\Vec{X}(\mathbb{I}, t)\in \mathbb{X}$ and the curvature $\kappa(\cdot, t)\in H^1(\mathbb{I})$ for $t>0$, such that 
\begin{subequations}\label{eqn:is_weak}
      \begin{align}
          &\label{eqn:is_weaka}(\partial_t\Vec{X}\cdot \Vec{n},\varphi)_{\Gamma(t)}+(\partial_s \kappa,\partial_s\varphi)_{\Gamma(t)}=0, \qquad  \forall \varphi \in H^1(\mathbb{I}),\\&\label{eqn:is_weakb}
          (\kappa\Vec{n},\Vec{\omega})_{\Gamma(t)}-\left(\partial_s\Vec{X},\partial_s\Vec{\omega}
          \right)_{\Gamma(t)}=0,\qquad \forall \Vec{\omega} \in \mathbb{X}.
      \end{align}
  \end{subequations}
Obviously, we can prove that 
the new defined system \eqref{eqn:is_model2} 
and corresponding weak formulation \eqref{eqn:is_weak}
also hold the area conservation and energy  dissipation properties that are defined in 
\eqref{eqn:is_str}.

 \subsection{Numerical schemes}\label{space_dis}

 For a positive integer $N > 2$, let $h=\frac{1}{N}$ denote the grid size. With this, the reference domain $\mathbb{I}$ can be uniformly partitioned into subintervals $\mathbb{I}=\cup_{j=1}^{N} \mathbb{I}_j$, where each subinterval $\mathbb{I}_j=[\rho_{j-1}, \rho_j]$ with $\rho_j=jh $ for $j=0,...,N$. Then, we define the finite element subspace as
 \begin{equation*}
     V^h({\mathbb{I}}):=\left\{u\in C(\mathbb{I}): u|_{\mathbb{I}_j} \in \mathbb{P}^1({\mathbb{I}_j}), \forall j=1,2,...,N \right\}\subseteq H^1(\mathbb{I}),
 \end{equation*}
 where $\mathbb{P}^1(\mathbb{I}_j)$ denotes the space of polynomials on subinterval $\mathbb{I}_j$ with degree at most $1$.

 Let $\Gamma^h(t):=\Vec{X}^h(\cdot,t) \in [V^h(\mathbb{I})]^2$ and $\kappa^h(\cdot,t) \in V^h(\mathbb{I})$ be the numerical approximations of the closed curve $\Gamma(t):=\Vec{X}(\cdot,t)\in \mathbb{X}$ and $\kappa(\cdot,t) \in H^1(\mathbb{I})$, respectively. For $t\ge 0$, the closed curve $\Gamma^h(t)$ is indeed composed by the connected line segments $\{\Vec{h}_j(t)\}_{j=1}^{N}$. To ensure the non-degenerate meshes, we always assume
 \begin{equation}
     h_{\min}(t):=\min_{1\le j \ge N}\left|\Vec{h}_j (t)\right|>0 \quad \text{with} \quad \Vec{h}_j(t):=\Vec{X}^h(\rho_j,t)-\Vec{X}^h(\rho_{j-1},t), \quad j=1,2,...,N,
 \end{equation}
 where $|\Vec{h}_j(t)|$ denotes the length of the vector $\Vec{h}_j(t)$ for $j=1,2,...,N$.

Given two piecewise linear scalar (or vector) functions $u,v$ defined on the interval $\mathbb{I}$ with possible jumps at the nodes $\{ \rho_j \}_{j=0}^{N}$, we define the mass-lumped inner product $(\cdot,\cdot)_{\Gamma^h(t)}^h$ over $\Gamma^h(t)$ as
\begin{equation} \label{eqn:spaceinnerp}
    (u,v)_{\Gamma^{h}(t)}^{h}:=\frac{1}{2}\sum_{j=1}^{N}|\Vec{h}_j|\left[(u\cdot v)(\rho_j^-)+(u\cdot v)(\rho_{j-1}^+)\right],
\end{equation}
where $u(\rho_j^{\pm})=\lim\limits_{\rho \to \rho_j^{\pm}}u(\rho)$ for $0\le j\le N $.

We then establish a spatial semi-discrete scheme for the weak formulation \eqref{eqn:is_weak} as follows: Given the initial curve $\Gamma^h(0):=\Vec{X}^h(\cdot,0)\in [V^h(\mathbb{I})]^2$, find the closed curve $\Gamma^h(t):=\Vec{X}^h(\cdot,t) \in [V^h(\mathbb{I})]^2$ and the curvature $\kappa^h(\cdot,t) \in V^h(\mathbb{I})$, such that
\begin{subequations}\label{eqn:is_spaceweak}
      \begin{align}
&\label{eqn:is_spaceweaka}\left(\partial_t\Vec{X}^h\cdot \Vec{n}^h,\varphi^h\right)_{\Gamma^h(t)}^h+\left(\partial_s \kappa^h,\partial_s\varphi^h\right)_{\Gamma^h(t)}^h=0, \qquad  \forall \varphi^h \in V^h(\mathbb{I}),\\&\label{eqn:is_spaceweakb}
          \left(\kappa^h\Vec{n}^h,\Vec{\omega}^h\right)_{\Gamma^h(t)}^h-\left(\partial_s\Vec{X}^h,\partial_s\Vec{\omega}^h
          \right)_{\Gamma^h(t)}^h=0,\qquad \forall \Vec{\omega}^h \in [V^h(\mathbb{I})]^2.
      \end{align}
  \end{subequations}
Based on the spatial semi-discrete scheme \eqref{eqn:is_spaceweak}, we design the following fully-discrete schemes with first-order and second-order in time. To this end, 
we divide the time domain $[0,T]=\cup_{m=0}^{M-1} [t_{m},t_{m+1}]$, with $M$ be a positive integer, and $0=t_0<t_1<...<t_M=T$ with uniform-time-steps $\Delta t:= t_{m+1}-t_{m}$. Then, for $m\ge0$, take $\Gamma^m=\Vec{X}^m(\rho)\in [V^h(\mathbb{I})]^2$ and $\kappa^m \in V^h(\mathbb{I}) $ as the numerical approximations of $\Gamma^h(t_m)=\Vec{X}^h(\cdot,t_m)$ and $\kappa^h(\cdot,t_m)$ respectively.   
\begin{itemize}
\item {\textbf{First-order time-accurate scheme: BDF1-SAV}}. 
Referring to the spatial semi-discrete scheme \eqref{eqn:is_spaceweak}, the full-discrete scheme of the weak formulation \eqref{eqn:is_model2}, with first-order accuracy in time, is as follows: Given the initial curve $\Gamma^{0}:=\Vec{X}^{0}(\cdot)\in [V^h(\mathbb{I})]^2$, find the closed curve $\overline{\Gamma}^{m+1}:=\Vec{\overline{X}}^{m+1}(\cdot) \in [V^h(\mathbb{I})]^2$ and the curvature $\overline{\kappa}^{m+1}(\cdot)\in V^h(\mathbb{I})$, such that
\begin{subequations}\label{eqn:first_dis}
    \begin{align}\label{eqn:first_disa}
        &\left(\frac{\Vec{\overline{X}}^{m+1}-\Vec{X}^{m}}{\Delta t}\cdot \Vec{n}^{m},\varphi^h \right)_{\Gamma^m}^h + \left(\partial_s \overline{\kappa}^{m+1},\partial_s\varphi^h \right)_{\Gamma^m}^h=0,\qquad 
        \forall \varphi^h\in V^h(\mathbb{I}),\\ \label{eqn:first_disb}
        &\left(\overline{\kappa}^{m+1}\Vec{n}^m, \Vec{\omega}^h\right)_{\Gamma^m}^h-\left(\partial_s\Vec{\overline{X}}^{m+1},\partial_s\Vec{\omega}^h\right)_{\Gamma^m}^h=0,\qquad 
        \forall \Vec{\omega}^h\in [V^h(\mathbb{I})]^2.
    \end{align}
\end{subequations}
Then, we define the discrete perimeter $L^h(\Vec{X}^{m})$ as
\begin{equation}\label{eqn:dis_perimeter}
    L^h(\Vec{X}^{m}):=\sum_{j=1}^{N}\big|\Vec{h}_j^m\big|.
\end{equation}
Therefore, by adopting the first-order finite difference method to solve \eqref{/eqn:is_sdf2c}, we obtain
\begin{equation}\label{eqn:ener_diff}
\frac{R^{m+1}-R^{m}}{\Delta t}=-\xi^{m+1}\left(\partial_s\overline{\kappa}^{m+1},\partial_s\overline{\kappa}^{m+1} \right)_{\overline{\Gamma}^{m+1}}^{h} \quad \text{with} \quad \xi^{m+1}=\frac{R^{m+1}}{L^h(\Vec{\overline{X}}^{m+1})}.
\end{equation}
To maintain temporal accuracy, we introduce a time-dependent variable:
\begin{equation}\label{eqn:eta_iter}
    \zeta^{m+1}=1-(1-\xi^{m+1})^r,
\end{equation}
where $r\geq 2$ is an undetermined constant. 
Then, the closed curve ${\Gamma}^{m+1}:=\Vec{{X}}^{m+1}(\cdot) \in [V^h(\mathbb{I})]^2$ and the curvature ${\kappa}^{m+1}(\cdot)\in V^h(\mathbb{I})$ are computed by 
 \begin{equation}\label{eqn:realX}
\Vec{X}^{m+1}=\zeta^{m+1}\Vec{\overline{X}}^{m+1}\quad \text{and} \quad \kappa^{m+1} =\zeta^{m+1}\overline{\kappa}^{m+1}.
 \end{equation}
\item {\textbf{First-order time-accurate scheme with approximately area-conservation: BDF1-CSAV.}}
Given $\Vec{X}^m\in V^h(\mathbb{I})$ and $R^m$, find $(\Vec{\overline{X}}^{m+1}, \overline{\kappa}^{m+1})\in [V^h(\mathbb{I})]^2\times V^h(\mathbb{I})$, $\xi^{m+1}$, $R^{m+1}$ and $(\Vec{X}^{m+1}, \kappa^{m+1})\in [V^h(\mathbb{I})]^2\times V^h(\mathbb{I})$, such that
\begin{subequations}\label{eqn:first_dis_ac}
    \begin{align}\label{eqn:first_disa_ac}
        &\left(\frac{\Vec{\overline{X}}^{m+1}-\Vec{X}^{m}}{\Delta t}\cdot \Vec{\overline{n}}^{m+\frac12},\varphi^h \right)_{\Gamma^m}^h + \left(\partial_s \overline{\kappa}^{m+1},\partial_s\varphi^h \right)_{\Gamma^m}^h=0,\qquad 
        \forall \varphi^h\in V^h(\mathbb{I}),\\ 
    \label{eqn:first_disb_ac}
        &\left(\overline{\kappa}^{m+1}\Vec{\overline{n}}^{m+\frac12}, \Vec{\omega}^h\right)_{\Gamma^m}^h-\left(\partial_s\Vec{\overline{X}}^{m+1},\partial_s\Vec{\omega}^h\right)_{\Gamma^m}^h=0,\qquad 
        \forall \Vec{\omega}^h\in [V^h(\mathbb{I})]^2,\\
\label{eqn:first_disc_ac}
& \frac{R^{m+1}-R^{m}}{\Delta t}=-\xi^{m+1}\left(\partial_s\overline{\kappa}^{m+1},\partial_s\overline{\kappa}^{m+1} \right)_{\overline{\Gamma}^{m+1}}^{h} \quad \text{with} \quad \xi^{m+1}=\frac{R^{m+1}}{L^h(\Vec{\overline{X}}^{m+1})},\\
\label{eqn:first_disd_ac}
 & \zeta^{m+1}=1-(1-\xi^{m+1})^r,\\
\label{eqn:first_dise_ac}
 & \Vec{X}^{m+1}=\zeta^{m+1}\Vec{\overline{X}}^{m+1}\quad \text{and} \quad \kappa^{m+1} =\zeta^{m+1}\overline{\kappa}^{m+1},
\end{align}
\end{subequations}
where 
\begin{align}
    \Vec{\overline{n}}^{m+\frac{1}{2}}:=-\frac{\left(\partial_s \Vec{\overline{X}}^{m+1}+\partial_s \Vec{X}^{m}\right)^\bot}{2}=-\frac{\left(\partial_\rho \Vec{\overline{X}}^{m+1}+\partial_\rho \Vec{X}^{m}\right)^\bot}{2|\partial_\rho \Vec{X}^m|}.
\end{align}
    \item {\textbf{Second-order time-accurate scheme: BDF2-SAV.}} 
    Given $\Vec{X}^{m}\in V^h(\mathbb{I})$, $\Vec{X}^{m-1}\in V^h(\mathbb{I})$, and $R^{m}$, we can compute the $(\Vec{\overline{X}}^{m+1}, \overline{\kappa}^{m+1})\in [V^h(\mathbb{I})]^2\times V^h(\mathbb{I})$, $\xi^{m+1}$, $R^{m+1}$ and $(\Vec{X}^{m+1}, \kappa^{m+1})\in [V^h(\mathbb{I})]^2\times V^h(\mathbb{I})$ by using the following second-order temporal scheme:
    \begin{subequations}\label{eqn:sec_diff}
    \begin{align}\label{eqn:sec_diffa}
        &\left ( \frac{\frac{3}{2}\Vec{\overline{X}}^{m+1}-2\Vec{X}^{m}+\frac{1}{2}\Vec{X}^{m-1}}{\Delta t}\cdot \Vec{\tilde{n}}^{m+1},\varphi^{h}
         \right )_{\tilde{\Gamma}^{m+1}}^{h} +\left ( \partial_s \overline{\kappa}^{m+1},\partial_s \varphi^{h}\right )_{\tilde{\Gamma}^{m+1}}^{h}=0 ,\\&\label{eqn:sec_diffb}
         \left ( {\overline{\kappa}^{m+1}}\Vec{\tilde{n}}^{m+1},\Vec{\omega}^h \right )_{\tilde{\Gamma}^{m+1}}^{h}-\left ( \partial_s{\Vec{\overline{X}}}^{m+1},\partial_s\Vec{\omega}^h\right )_{\tilde{\Gamma}^{m+1}}^{h}=0, \\&\label{eqn:sec_diffc}\frac{R^{m+1}-R^{m}}{\Delta t}=-\xi^{m+1}\left (\partial_s\overline{\kappa}^{m+1},\partial_s\overline{\kappa}^{m+1}\right )_{\overline{\Gamma}^{m+1}}^{h}\quad \text{with}\quad
         \xi^{m+1}=\frac{R^{m+1}}{L^h(\Vec{\overline{X}}^{m+1})},  
         \\& \label{eqn:sec_diffd}
         \zeta^{m+1}=1-(1-\xi^{m+1})^r, \\&
         \label{eqn:sec_diffe}\Vec{X}^{m+1}=\zeta^{m+1}\Vec{\overline{X}}^{m+1}\quad \text{and} \quad \kappa^{m+1}=\zeta^{m+1}\overline{\kappa}^{m+1},
    \end{align}
\end{subequations}
where $r\ge 3$ and $\tilde{\Gamma}^{m+1}=\Vec{\tilde{X}}^{m+1}$ is computed from the $\Vec{X}^{m}$ using the first-order temporal scheme, which provides a  suitable approximation for the closed curve $\Gamma^{m+1}$. Since the scheme is not selfstarting, the first step value $\vec X^1$ can be given by the BDF1-SAV method.  
\end{itemize}

 \begin{rem}\label{rem:order_r}
      For above proposed schemes, in order to maintain $k$-order time accuracy, $r$ must satisfy $r\ge k+1$, $k=1$, $2$. Indeed, from \eqref{eqn:eta_iter}, there holds
    \begin{subequations}
          \begin{align*}
    &\frac{\alpha_{k}\Vec{\overline{X}}^{m+1}-A_k(\Vec{X}^{m})}{\Delta t}=\frac{\alpha_{k}\Vec{X}^{m+1}-A_k(\Vec{X}^{m})}{\Delta t}+\frac{\alpha_{k}\Vec{\overline{X}}^{m+1}-\alpha_k\Vec{X}^{m+1}}{\Delta t}\\&=\frac{\alpha_{k}\Vec{X}^{m+1}-A_k(\Vec{X}^{m})}{\Delta t}+\alpha_{k}\Vec{\overline{X}}^{m+1}\frac{1-\zeta^{m+1}}{\Delta t}=\frac{\alpha_{k}\Vec{X}^{m+1}-A_k(\Vec{X}^{m})}{\Delta t}+\alpha_{k}\Vec{\overline{X}}^{m+1}\frac{(1-\xi^{m+1})^r}{\Delta t},
          \end{align*}
      \end{subequations}
where $\alpha_1=1$, $A_1(\Vec{X}^{m})=\Vec{X}^{m}$ and $\alpha_2=\frac{3}{2}$, $A_2(\Vec{X}^{m})=2\Vec{X}^{m}-\frac{1}{2}\Vec{X}^{m-1}$. Due to \eqref{eqn:ener_diff}, $\xi^{m+1}$ is a first-order approximation to $1$. Therefore, to obtain the desired convergence order of the numerical scheme, we should select the suitable constant 
$r$ that holds the condition: $r\ge k+1$, $k=1$, $2$.
 \end{rem}

We can directly prove the following properties of the full-discrete schemes proposed above. 

\begin{thm}\label{thm:propertiesofiso}
For the BDF1-SAV, BDF1-CSAV and BDF2-SAV schemes,
given the energy $R^m\ge 0$, then there hold 
\begin{itemize}
    \item[(i)] $\xi^{m+1}\ge 0$ , $R^{m+1}\ge 0$;
    \item[(ii)] $\forall m\ge 0$, $R^{m+1}\le R^m$;
    \item[(iii)] $For\text{ } \forall r >0,\text{ }\exists M_r>0, \text{ }satisfying \text{ } that \text{ }\forall m \ge 0$, there holds $L^h(\Vec{X}^m) \le M_r$.
\end{itemize}
\end{thm}
\begin{proof}
  We don't intend to provide the proof here, since the results can be demonstrated following similar methods as ones presented in Section \ref{sec:anisotropic}. 
\end{proof}
     %%能量稳定性定理

\section{For anisotropic SDF}\label{sec:anisotropic}
\numberwithin{equation}{section}%(按章节编号)
In this section, we consider the anisotropic SDF, 
\begin{subequations}\label{eqn:model_ani}
 \begin{align}
 \label{eqn:modela_ani}
 &\partial_t\vec X = \partial_{ss}\mu\,\vec{n},\quad 0<s<L(t),\quad t>0,\\
&\mu=\left[\gamma(\theta)+\gamma''(\theta)\right]\kappa,\qquad \kappa = -(\partial_{ss}\vec X)\cdot\vec n,
 \label{eqn:modelb_ani}
 \end{align}
\end{subequations}
where $\mu$ is the chemical potential. 
The anisotropic SDF \eqref{eqn:model_ani} can be regarded as the $H^{-1}$-gradient flow of the energy functional $W(t)$:
\begin{equation}\label{eqn:ener_fun}
    W(t):=\int_{\Gamma(t)} \gamma(\theta)\, ds, \quad t\ge0.
\end{equation}
%%待改语句
Introduce a matrix $\mat{B}(\theta)$ as 
\begin{equation}\label{Matrix:Bqx}
\mat{B}(\theta)=
\begin{pmatrix}
\gamma(\theta)&-\gamma'(\theta) \\
\gamma'(\theta)&\gamma(\theta)
\end{pmatrix}
\begin{pmatrix}
\cos2\theta&\sin2\theta \\
\sin2\theta&-\cos2\theta
\end{pmatrix}+
\mathscr S(\theta)\left[\frac{1}{2}\mat{I}-\frac{1}{2}\begin{pmatrix}
\cos2\theta&\sin2\theta \\
\sin2\theta&-\cos2\theta
\end{pmatrix}\right].
\end{equation}
where $\mat I$ is a $2\times 2$ identity matrix and $\mathscr S(\theta)$ is a stability function to be determined later. Additionally, if the $\gamma(\theta)$ satisfies that $\gamma(\theta)=\gamma(\theta+\pi)$, the matrix $B(\theta)$ can be proven to be positive definite. 
\begin{rem}
    For the stability function $\mathscr S (\theta)$, there always exists a minimal stability function $\mathscr S_0(\theta)$. When $\mathscr S(\theta)\ge \mathscr S_0(\theta)$, it ensures that $\mat{B}(\theta)$ is a symmetric positive definite matrix. Further information about it will not be discussed here.
\end{rem}
Based on the definitions of above matrices, we can prove  \cite{li2024parametric}:
\begin{align}
\left[\gamma(\theta)+\gamma''(\theta)\right]\kappa\vec n=   
-\partial_s\left[ \mat{B}(\theta)\partial_s\vec{X}\right].
\end{align}
Therefore, the formulation \eqref{eqn:model_ani} can be reformulated into:
\begin{subequations}\label{eqn:model_ani2}
 \begin{align}
 \label{eqn:modela_ani2}
 &\partial_t\vec X = \partial_{ss}\mu\,\vec{n},\quad 0<s<L(t),\quad t>0,\\
&\mu\Vec{n}=-\partial_s\left[\mat{B}(\theta)\partial_s\Vec{X}\right].
 \label{eqn:modelb_ani2}
 \end{align}
\end{subequations}
For the system \eqref{eqn:model_ani2}, we can obtain the following energy-dissipative property: 
\begin{equation}\label{eqn:ener_disp}
    \frac{d}{dt}W(t) =\int_{\Gamma(t)}\left[\mat{B}(\theta)\partial_s\Vec{X}\right]\cdot\partial_s\partial_t\Vec{X}\,ds=-\int_{\Gamma(t)}(\partial_s\mu)^2\,ds\le0, \qquad \forall t\ge0.
\end{equation}
The proof of area-conservation property is similar to \eqref{eqn:is_stra} and will be not repeated here.

\subsection{The new formulation}
As the isotropic case, we introduce the following time-dependent auxiliary variables:
\begin{equation}\label{eqn:aid_1}
  R(t)=W(t) \quad \text{and} \quad\xi(t)=\frac{R(t)}{W(t)} \equiv 1,
\end{equation}
By the definition of $R(t)$, one has
\begin{equation}\label{eqn:aid_2}
    \frac{d}{dt}R(t)=-\xi(t)\int_{\Gamma(t)}(\partial_s\mu)^2\,ds.
\end{equation}
Therefore, the extended form of \eqref{eqn:model_ani} is denoted by:
\begin{subequations}\label{eqn:model_ani3}
   \begin{align}\label{eqn:model_ani3a}
      &\partial_t \Vec{X}\cdot \Vec{n}=\partial_{ss}\mu,\\&\label{eqn:model_ani3b}
      \mu\Vec{n}=-\partial_s\left[\mat{B}(\theta)\partial_s \Vec{X}\right],\\& \label{eqn:model_ani3c}
      \frac{d}{dt}R(t)=-\xi(t)\int_{\Gamma(t)} (\partial_s \mu)^2\,ds.
   \end{align}
\end{subequations}
Then, the weak formulation of \eqref{eqn:model_ani3} is stated as follows: Given the initial curve $\Gamma(0):=\Vec{X}(\mathbb{I},0)\in \mathbb{X}$, find the closed curve $\Gamma(t)=\Vec{X}(\mathbb{I},t)\in \mathbb{X}$ and the chemical potential $\mu(\cdot,t)\in H^1(\mathbb{I})$ for $t>0$, such that 
\begin{subequations}\label{eqn:ani_weak}
      \begin{align}
          &\label{eqn:ani_weaka}(\partial_t\Vec{X}\cdot \Vec{n},\varphi)_{\Gamma(t)}+(\partial_s \mu,\partial_s\varphi)_{\Gamma(t)}=0, \qquad  \forall \varphi \in H^1(\mathbb{I}),\\&\label{eqn:ani_weakb}
          (\mu\Vec{n},\Vec{\omega})_{\Gamma(t)}-\left(\mat{B}(\theta)\partial_s\Vec{X},\partial_s\Vec{\omega}
          \right)_{\Gamma(t)}=0,\qquad \forall \Vec{\omega} \in \mathbb{X}.
      \end{align}
  \end{subequations}
Furthermore, the spatial semi-discrete scheme of \eqref{eqn:ani_spaceweak}: Given the initial curve $\Gamma^h(0):=\Vec{X}^h(\cdot,0)\in[V^h(\mathbb{I}]^2$, find the closed curve $\Gamma^h(t):=\Vec{X}^h(\cdot,t)$ and the chemical potential $\mu^h(\cdot,t)\in V^h(\mathbb{I})$, such that 
\begin{subequations}\label{eqn:ani_spaceweak}
      \begin{align}
&\label{eqn:ani_spaceweaka}\left(\partial_t\Vec{X}^h\cdot \Vec{n}^h,\varphi^h\right)_{\Gamma^h(t)}^h+\left(\partial_s \mu^h,\partial_s\varphi^h\right)_{\Gamma^h(t)}^h=0, \qquad  \forall \varphi^h \in V^h(\mathbb{I}),\\&\label{eqn:ani_spaceweakb}
          \left(\mu^h\Vec{n}^h,\Vec{\omega}^h\right)_{\Gamma^h(t)}^h-\left(\mat{B}(\theta^h)\partial_s\Vec{X}^h,\partial_s\Vec{\omega}^h
          \right)_{\Gamma^h(t)}^h=0,\qquad \forall \Vec{\omega}^h \in [V^h(\mathbb{I})]^2.
      \end{align}
\end{subequations}

%%各向异性数值格式
\subsection{Numerical schemes}
We discretize the semi-discretization \eqref{eqn:ani_spaceweak} in time by using the methods in Section \ref{sec:isotropic}, and denote discrete energy $W^h(\Vec{X}^{m})$ of $\Gamma^m$ as
\begin{equation}\label{eqn:disc_ener}
    W^h(\Vec{X}^{m}):=\sum_{j=1}^{N}|\Vec{h}_j^m|\gamma(\theta^m_j),
\end{equation}
where $\theta_j^m$ is inclination angle of the curve $\Gamma^m$ on subinterval $\mathbb{I}_j$. We establish the following fully discrete numerical schemes.
 \begin{itemize}
\item {\textbf{BDF1-SAV for anisotropic SDF}}: Given $\Vec{X}^m\in V^h(\mathbb{I})$ and $R^m$, find $(\Vec{\overline{X}}^{m+1}, \overline{\mu}^{m+1})\in [V^h(\mathbb{I})]^2\times V^h(\mathbb{I})$, $\xi^{m+1}$, $R^{m+1}$ and $(\Vec{X}^{m+1}, \mu^{m+1})\in [V^h(\mathbb{I})]^2\times V^h(\mathbb{I})$, such that
\begin{subequations}\label{eqn:first_dis2}
\begin{align}\label{eqn:first_dis2a}
        &\left(\frac{\Vec{\overline{X}}^{m+1}-\Vec{X}^{m}}{\Delta t}\cdot \Vec{n}^{m},\varphi^h \right)_{\Gamma^m}^h + \left(\partial_s \overline{\mu}^{m+1},\partial_s\varphi^h \right)_{\Gamma^m}^h=0,\qquad 
        \forall \varphi^h\in V^h(\mathbb{I}),\\ \label{eqn:first_dis2b}
        &\left(\overline{\mu}^{m+1}\Vec{n}^m, \Vec{\omega}^h\right)_{\Gamma^m}^h-\left(\mat{B}(\theta^m)\partial_s\Vec{\overline{X}}^{m+1},\partial_s\Vec{\omega}^h\right)_{\Gamma^m}^h=0,\qquad 
        \forall \Vec{\omega}^h\in [V^h(\mathbb{I})]^2,\\&\label{eqn:first_dis2c}
        \frac{R^{m+1}-R^{m}}{\Delta t}=-\xi^{m+1}\left(\partial_s\overline{\mu}^{m+1},\partial_s\overline{\mu}^{m+1} \right)_{\overline{\Gamma}^{m+1}}^{h} \quad \text{with} \quad \xi^{m+1}=\frac{R^{m+1}}{W^h(\Vec{\overline{X}}^{m+1})},\\& \label{eqn:first_dis2d}
         \zeta^{m+1}=1-(1-\xi^{m+1})^r,\\&
         \label{eqn:first_dis2e}\Vec{X}^{m+1}=\zeta^{m+1}\Vec{\overline{X}}^{m+1}\quad \text{and} \quad \mu^{m+1} =\zeta^{m+1}\overline{\mu}^{m+1}.
    \end{align}
\end{subequations}
%%一阶保面积
\item {\textbf{BDF1-CSAV for anisotropic SDF}}: Given $\Vec{X}^m\in V^h(\mathbb{I})$ and $R^m$, find $(\Vec{\overline{X}}^{m+1}, \overline{\mu}^{m+1})\in [V^h(\mathbb{I})]^2\times V^h(\mathbb{I})$, $\xi^{m+1}$, $R^{m+1}$ and $(\Vec{X}^{m+1}, \mu^{m+1})\in [V^h(\mathbb{I})]^2\times V^h(\mathbb{I})$, such that
\begin{subequations}\label{eqn:first_acdis2}
    \begin{align}\label{eqn:first_acdis2a}
        &\left(\frac{\Vec{\overline{X}}^{m+1}-\Vec{X}^{m}}{\Delta t}\cdot \Vec{\overline{n}}^{m+\frac{1}{2}},\varphi^h \right)_{\Gamma^m}^h + \left(\partial_s \overline{\mu}^{m+1},\partial_s\varphi^h \right)_{\Gamma^m}^h=0,\qquad 
        \forall \varphi^h\in V^h(\mathbb{I}),\\ \label{eqn:first_acdis2b}
        &\left(\overline{\mu}^{m+1}\Vec{\overline{n}}^{m+\frac{1}{2}}, \Vec{\omega}^h\right)_{\Gamma^m}^h-\left(\mat{B}(\theta^m)\partial_s\Vec{\overline{X}}^{m+1},\partial_s\Vec{\omega}^h\right)_{\Gamma^m}^h=0,\qquad 
        \forall \Vec{\omega}^h\in [V^h(\mathbb{I})]^2,\\&
        \frac{R^{m+1}-R^{m}}{\Delta t}=-\xi^{m+1}\left(\partial_s\overline{\mu}^{m+1},\partial_s\overline{\mu}^{m+1} \right)_{\overline{\Gamma}^{m+1}}^{h} \quad \text{with} \quad \xi^{m+1}=\frac{R^{m+1}}{W^h(\Vec{\overline{X}}^{m+1})},\\&
         \zeta^{m+1}=1-(1-\xi^{m+1})^r,\\&
         \Vec{X}^{m+1}=\zeta^{m+1}\Vec{\overline{X}}^{m+1}\quad \text{and} \quad \mu^{m+1} =\zeta^{m+1}\overline{\mu}^{m+1}.
    \end{align}
\end{subequations}
%%二阶格式
\item {\textbf{BDF2-SAV for anisotropic SDF}}: 
    Given $\Vec{X}^{m}\in V^h(\mathbb{I})$, $\Vec{X}^{m-1}\in V^h(\mathbb{I})$, and $R^{m}$, find $(\Vec{\overline{X}}^{m+1}, \overline{\mu}^{m+1})\in [V^h(\mathbb{I})]^2\times V^h(\mathbb{I})$, $\xi^{m+1}$, $R^{m+1}$ and $(\Vec{X}^{m+1}, \mu^{m+1})\in [V^h(\mathbb{I})]^2\times V^h(\mathbb{I})$ , such that :
    \begin{subequations}\label{eqn:sec_diff2}
    \begin{align}\label{eqn:sec_diff2a}
        &\left ( \frac{\frac{3}{2}\Vec{\overline{X}}^{m+1}-2\Vec{X}^{m}+\frac{1}{2}\Vec{X}^{m-1}}{\Delta t}\cdot \Vec{\tilde{n}}^{m+1},\varphi^{h}
         \right )_{\tilde{\Gamma}^{m+1}}^{h} +\left ( \partial_s \overline{\mu}^{m+1},\partial_s \varphi^{h}\right )_{\tilde{\Gamma}^{m+1}}^{h}=0, \qquad 
        \forall \varphi^h\in V^h(\mathbb{I}),\\&\label{eqn:sec_diff2b}
         \left ( {\overline{\mu}^{m+1}}\Vec{\tilde{n}}^{m+1},\Vec{\omega}^h \right )_{\tilde{\Gamma}^{m+1}}^{h}-\left ( \mat{B}(\tilde{\theta}^{m+1})\partial_s{\Vec{\overline{X}}}^{m+1},\partial_s\Vec{\omega}^h\right )_{\tilde{\Gamma}^{m+1}}^{h}=0, \qquad \forall \Vec{\omega}^h\in [V^h(\mathbb{I})]^2, \\&\label{eqn:sec_diff2c}\frac{R^{m+1}-R^{m}}{\Delta t}=-\xi^{m+1}\left (\partial_s\overline{\mu}^{m+1},\partial_s\overline{\mu}^{m+1}\right )_{\overline{\Gamma}^{m+1}}^{h}\quad \text{with}\quad
         \xi^{m+1}=\frac{R^{m+1}}{W^h(\Vec{\overline{X}}^{m+1})},  
         \\& \label{eqn:sec_diff2d}
         \zeta^{m+1}=1-(1-\xi^{m+1})^r, \\&
         \label{eqn:sec_diff2e}\Vec{X}^{m+1}=\zeta^{m+1}\Vec{\overline{X}}^{m+1}\quad \text{and} \quad \mu^{m+1}=\zeta^{m+1}\overline{\mu}^{m+1}.
    \end{align}
\end{subequations}
\end{itemize}
%%各向异性数值格式
%%格式的特性
\subsection{The properties}
For the proposed numerical schemes, we demonstrate that the modified energy is unconditionally stable, and the original energy remains unconditionally bounded. 
In addition, it has been established through rigorous proof that the scheme \eqref{eqn:first_acdis2} maintains approximately area-conservation.
%%能量稳定性和有界性的证明
\begin{thm} \label{thm:ener}
For the schemes \eqref{eqn:first_dis2}, \eqref{eqn:first_acdis2} and \eqref{eqn:sec_diff2}, 
given $R^m\ge 0$, we can obtain $\xi^{m+1}\ge0$ and $R^{m+1}\ge 0$.  
Furthermore, there holds
    \begin{equation}\label{eqn:ener_S}
        R^{m+1}-R^{m}=-\Delta t \xi^{m+1}(\partial_s\overline{\mu}^{m+1},\partial_s\overline{\mu}^{m+1})_{\overline{\Gamma}^{m+1}}^{h}\le 0, 
    \end{equation}
i.e., the schemes are unconditionally energy-stable in the sense of a modified energy. 

Additionally, if $\zeta^{m+1}\ge 0$, there exists a bounded and positive constant $M_r$ that depends on $r$, such that 
\begin{equation}\label{eqn:num_bound}
    W^h(\Vec{X}^{m+1})\le M_r,\qquad \forall m\ge0.
\end{equation}
If $\zeta^{m+1}<0$, to make \eqref{eqn:num_bound} hold, it needs to satisfy that $\gamma(\theta)=\gamma(\theta+\pi)$.
\begin{proof}
    According to \eqref{eqn:first_dis2c}, $\xi^{m+1}$ can be expressed as
    \begin{equation}\label{eqn:xi}
        \xi^{m+1}= \frac{R^m}{W^h(\Vec{\overline{X}}^{m+1})+\Delta t(\partial_s\overline{\mu}^{m+1},\partial_s\overline{\mu}^{m+1})_{\overline{\Gamma}^{m+1}}^h}.
    \end{equation}
    Therefore, if $R^m\ge 0$, we can obtain $\xi^{m+1}\ge 0$. Since $R^{m+1}=\xi^{m+1}W^h(\Vec{\overline{X}}^{m+1})$, we can directly get $R^{m+1}\ge 0$. $\xi^{m+1}\ge 0$ also implies that \eqref{eqn:ener_S} holds.
    
    Let's denote $M:=R^0=W^h(\Vec{X}^0)$, then the energy stability implies that $R^m\le \cdots \le R^0=M$. Moreover, from \eqref{eqn:xi}, we have
    \begin{equation}\label{eqn:xi_bound}
        \xi^{m+1}\le\frac{M}{W^h(\Vec{\overline{X}}^{m+1})+\Delta t(\partial_s\overline{\mu}^{m+1},\partial_s\overline{\mu}^{m+1})_{\overline{\Gamma}^{m+1}}^h}.
    \end{equation}
     From \eqref{eqn:first_dis2d}, one has $\zeta^{m+1}=P_r(\xi^{m+1})\xi^{m+1}$, where $P_r$ is a polynomial of degree $r$. Then, according to \eqref{eqn:xi_bound}, there must exist a bound constant $M_r>0$ satisfying that  
    \begin{equation}\label{eqn:eta_bound}
        \big|\zeta^{m+1}\big|=\big|P_r(\xi^{m+1}) \xi^{m+1}\big| \le \frac{M_r}{W^h(\Vec{\overline{X}}^{m+1})+\Delta t(\partial_s \overline{\mu}^{m+1},\partial_s \overline{\mu}^{m+1})_{\overline{\Gamma}^{m+1}}^h}.
    \end{equation} 
From \eqref{eqn:first_dis2e}, we obtain
   \begin{equation}\label{eqn:relate1}
  \big|\Vec{h}_j^{m+1}\big|=\big|\Vec{X}^{m+1}(\rho_j)-\Vec{X}^{m+1}(\rho_{j-1})\big|=\big|\zeta^{m+1}\big| \big|\Vec{\overline{h}}_j^{m+1}\big|.
   \end{equation}
The unit normal vector $\Vec{n}^m$ can be computed by $\Vec{n}^m\big|_{\mathbb{I}_j}=\frac{(\Vec{h}_j^m)^{\bot}}{\big|\Vec{h}_j^m\big|}$ on each interval $\mathbb{I}_j$. Noticing \eqref{eqn:relate1}, we get
\begin{equation}\label{eqn:relate2}
\Vec{n}^{m+1}\big|_{\mathbb{I}_j}=\frac{\zeta^{m+1}(\Vec{\overline{h}}_j^{m+1})^{\bot}}{\big|\zeta^{m+1}\big|\big|\Vec{\overline{h}}_j^{m+1}\big|}=\frac{\zeta^{m+1}}{\big|\zeta^{m+1}\big|}\Vec{\overline{n}}^{m+1}\big|_{\mathbb{I}_j}.
\end{equation}
Therefore, based on the sign of $\zeta^{m+1}$, the relationship between $\Vec{{n}}^{m+1}$ and $\Vec{\overline{n}}^{m+1}$
has two cases
\begin{equation}\label{eqn:relate_n}
   \Vec{n}^{m+1}=
    \begin{cases}
        \Vec{\overline{n}}^{m+1},\qquad &\zeta^{m+1}\ge 0,\\ -\Vec{\overline{n}}^{m+1},\qquad &\zeta^{m+1} <0.
    \end{cases}
\end{equation}
Due to $\theta^{m+1}$ is the angle between $\Vec{n}^{m+1}$ and the $y$-axis, we have
\begin{equation}\label{eqn:relate_theta}
   \theta^{m+1}=
    \begin{cases}
        \overline{\theta}^{m+1},\qquad &\zeta^{m+1}\ge 0,\\ \overline{\theta}^{m+1}+\pi,\qquad &\zeta^{m+1} <0.
    \end{cases}
\end{equation}
Hence, when $\zeta^{m+1}\ge 0$, by definition of $W^h(\Vec{X}^m)$ , we can directly get 
\begin{equation}\label{eqn:ener_relate}
    W^h(\Vec{X}^{m+1})=|\zeta^{m+1}|W^h(\Vec{\overline{X}}^{m+1}).
\end{equation}
If $\zeta^{m+1}<0$, the validity of \eqref{eqn:ener_relate} needs to satisfy that $\gamma(\overline{\theta}^{m+1})=\gamma(\overline{\theta}^{m+1}+\pi)$. 
Substituting \eqref{eqn:eta_bound} into \eqref{eqn:ener_relate} gives that 
\begin{equation}\label{eqn:bound}
    W^h(\Vec{X}^{m+1})\le \frac{M_r}{W^h(\Vec{\overline{X}}^{m+1})+\Delta t\left(\partial_s \overline{X}^{m+1}, \partial_s \overline{X}^{m+1}\right)_{\overline{\Gamma}^{m+1}}^h} W^h(\Vec{\overline{X}}^{m+1})\le M_r.
\end{equation}
Therefore, the boundedness of the original energy has been proven.
\end{proof}
\end{thm}
%%时间高阶以及各向同性
\begin{rem}
    Theorem \ref{thm:propertiesofiso} is indeed a special case of Theorem \ref{thm:ener}. Otherwise, $\gamma(\theta)=\gamma(\theta+\pi)$ is just the condition for ensuring the symmetric positive definition of the matrix $\mat{B}(\theta)$. From this perspective, $\gamma(\theta)=\gamma(\theta+\pi)$ is a very mild condition.
\end{rem}
%%定理2 拟保面积性
We denote the total enclosed area $A^h(\Vec{X}^m)$ of $\Gamma^m$ as
\begin{equation}\label{eqn:disc_area}
    A^h(\Vec{X}^m):=\frac{1}{2}\sum_{j=1}^{N}{(x_j^m-x_{j-1}^{m})}{(y_{j-1}^{m}+y_{j}^m)}.
\end{equation}
We are able to prove the approximately area-conservation of \eqref{eqn:first_acdis2}.
\begin{thm}(\textbf{Approximately area-conservation}) \label{thm:Quasi-ac}
Let $\left(\Vec{X}^{m+1}(\cdot),\kappa^{m+1}(\cdot)\right)$ represent numerical solution of the scheme \eqref{eqn:first_acdis2}. Then, it holds the following property:
\begin{equation}\label{eqn:quasi_area}
        A^h\left(\Vec{X}^{m+1}\right)-A^h\left(\Vec{X}^m\right)= O(\Delta t^r),\qquad m\ge 0.
    \end{equation}
    \end{thm}
    \begin{proof}
       The approximate solution $\Gamma^h(\alpha)$ can be defined through a linear interpolation of $\Vec{\overline{X}}^{m+1}$ and $\Vec{X}^m$:
       \begin{equation}\label{eqn:interpolation}
           \Vec{X}^h(\rho,\alpha):=(1-\alpha)\Vec{X}^m(\rho)+\alpha \Vec{\overline{X}}^{m+1}(\rho),\qquad 0\le \rho \le 1,\text{ } 0 \le \alpha \le 1.
       \end{equation}
Then, by applying Theorem 2.1 in  \cite{Bao17} and setting $\varphi^h=\Delta t$ in \eqref{eqn:first_acdis2a}, we directly obtain 
   \begin{equation}
     \begin{aligned}
         A(1)-A(0)&= \int_{\mathbb{I}} [\Vec{\overline{X}}^{m+1}-\Vec{X}^m]\cdot \left[-\frac{1}{2}\left(\partial_\rho \Vec{X}^m+\partial_\rho \Vec{\overline{X}}^{m+1}\right)\right]^{\bot}\,d\rho\\ & 
         =\left((\Vec{\overline{X}}^{m+1}-\Vec{X})\cdot \Vec{\overline{n}}^{m+\frac{1}{2}},1\right)_{\Gamma^m}^h \equiv 0.
       \end{aligned}
   \end{equation} 
Therefore, we obtain $A(1)-A(0)=0$, where directly implies that $A^h(\Vec{\overline{X}}^{m+1})-A^h(\Vec{X}^m)=0$.
By the definition of \eqref{eqn:ener_fun}, we obtain
\begin{equation}
    A^h(\Vec{X}^{m+1}) - A^h(\Vec{\overline{X}}^{m+1})= (\zeta^2-1) A^h(\Vec{\overline{X}}^{m+1}) = (\zeta^2-1)A^h(\Vec{X}^{0}) =O(\Delta t^r).
\end{equation}
Hence, we have $A^h\left(\Vec{X}^{m+1}\right)-A^h\left(\Vec{X}^m\right)= O(\Delta t^r)$.
 \end{proof}    
%%定理2 拟保面积性 
%%格式的特性
%%chapter3 拓展到固态去湿问题中
\section{Extension to SSD} \label{sec:SSD}
 In this section, we extend the mentioned schemes to the SSD of thin films with anisotropic surface energy.
 %%固态去湿基础知识介绍
 The issue of SSD involves the study of the evolution of an open curve $\Gamma:=\Gamma(t)$ with two contact points $x_c^l:=x_c^l(t)$ and $x_c^r:=x_c^r(t)$ moving along the rigid flat substrate. We denote $\Gamma(t):=\Vec{X}(\rho,t)=(x(s,t),y(s,t))^T$ with $0\le s \le L(t).$
 When $0<s<L(t)$, the relationship \eqref{eqn:modela_ani2}-\eqref{eqn:modelb_ani2} still holds. Different from the system \eqref{eqn:is_model1}, the initial curve is given as
 \begin{equation}\label{eqn:SSD_initial}
     \Vec{X}(s,0):=\Vec{X}_0(s)=(x(s,0),y(s,0))^T=(x_0(s),y_0(s))^T, \quad, 0\le s \le L_0:=L(0),
 \end{equation}
satisfying $y_0(0)=y_0(L_0)=0$ and $x_0(0)<x_0(L_0)$. In addition, it also needs to satisfy the following boundary conditions: 

(i) contract line condition
\begin{equation}\label{eqn:boundary_1}
    y(0,t)=y(L,t)=0,\qquad t\ge0;
\end{equation}

(ii) relaxed contact angle conditions
\begin{equation}\label{eqn:boundary_2}
    \frac{dx_c^l(t)}{dt}=\eta f(\theta_d^l;\sigma),\qquad \frac{dx_c^r(t)}{dt}=-\eta f(\theta_d^r;\sigma),\qquad t\ge0; 
\end{equation}
where $f(\theta;\sigma)$ is defined as follows: 
\begin{equation}\label{eqn:SSD_f}
    f(\theta;\sigma)=\gamma(\theta)\cos(\theta)-{\gamma}'(\theta)sin(\theta)-\sigma, \qquad \theta \in [-\pi,\pi], \qquad \sigma=\frac{\gamma_{VS}-\gamma_{FS}}{\gamma_{FV}}.
\end{equation}
The definition of material constant $\sigma$ with $\gamma_{VS}$, $\gamma_{FS}$ and $\gamma_{FV}$ respectively represent surface energy densities of the vapor/substrate, film/substrate and film/vapor.

(iii) zero-mass flux condition
\begin{equation}\label{eqn:boundary_3}
    \partial_s \mu(0,t)=0, \qquad \partial_s \mu(L,t)=0, \qquad t\ge0.
\end{equation}
\begin{rem}
    In the boundary conditions mentioned above, condition (i) is to ensure that the two contact points always move along  the flat substrate, condition (ii) allows for the relaxation of the contact angle, and condition (iii) indicates that there is no-mass flux at the contact points, ensuring that the total area/mass of the thin film is conserved.
\end{rem}

Define the total free energy  of the system $W(t)$ as
\begin{equation} \label{eqn:ener_SSD}
    W(t)=\int_{\Gamma(t)} \gamma(\theta)\, ds -(x^r_c(t)-x^l_c(t))\sigma, \qquad t\ge0.
\end{equation}
We can directly prove that
\begin{equation}\label{eqn:SSD_ener_dec}
 \frac{d}{dt}A(t)=0, \qquad
    \frac{d}{dt} W(t) = -\int_{\Gamma(t)} (\partial_s \mu)^2 \,ds -\frac{1}{\eta}\left[(\frac{dx_c^l}{dt})^2+(\frac{dx_c^r}{dt})^2\right]\le 0, \qquad \forall t \ge 0,
\end{equation}
which indicates that the SSD also satisfies the two geometric properties: area conservation and energy dissipation. 
%%6.15 固态去湿的具体内容
\subsection{The new formulation}
Similarly, the introduced auxiliary variables are as follows:
\begin{equation}\label{eqn:SSD_aux}
    R(t)=W(t) \qquad \text{and} \qquad \xi(t)=\frac{R(t)}{W(t)}\equiv 1.
\end{equation}
Therefore, by the definition of $R(t)$, one has
\begin{equation} \label{eqn:SSD_diffR}
   \frac{d}{dt}R(t)= -\xi(t)\left(\int_{\Gamma(t)} (\partial_s \mu)^2\,ds+\frac{1}{\eta}\left[\left(\frac{dx_c^l}{dt}\right)^2+\left(\frac{dx_c^r}{dt}\right)^2\right]\right).
\end{equation}

Combined with \eqref{eqn:SSD_diffR}, we have the following new formulation of SSD:
\begin{subequations}
     \begin{align}\label{eqn:SSD_new}
      &\partial_t \Vec{X}\cdot \Vec{n}=\partial_{ss}\mu,\\&\label{eqn:SSD_newa}
      \mu\Vec{n}=-\partial_s\left[\mat{B}(\theta)\partial_s \Vec{X}\right],\\& \label{eqn:SSD_newb}
      \frac{d}{dt}R(t)=-\xi(t)\left(\int_{\Gamma(t)} (\partial_s \mu)^2\,ds+\frac{1}{\eta}\left[\left(\frac{dx_c^l}{dt}\right)^2+\left(\frac{dx_c^r}{dt}\right)^2\right]\right).
      \end{align}
\end{subequations}
Then, we derive the weak formulation of \eqref{eqn:SSD_new} with boundary conditions \eqref{eqn:boundary_1}-\eqref{eqn:boundary_3} and initial condition \eqref{eqn:SSD_initial} as follows: Given the initial curve $\Gamma(0)=\Vec{X}(\mathbb{I},0)\in \mathbb{X}$ with $\Vec{X}(\rho,0)=\Vec{X}_0(L_0 \rho)=\Vec{X}_0(s)$ and set $x_c^l(0)=x_0(s=0)<x_c^r(0)=x_0(s=L_0)$, find the evolution curve $\Gamma(t):=\Vec{X}(\cdot,t)\in \mathbb{X}$, and the chemical potential $\mu(\cdot,t)\in H^1(\mathbb{I})$ for $t>0$, such that
\begin{subequations}\label{SSD_weak}
    \begin{align}
        & \label{SSD_weaka}\left(\partial_t \Vec{X}\cdot \Vec{n},\varphi \right)_{\Gamma(t)}+\left(\partial_s \mu,\partial_s\varphi\right)_{\Gamma(t)}=0, \qquad \forall \varphi \in H^1(\mathbb{I}), \\
        & \label{SSD_weakb} \notag \left(\mu \Vec{n},\Vec{\omega} \right)_{\Gamma(t)}-\left(B(\theta)\partial_s \Vec{X},\partial_s \Vec{\omega} \right)_{\Gamma(t)}-\frac{1}{\eta}\left[\frac{dx_c^l(t)}{dt}\omega_1(0)+\frac{dx_c^r(t)}{dt}\omega_0(1)\right] \\&  +\sigma\left[\omega_1(1)-\omega_1(0) \right]=0, \qquad \forall \Vec{\omega}=(\omega_1,\omega_2)^T\in \mathbb{X}.
    \end{align}
\end{subequations}
With the above formulation, we further present the spatial semi-discrete scheme as: Given the initial curve $\Gamma^h(0)=\Vec{X}^h(\cdot,0)$ and set $x_c^l(0)=x_0(s=0)<x_c^r(0)=x_0(L_0)$, find the evolution curve $\Gamma(t)=\Vec{X}^h(\cdot,t)=\left(x^h(\cdot,t),y^h(\cdot,t)\right)^T \in [V^h(\mathbb{I})]^2$ and the chemical potential  $\mu^h(\cdot,t)\in V^h(\mathbb{I})$, such that
\begin{subequations}\label{SSD_spaceweak}
    \begin{align}
        & \label{SSD_spaceweaka}\left(\partial_t \Vec{X}^h\cdot \Vec{n}^h,\varphi^h \right)_{\Gamma^h(t)}^h+\left(\partial_s \mu^h,\partial_s\varphi^h\right)_{\Gamma^h(t)}^h=0, \qquad \forall \varphi^h \in V^h(\mathbb{I}), \\
        & \label{SSD_spaceweakb} \notag \left(\mu^h \Vec{n}^h,\Vec{\omega}^h \right)_{\Gamma^h(t)}^h-\left(B(\theta^h)\partial_s \Vec{X}^h,\partial_s \Vec{\omega}^h \right)_{\Gamma^h(t)}^h-\frac{1}{\eta}\left[\frac{dx_c^l(t)}{dt}\omega^h_1(0)+\frac{dx_c^r(t)}{dt}\omega^h_0(1)\right] \\&  +\sigma\left[\omega^h_1(1)-\omega^h_1(0) \right]=0, \qquad \forall \Vec{\omega}^h=(\omega^h_1,\omega^h_2)^T\in [V^h(\mathbb{I})]^2,
    \end{align}
\end{subequations}
where $x_c^l(t)=x^h(\rho_0=0,t) \le x_c^r(t)=x^h(\rho_N=1,t)$.
\subsection{Numerical schemes}
Define the discrete energy of $\Gamma^{m}$ as
\begin{equation}\label{eqn:SSD_discrete_ener}
   W^h(\Vec{X}^m):= \sum_{j=1}^{N}\big|\Vec{h}_j^m\big|  \gamma(\theta^m_j)-(x^{m}_r-x^{m}_l)\sigma.   
 \end{equation}
The mentioned fully-discrete numerical schemes can be extended to SSD, including 
\begin{itemize}
    \item \textbf{BDF1-SAV for SSD}: Given $\Vec{X}^m\in V^h(\mathbb{I})$ and $R^m$, find $(\Vec{\overline{X}}^{m+1}, \overline{\mu}^{m+1})\in [V^h(\mathbb{I})]^2\times V^h(\mathbb{I})$, $\xi^{m+1}$, $R^{m+1}$ and $(\Vec{X}^{m+1}, \mu^{m+1})\in [V^h(\mathbb{I})]^2\times V^h(\mathbb{I})$ such that
\begin{subequations}\label{eqn:first_dis3}
\begin{align}\label{eqn:first_dis3a}
        &\left(\frac{\Vec{\overline{X}}^{m+1}-\Vec{X}^{m}}{\Delta t}\cdot \Vec{n}^{m},\varphi^h \right)_{\Gamma^m}^h + \left(\partial_s \overline{\mu}^{m+1},\partial_s\varphi^h \right)_{\Gamma^m}^h=0,\qquad 
        \forall \varphi^h\in V^h(\mathbb{I}),\\ \label{eqn:first_dis3b} \notag
        &\left(\overline{\mu}^{m+1}\Vec{n}^m, \Vec{\omega}^h\right)_{\Gamma^m}^h-\left(\mat{B}\left(\theta^m\right)\partial_s\Vec{\overline{X}}^{m+1},\partial_s\Vec{\omega}^h\right)_{\Gamma^m}^h +\sigma[\omega_1^h(1)-\omega_1^h(0)]\\ &-\frac{1}{\eta}\left[\frac{\overline{x}_l^{m+1}-x_l^m}{\Delta t}\omega^h_1(0)+\frac{\overline{x}_r^{m+1}-x_r^m}{\Delta t}\omega^h_1(1)\right]=0 ,\qquad 
        \forall \Vec{\omega}^h\in [V^h(\mathbb{I})]^2,\\&\label{eqn:first_dis3c} \notag
        \frac{R^{m+1}-R^{m}}{\Delta t}=-\xi^{m+1}\left[\left(\partial_s\overline{\mu}^{m+1},\partial_s\overline{\mu}^{m+1} \right)_{\overline{\Gamma}^{m+1}}^{h}+\frac{1}{\eta}\left(\left(\frac{\overline{x}_l^{m+1}-x_l^{m}}{\Delta t }\right)^2+\left(\frac{\overline{x}_r^{m+1}-x_r^{m}}{\Delta t }\right)^2\right)\right] \\ & \quad \text{with} \quad \xi^{m+1}=\frac{R^{m+1}}{W^h(\Vec{\overline{X}}^{m+1})},\\& \label{eqn:first_dis3d}
         \zeta^{m+1}=1-(1-\xi^{m+1})^r,\\&
         \label{eqn:first_dis3e}\Vec{X}^{m+1}=\zeta^{m+1}\Vec{\overline{X}}^{m+1}\quad \text{and} \quad \mu^{m+1} =\zeta^{m+1}\overline{\mu}^{m+1}.
    \end{align}
\end{subequations}
\item \textbf{BDF1-CSAV for SSD}: Given $\Vec{X}^m\in V^h(\mathbb{I})$ and $R^m$, find $(\Vec{\overline{X}}^{m+1}, \overline{\mu}^{m+1})\in [V^h(\mathbb{I})]^2\times V^h(\mathbb{I})$, $\xi^{m+1}$, $R^{m+1}$ and $(\Vec{X}^{m+1}, \mu^{m+1})\in [V^h(\mathbb{I})]^2\times V^h(\mathbb{I})$ such that
\begin{subequations}\label{eqn:first_acdis3}
\begin{align}\label{eqn:first_acdis3a}
        &\left(\frac{\Vec{\overline{X}}^{m+1}-\Vec{X}^{m}}{\Delta t}\cdot \Vec{\overline{n}}^{m+\frac{1}{2}},\varphi^h \right)_{\Gamma^m}^h + \left(\partial_s \overline{\mu}^{m+1},\partial_s\varphi^h \right)_{\Gamma^m}^h=0,\qquad 
        \forall \varphi^h\in V^h(\mathbb{I}),\\ \label{eqn:first_acdis3b} \notag
        &\left(\overline{\mu}^{m+1}\Vec{\overline{n}}^{m+\frac{1}{2}}, \Vec{\omega}^h\right)_{\Gamma^m}^h-\left(\mat{B}\left(\theta^m\right)\partial_s\Vec{\overline{X}}^{m+1},\partial_s\Vec{\omega}^h\right)_{\Gamma^m}^h +\sigma[\omega_1^h(1)-\omega_1^h(0)]\\ &-\frac{1}{\eta}\left[\frac{\overline{x}_l^{m+1}-x_l^m}{\Delta t}\omega^h_1(0)+\frac{\overline{x}_r^{m+1}-x_r^m}{\Delta t}\omega^h_1(1)\right]=0 ,\qquad 
        \forall \Vec{\omega}^h\in [V^h(\mathbb{I})]^2,\\&\label{eqn:first_acdis3c} \notag
        \frac{R^{m+1}-R^{m}}{\Delta t}=-\xi^{m+1}\left[\left(\partial_s\overline{\mu}^{m+1},\partial_s\overline{\mu}^{m+1} \right)_{\overline{\Gamma}^{m+1}}^{h}+\frac{1}{\eta}\left(\left(\frac{\overline{x}_l^{m+1}-x_l^{m}}{\Delta t }\right)^2+\left(\frac{\overline{x}_r^{m+1}-x_r^{m}}{\Delta t }\right)^2\right)\right] \\ & \quad \text{with} \quad \xi^{m+1}=\frac{R^{m+1}}{W^h(\Vec{\overline{X}}^{m+1})},\\& \label{eqn:first_acdis3d}
         \zeta^{m+1}=1-(1-\xi^{m+1})^r,\\&
         \label{eqn:first_acdis3e}\Vec{X}^{m+1}=\zeta^{m+1}\Vec{\overline{X}}^{m+1}\quad \text{and} \quad \mu^{m+1} =\zeta^{m+1}\overline{\mu}^{m+1}.
    \end{align}
\end{subequations}
\item {\textbf{BDF2-SAV for anisotropic SDF}}: 
    Given $\Vec{X}^{m}\in V^h(\mathbb{I})$, $\Vec{X}^{m-1}\in V^h(\mathbb{I})$, and $R^{m}$, find $(\Vec{\overline{X}}^{m+1}, \overline{\mu}^{m+1})\in [V^h(\mathbb{I})]^2\times V^h(\mathbb{I})$, $\xi^{m+1}$, $R^{m+1}$ and $(\Vec{X}^{m+1}, \mu^{m+1})\in [V^h(\mathbb{I})]^2\times V^h(\mathbb{I})$ such that :
    \begin{subequations}\label{eqn:sec_diff3}
    \begin{align}\label{eqn:sec_diff3a}
        &\left ( \frac{\frac{3}{2}\Vec{\overline{X}}^{m+1}-2\Vec{X}^{m}+\frac{1}{2}\Vec{X}^{m-1}}{\Delta t}\cdot \Vec{\tilde{n}}^{m+1},\varphi^{h}
         \right )_{\tilde{\Gamma}^{m+1}}^{h} +\left ( \partial_s \overline{\mu}^{m+1},\partial_s \varphi^{h}\right )_{\tilde{\Gamma}^{m+1}}^{h}=0 ,\qquad 
        \forall \varphi^h\in V^h(\mathbb{I}),\\ & \label{eqn:sec_diff3b}
        \notag
        \left(\overline{\mu}^{m+1}\Vec{\tilde{n}}^{m+1}, \Vec{\omega}^h\right)_{\tilde{\Gamma}^{m+1}}^h-\left(\mat{B}\left(\tilde{\theta}^{m+1}\right)\partial_s\Vec{\overline{X}}^{m+1},\partial_s\Vec{\omega}^h\right)_{\tilde{\Gamma}^{m+1}}^h +\sigma[\omega_1^h(1)-\omega_1^h(0)]\\ &-\frac{1}{\eta}\left[\frac{\frac{3}{2}\overline{x}_l^{m+1}-2x_l^m+\frac{1}{2}x_1^{m-1}}{\Delta t}\omega^h_1(0)+\frac{\frac{3}{2}\overline{x}_r^{m+1}-2x_r^m+\frac{1}{2}x_r^{m-1}}{\Delta t}\omega^h_1(1)\right]=0 ,\qquad 
        \forall \Vec{\omega}^h\in [V^h(\mathbb{I})]^2,\\&\label{eqn:sec_diff3c}\notag
        \frac{R^{m+1}-R^{m}}{\Delta t}=-\xi^{m+1}\left[\left(\partial_s\overline{\mu}^{m+1},\partial_s\overline{\mu}^{m+1} \right)_{\overline{\Gamma}^{m+1}}^{h}+\frac{1}{\eta}\left(\left(\frac{\overline{x}_l^{m+1}-x_l^{m}}{\Delta t }\right)^2+\left(\frac{\overline{x}_r^{m+1}-x_r^{m}}{\Delta t }\right)^2\right)\right] \\ & \quad \text{with} \quad \xi^{m+1}=\frac{R^{m+1}}{W^h(\Vec{\overline{X}}^{m+1})}, 
         \\& \label{eqn:sec_diff3d}
         \zeta^{m+1}=1-(1-\xi^{m+1})^r, \\&
         \label{eqn:sec_diff3e}\Vec{X}^{m+1}=\zeta^{m+1}\Vec{\overline{X}}^{m+1}\quad \text{and} \quad \mu^{m+1}=\zeta^{m+1}\overline{\mu}^{m+1}.
    \end{align}
\end{subequations}
\end{itemize}

     \begin{rem}
We can directly extend the above methods to the BDFk schemes ($k\geq 3$); however, to avoid excessive length of the article, we will not elaborate on it in this paper. Additionally, we can develop the corresponding variable-time-step methods, and the results given in Theorem \ref{thm:propertiesofiso} can also be directly  proved. 
Moreover, the energy-stable methods proposed in this work have high applicability and can also be directly used in other geometric flows in both two-dimensional and three-dimensional spaces.
\end{rem}

%%固态去湿基础知识介绍
%%chapter3 拓展到固态去湿问题中

%%chapter4 数值实验
\section{Numerical results}\label{sec:experiment}
In this section, we conduct a series of experiments to demonstrate the superiority of our proposed schemes. 
%Initially, we utilize manifold distance to measure the consistency of temporal convergence accuracy of BDFk(k=1,2) across the newly introduced schemes. Subsequently, we illustrate the temporal evolution of relative area/mass loss, modified energy, original energy, and mesh ratio under varying surface energy densities for these schemes. Following this, based on the findings of \ref{thm:Quasi-ac}, we compare the area loss of scheme BDF1-CSAV at different values of $r$. Simultaneously, we analyze the difference between modified and original energy at different time steps. Finally, using a rectangle as the initial curve, we illustrate the temporal evolution of the curve under different schemes.
In the experimental process of the SDF, we adopt the ellipse defined by $\frac{x^2}{4}+y^2=1$ as the initial closed curve. For SSD, we select the upper half of the ellipse, where $y \ge 0$, as the initial open curve. Except for ellipse, we also use rectangle as initial curve in the study of morphological evolution. Additionally, we consistently choose the contact line mobility $\eta = 100$.

\subsection{Isotropic/anisotropic SDF}

\textbf{Example 1} (Convergence tests) 
Let $\Omega_1$ and $\Omega_2$ denote the inner regions enclosed by $\Gamma_1$ and $\Gamma_2$, respectively, then the manifold distance between the two closed curves is defined as  \cite{Zhao20}:
\begin{equation}\label{eqn:manifold}
    \text{M}(\Gamma_1,\Gamma_2):= \left|(\Omega_1\setminus \Omega_2)\cup(\Omega_2 \setminus \Omega_1)\right| = \left|\Omega_1\right|+\left|\Omega_2\right|-2\left|\Omega_1\cup\Omega_2\right|,
\end{equation}
where $\left|\Omega\right|$ represents the area of $\Omega$.
For the purpose of testing temporal errors, we fix the number of spatial divisions $N$ to be sufficiently large so that spatial error can be neglected compared with the temporal discrete error. Subsequently, we take different time step and then the errors can be computed as follows:
\begin{equation}\label{eqn:error_order}
    e_\tau(T):=\text{M}(\Vec{X}_{h,T/\tau},\Vec{X}_{h,T/\frac{\tau}{2}}).
\end{equation}

 We plot the errors of the BDF1-SAV and BDF2-SAV methods in Figure \ref{fig:order_SDF}. It can be observed that under varying surface energy densities, the convergence orders of the two schemes are consistent with our desired results. During the tests, we set \( r = 2 \) for the BDF1-SAV method and \( r = 3 \) for the BDF2-SAV method, as noted in Remark \ref{rem:order_r}. 
%% SDF收敛性测试%%
\begin{figure}[!htp]
         \centering
    \includegraphics[width=0.45\textwidth]{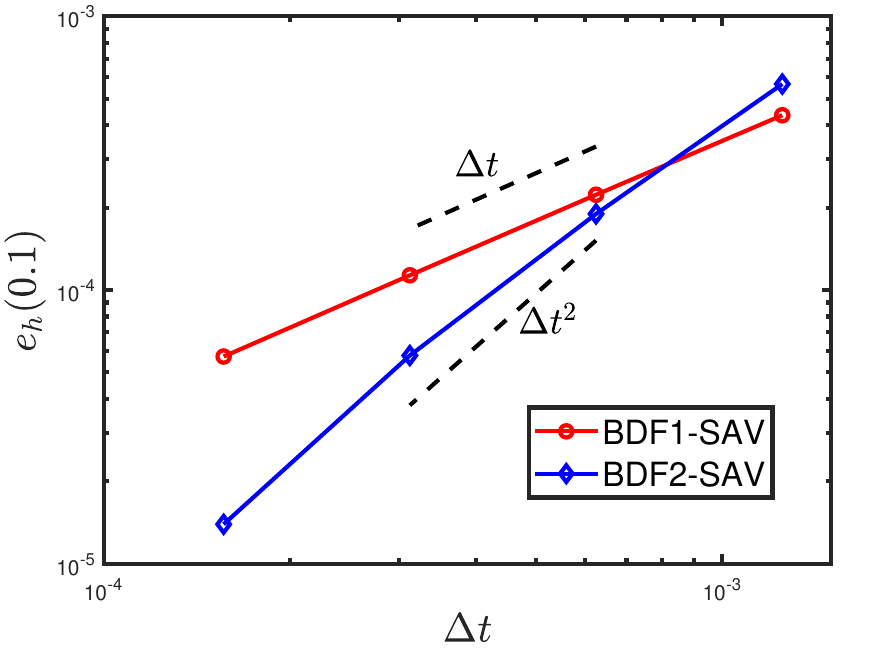}
    \includegraphics[width=0.45\textwidth]{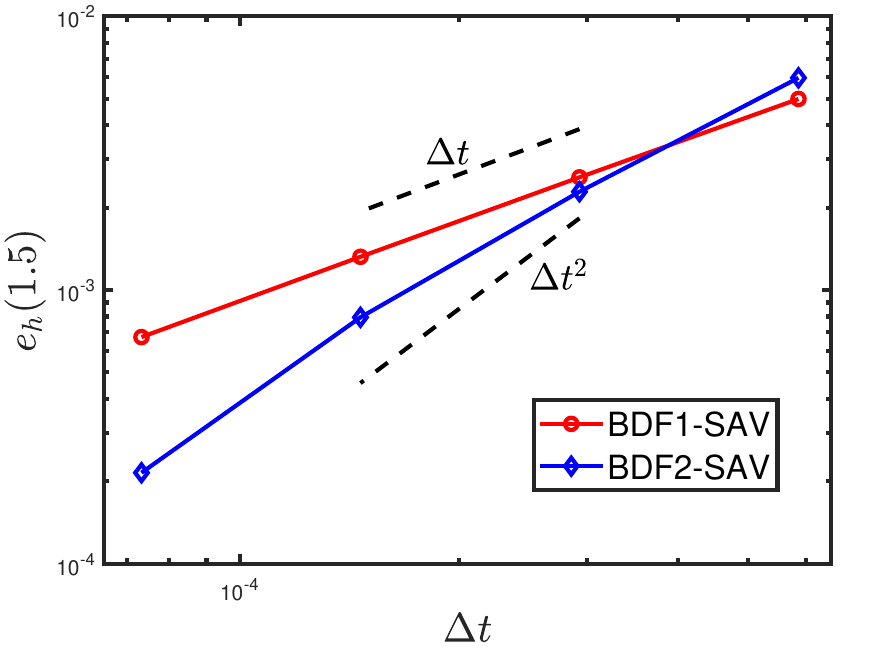}
    \includegraphics[width=0.45\textwidth]{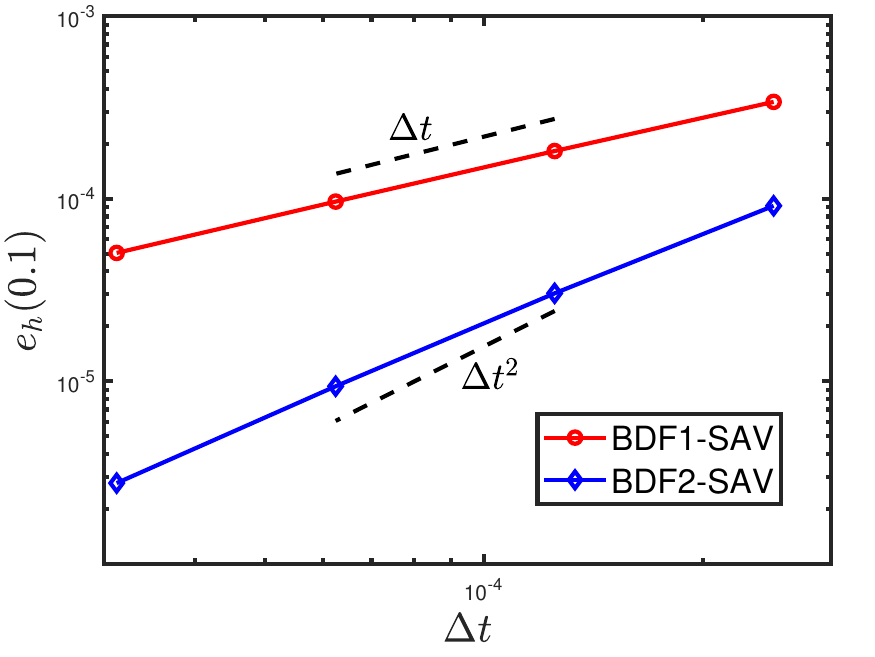}
    \includegraphics[width=0.45\textwidth]{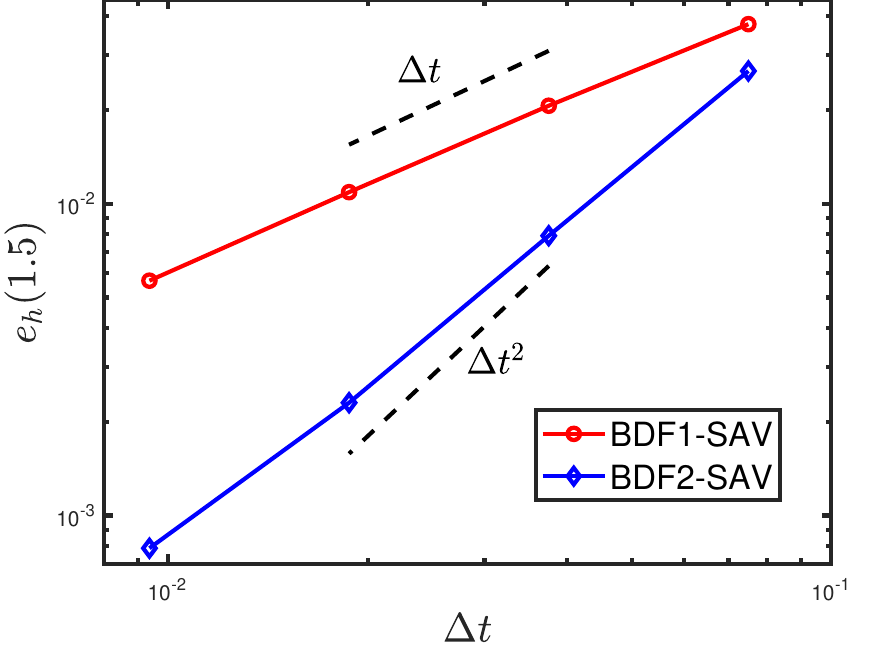}
    \caption{Convergence rates of BDF1-SAV and BDF2-SAV at times $T = 0.1$, $1.5$ with different surface energy densities: $\gamma(\theta)\equiv 1$ and $\gamma(\theta)=1+0.05\cos(4\theta)$.}
    \label{fig:order_SDF}
    \end{figure}

%We specifically test the temporal convergence rates of the BDF1-SAV and BDF2-SAV schemes by setting $r$ to 2 and 3, respectively. According to Remark \ref{rem:order_r}, this ensures that these schemes preservers the convergence orders of BDFk ($k=1,2$). 
%% SDF收敛性测试%%

%% example2 SDF面积损耗和能量耗散测试 %%
\textbf{Example 2} (Area/mass loss \& Energy stability \& Mesh quality) In this example, to measure the area loss during the evolution, we compute relative area loss $\Delta A^h(t)$ as follows:
\begin{equation}\label{area_loss_SDF}
    \Delta A^h(t)\big|_{t=t_m}:=\frac{A^h(\Vec{X}^m)-A^h(\Vec{X}^0)}{A^h(\Vec{X}^0)},
\end{equation}
where $A^h(\Vec{X}^m)$ is defined in \eqref{eqn:disc_area}. 
The modified energy $R(t)$ and the original energy $L^h(t)$ or $W^h(t)$ are defined as:
\begin{equation} \label{eqn:SDF_ener}
    R(t)\big|_{t=t_m}=R^m,\qquad L^h(t)|_{t=t_m} = L^h(\Vec{X}^m), \qquad W^h(t)|_{t=t_m}=W^h(\Vec{X}^m).
\end{equation}
The mesh ratio indicator $\Psi(t)$ is defined by:
\begin{equation}\label{eqn:mesh_indicator}
\Psi(t)\big|_{t=t_m}=\Psi^m:=\frac{\max_{1\le j \le N}|\Vec{h}_j^m|}{\min_{1\le j \le N}|\Vec{h}_j^m|}.
\end{equation}
We further define the difference in terms of modified and original anisotropic surface energy, given by 
\begin{equation}\label{eqn:ener_dif}
\Delta W^h(t)\bigg|_{t=t_m} = \big|R^m - W^h(\Vec{X}^m)\big|.
\end{equation}
In this example, we mainly do the following numerical tests:
\begin{itemize}
    \item In Figure \ref{fig:SDF_area_scheme}, we test the area loss for the three schemes under different surface energy densities, with the parameter $r$ set to $3$ and $6$, respectively.
The curves of the area loss for the three schemes are relatively unstable when $r = 3$, while it converges to a stable level when $r = 6$. In the case of $r = 6$, BDF1-CSAV exhibits a smaller area loss, approaching area conservation, compared to the other two schemes. 
\begin{figure}[!htp]
         \centering
    \includegraphics[width=0.45\textwidth]{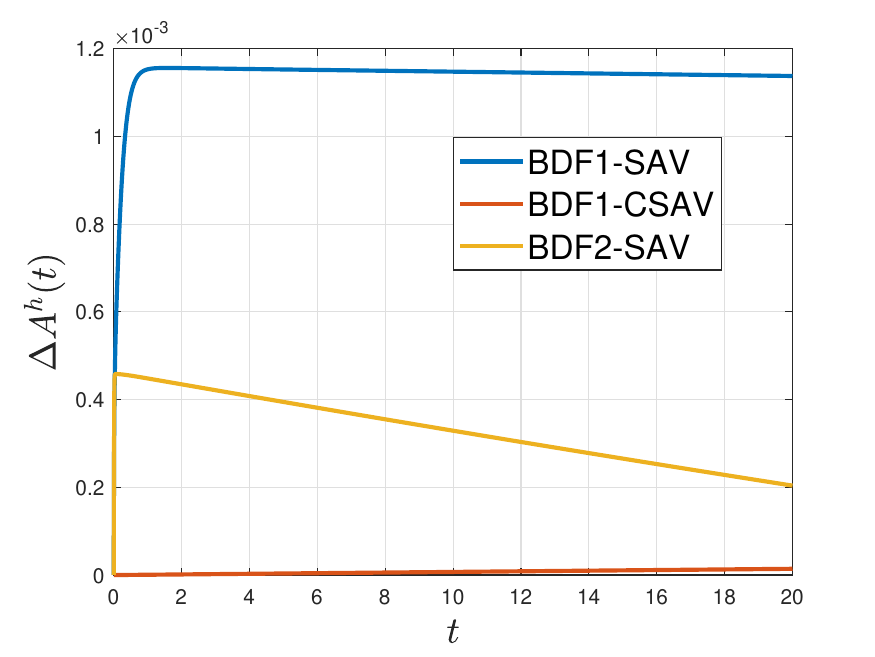}\quad
    \includegraphics[width=0.45\textwidth]{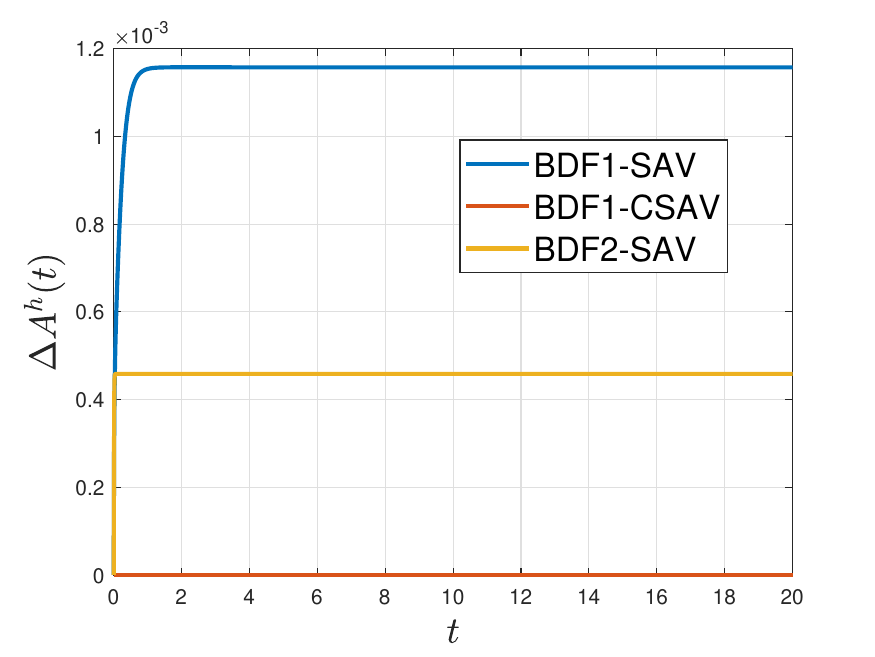}
    \includegraphics[width=0.45\textwidth]{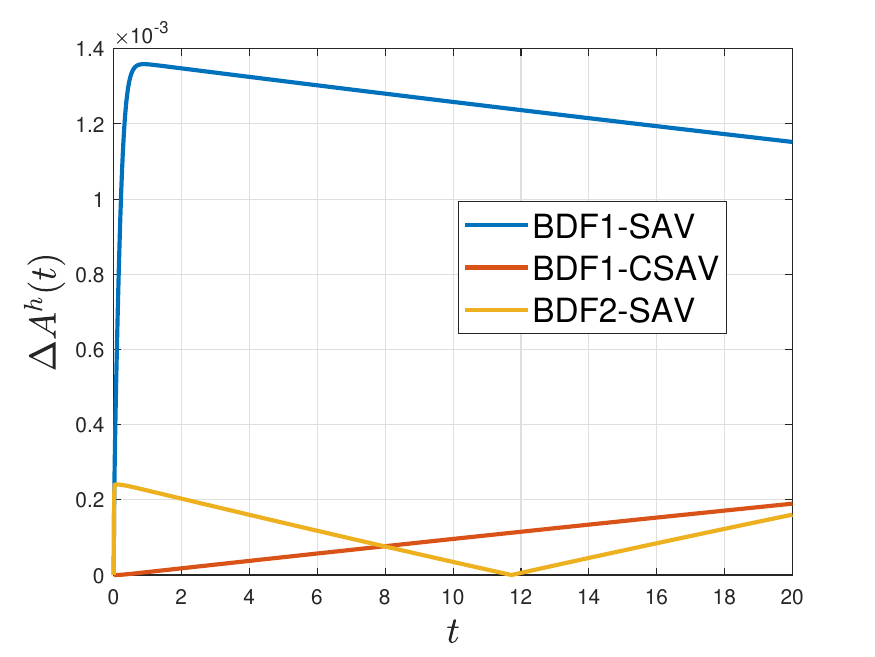}\quad
    \includegraphics[width=0.45\textwidth]{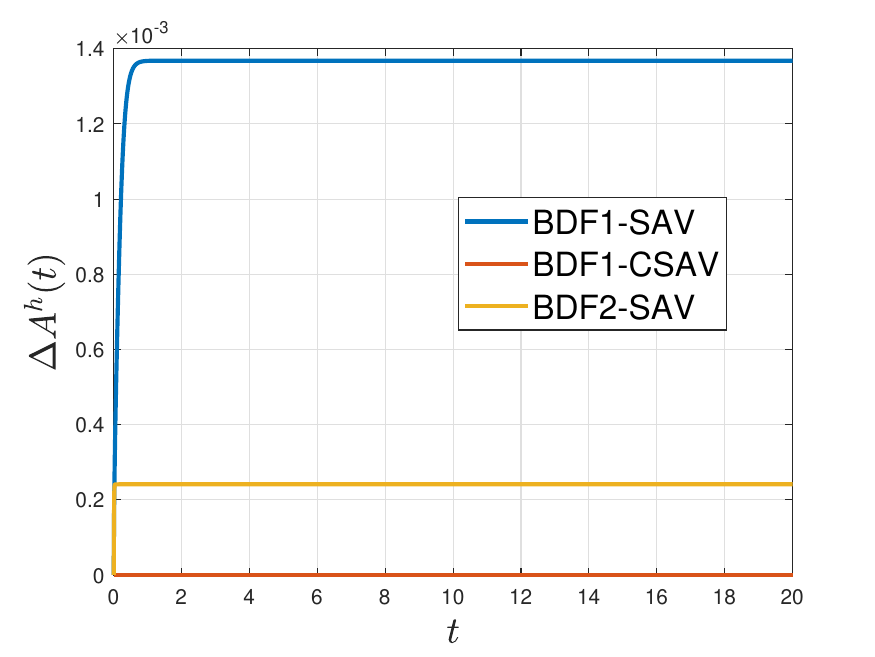}
    \includegraphics[width=0.45\textwidth]{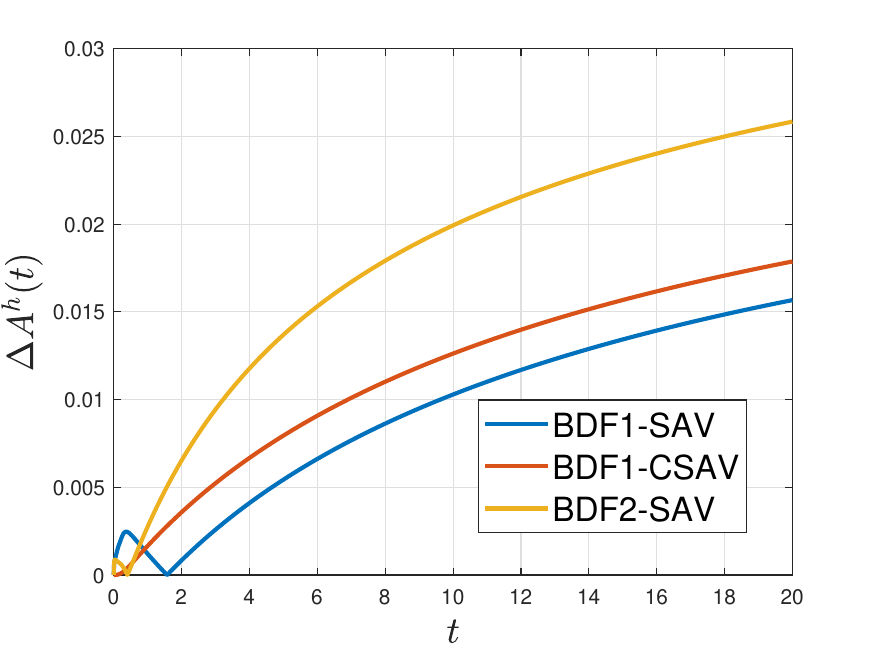}\quad
    \includegraphics[width=0.45\textwidth]{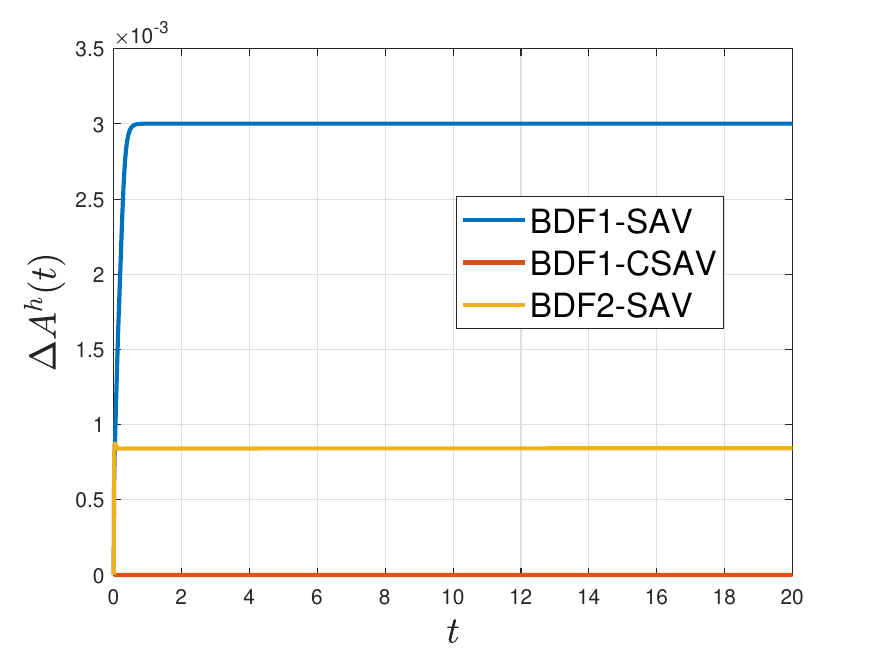}
    \caption{The relative area loss $\Delta A^h(t)$ for the three schemes at $r = 3$, 6 under surface energy $\gamma(\theta) = 1+ \beta \cos(4\theta)$: $\beta =0,\frac{1}{20},\frac{1}{10}$. Parameters are chosen as $N = 80$, $\Delta t = \frac{1}{160}$.}
    \label{fig:SDF_area_scheme}
\end{figure}
\item We further check the energy stability law shown in Theorem \ref{thm:ener}. 
In Figure \ref{fig:ener_SDF1}, we present the evolution of the modified and original energies over time for three schemes under varying surface energy densities. 
We observe that although the energy stability in theory is a type of modified energy, our numerical experiments demonstrate that both the original and modified forms of energy are dissipative over time, and their values are very close to each other.
Furthermore, in combination with the observation from Figure \ref{fig:SDF_area_scheme}, the scheme with better area conservation tends to exhibit less decline in the original energy. 
\begin{figure}[!htp]
         \centering
    \includegraphics[width=0.45\textwidth]{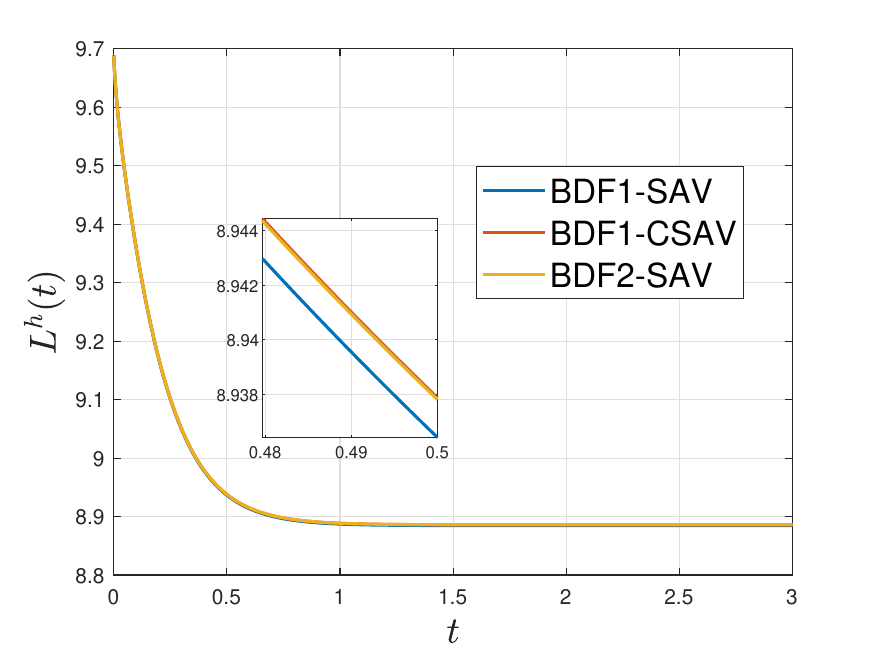}
    \includegraphics[width=0.45\textwidth]{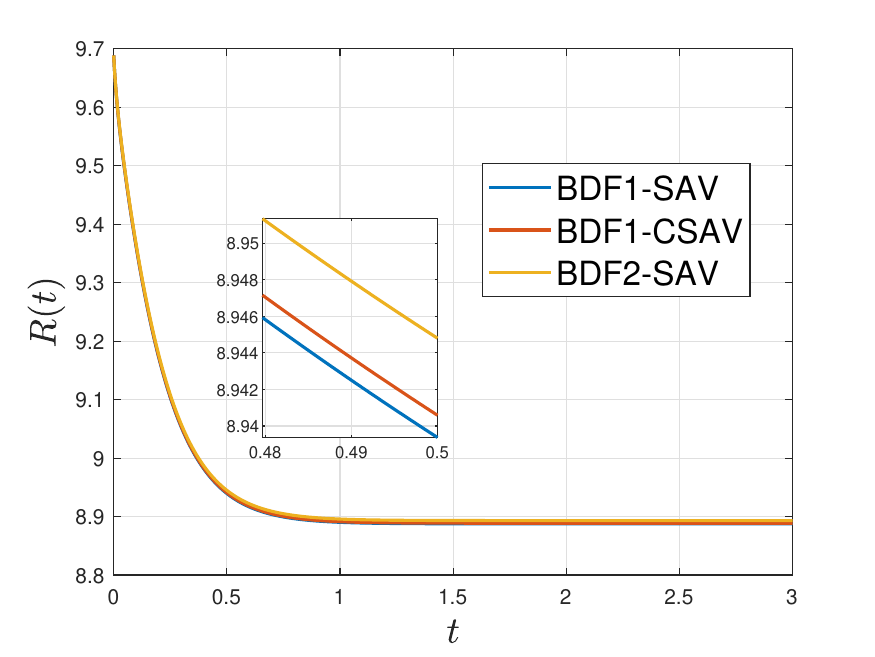}
    \includegraphics[width=0.45\textwidth]{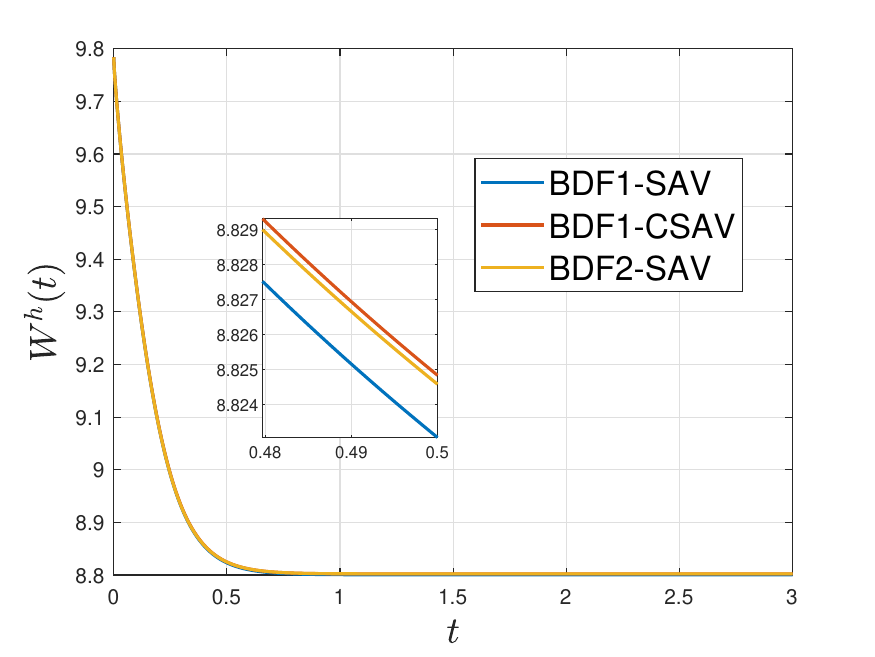}
    \includegraphics[width=0.45\textwidth]{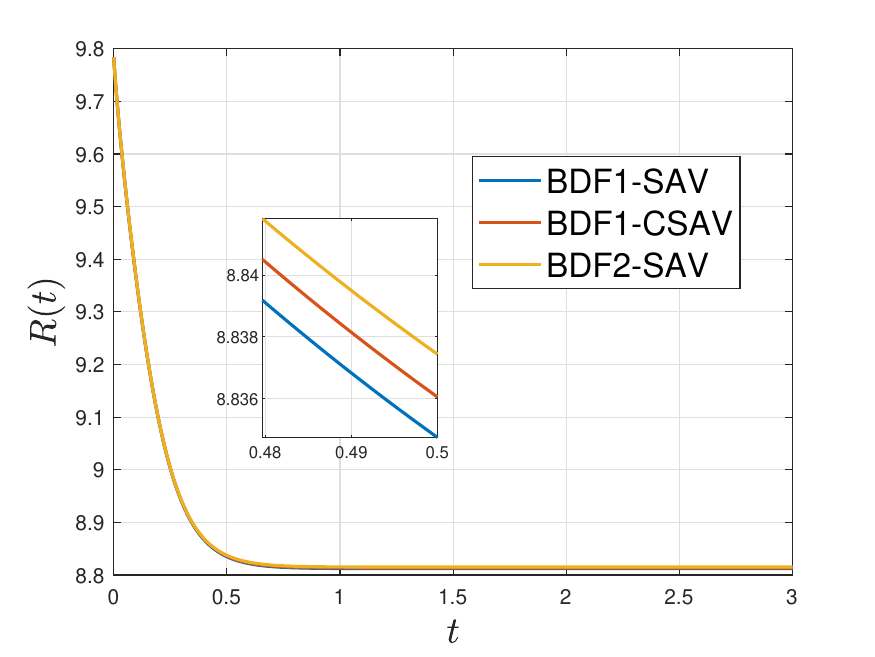}
    \includegraphics[width=0.45\textwidth]{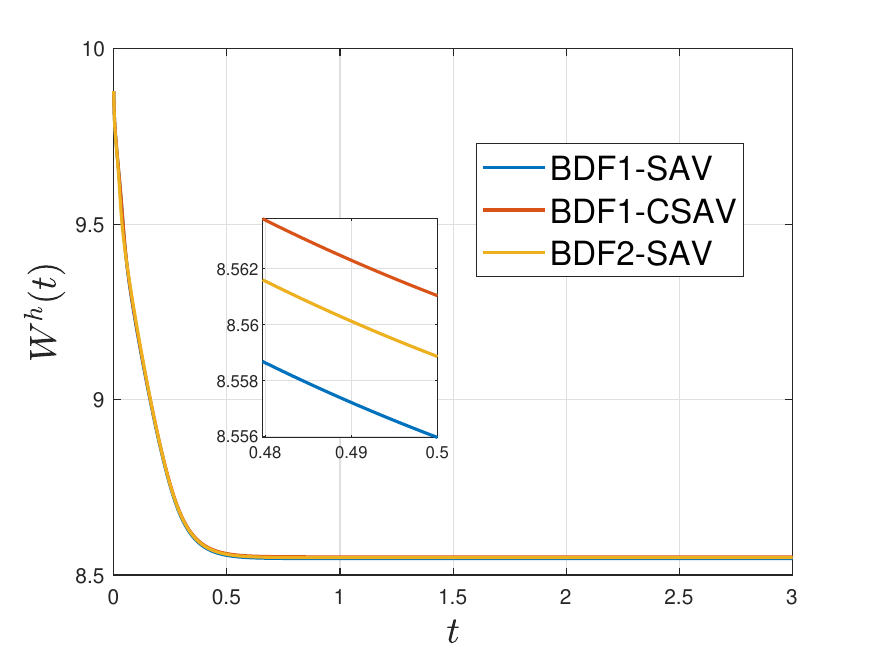}
    \includegraphics[width=0.45\textwidth]{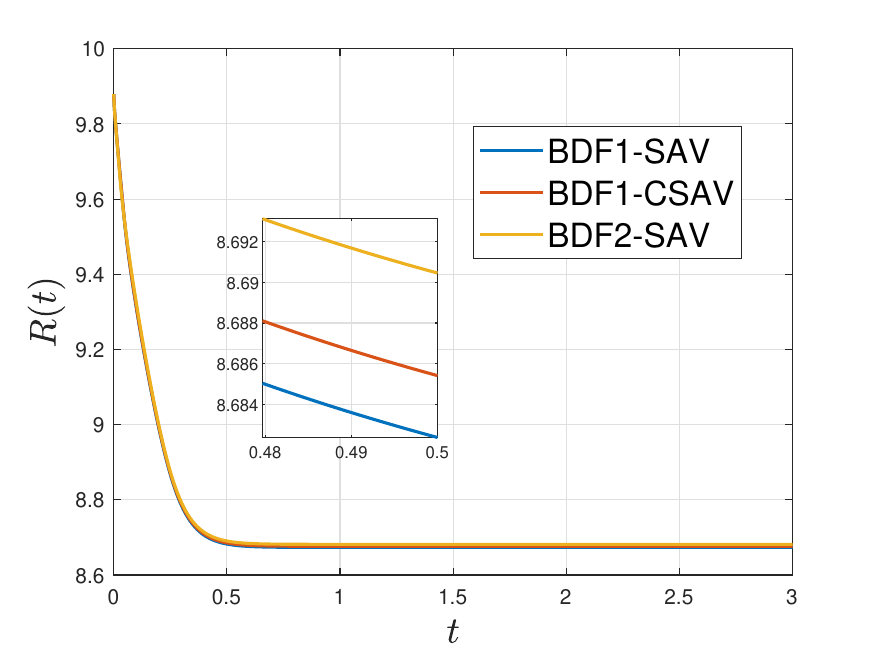}
    \caption{The original energy and modified energy for the three schemes under surface energy $\gamma(\theta) = 1+ \beta \cos(4\theta)$: $\beta =0,\frac{1}{20},\frac{1}{10}$. Parameters are chosen as $N = 640$, $\Delta t = \frac{1}{640}$, $r = 6$.}
    \label{fig:ener_SDF1}
\end{figure}
\item 

We demonstrate through numerical experiments that the constructed schemes maintain excellent mesh quality throughout the evolution process.
Figure \ref{fig:mesh_SDF} shows the temporal evolution of the mesh ratio for the three schemes under different surface energy densities. We observe that for flows with isotropic surface energy, the mesh ratio of all three schemes converges to $1$ over time, while for flows with anisotropic surface energy, the mesh ratio initially decreases with time and eventually converges to a  constant $C\approx 1.87$. 
\end{itemize}

      \begin{figure}[!htp]
         \centering
    \includegraphics[width=0.45\textwidth]{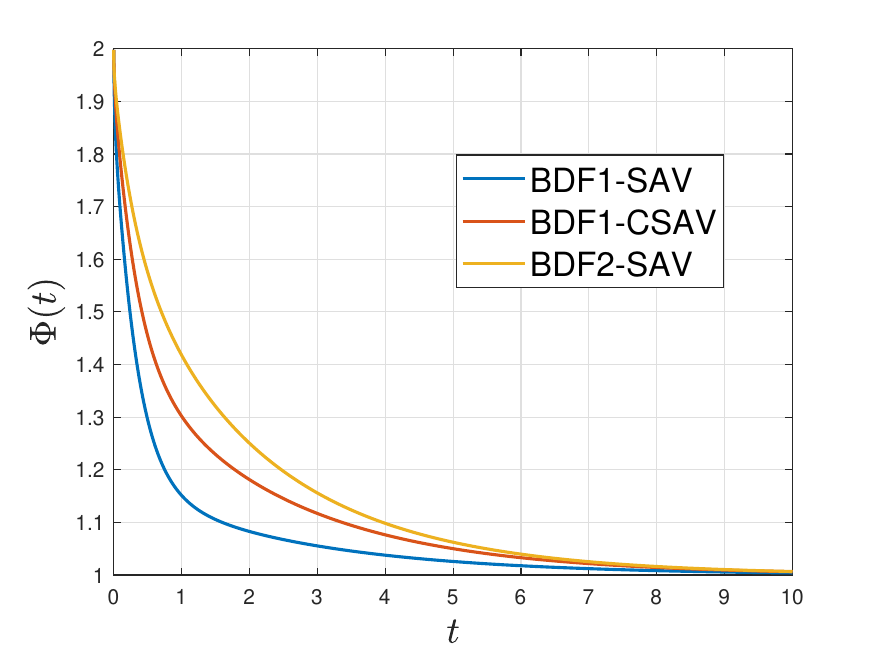}
    \includegraphics[width=0.45\textwidth]{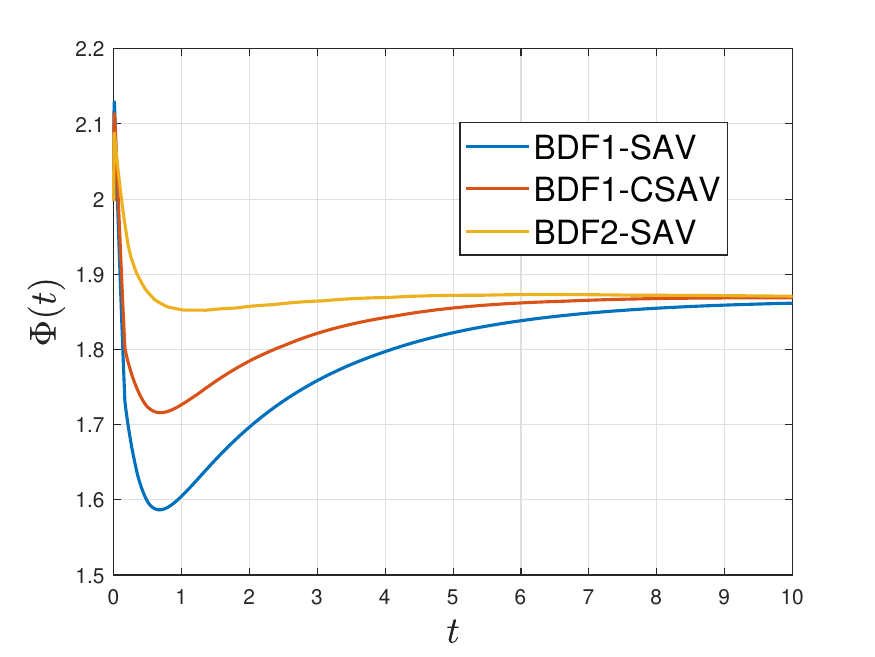}
    
    \caption{Temporal evolution of the mesh ratio $\Psi(t)$ for the three schemes with different surface energies $\gamma(\theta)=1+\beta \cos(4\theta)$: $\beta = 0, \frac{1}{20}.$ Parameters are chosen as $N=2^{-7}$, $\Delta t = 10^{-3}$, $r = 3$.}
    \label{fig:mesh_SDF}
    \end{figure}
      % 网格比%
%% example2 SDF面积损耗和能量耗散测试 %%

%% example3 面积关于r值得变化以及能量关于时间步长 %%
\textbf{Example 3} (Factors related to area and energy) For the BDF1-CSAV, based on Theorem \ref{thm:Quasi-ac}, we test area loss under different values of $r$. Moreover, since $\xi^{m+1} = 1 + O(\Delta t)$, we observe that the difference between the original energy and the modified energy decreases as the time step size decreases. 
Figure \ref{fig:CSAV_area_SDF} indicates that for the BDF1-CSAV, different values of $r$ result in varying area losses. When the time step $\Delta t = \frac{1}{160}$, the area preservation of the scheme improves as the value of r increases. Figure \ref{fig:ener_dis_SDF} demonstrates that for anisotropic surface energy, $\Delta W^h(t)$ in all three schemes decreases as the time step decreases, indicating a closer approximation between the original energy and the modified energy.

    %area 关于r的变化%

\begin{figure}[!htp]
         \centering
    \includegraphics[width=0.33\textwidth]{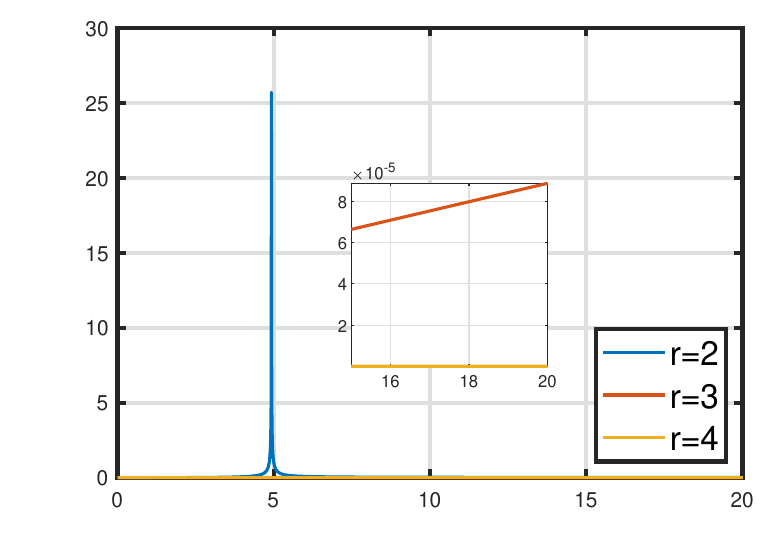}
    \includegraphics[width=0.33\textwidth]{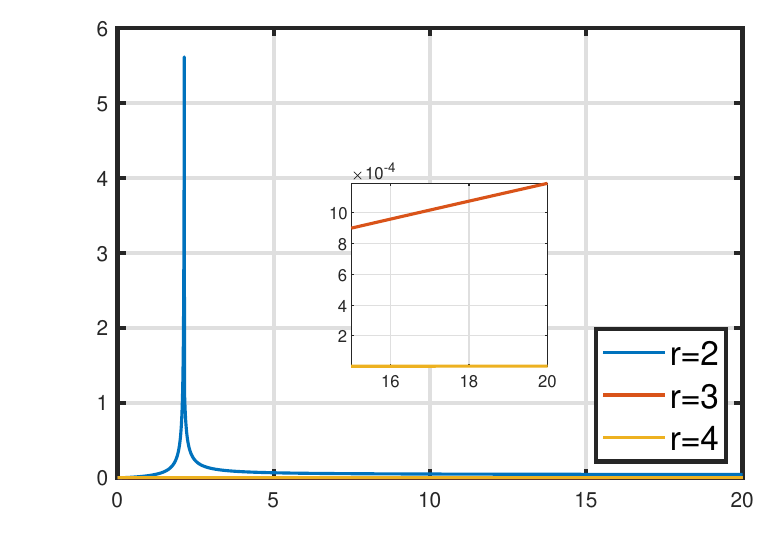}
    \includegraphics[width=0.33\textwidth]{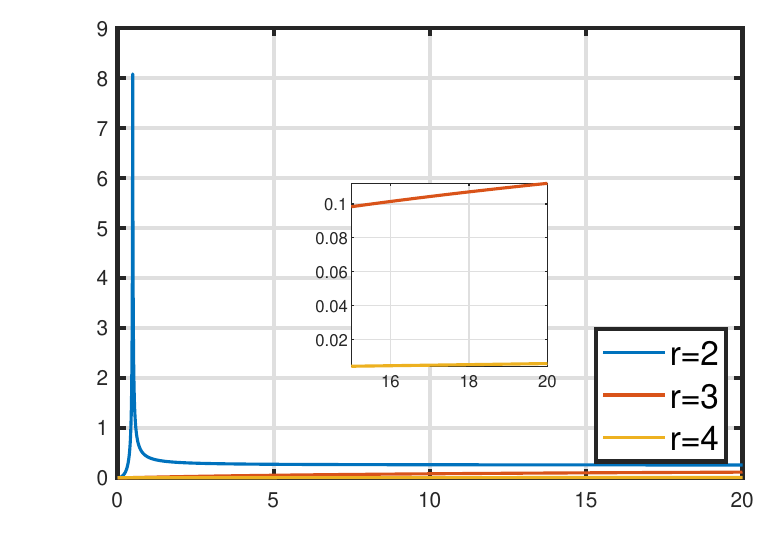}
    \caption{The area loss $ \left|A^h(\Vec{X}^m)-A^h(\Vec{X}^0) \right| $ for the BDF1-CSAV at r = 2,3,4 under surface energy $\gamma(\theta) = 1+ \beta \cos(4\theta)$: $\beta =0,\frac{1}{20},\frac{1}{10}$. Parameters are chosen as $N = 80$, $\Delta t = \frac{1}{160}$.}
    \label{fig:CSAV_area_SDF}
\end{figure}

    %area 关于r的变化%
    
    %energy 关于t的变化$

    \begin{figure}[!htp]
         \centering
    \includegraphics[width=0.33\textwidth]{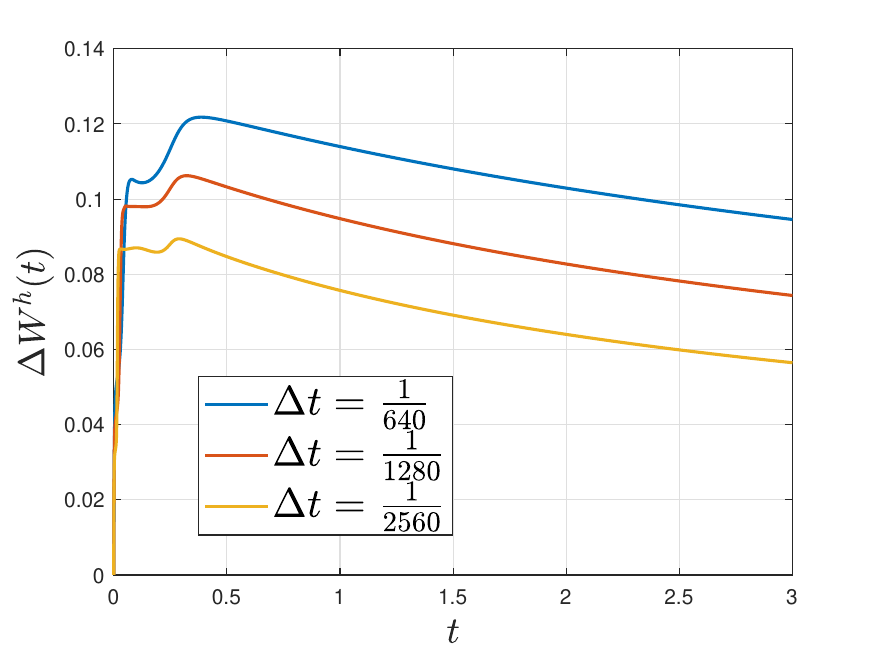}
    \includegraphics[width=0.33\textwidth]{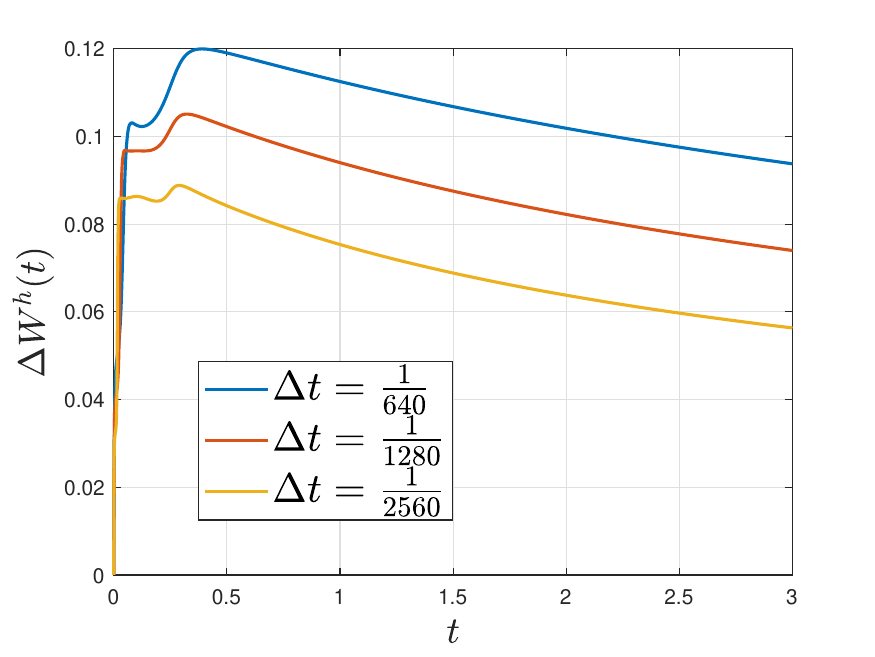}
    \includegraphics[width=0.33\textwidth]{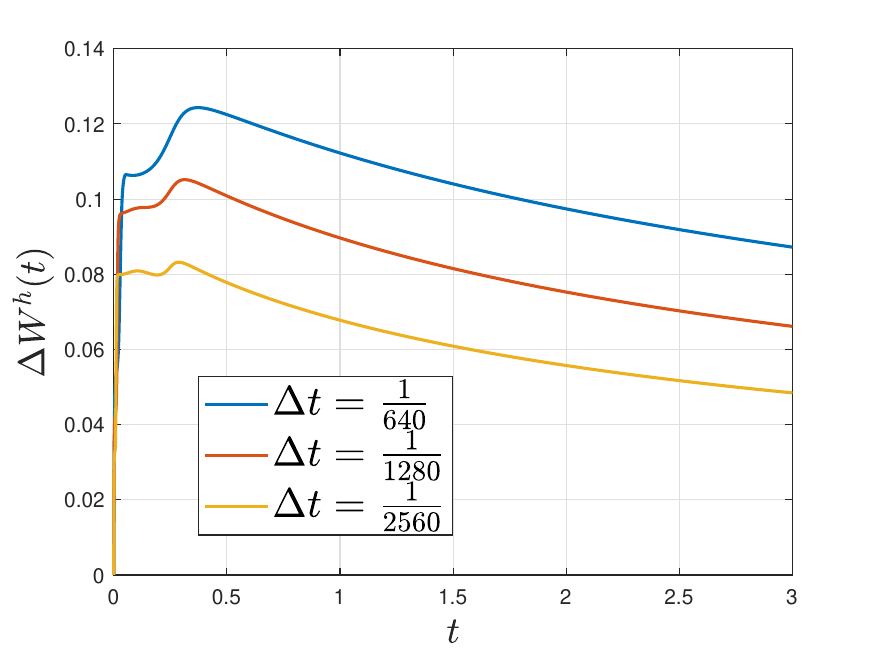}
    
    \caption{Temporal evolution of $\Delta W^h(t)$ for the three schemes with $\gamma(\theta)=1+\frac{1}{10}\cos(4\theta)$. Parameters are chosen as $N=640$, $r=3$.}
    \label{fig:ener_dis_SDF}
    \end{figure}

    %energy 关于t的变化$ 
    
%% example3 面积关于r值得变化以及能量关于时间步长 %%

\subsection{SSD}
We in this subsection utilize three schemes to conduct a series of relevant experiments for the SSD. The specific results are as follows:
\begin{itemize}
    \item In the convergence testing for SSD, we maintain the selection of $r$ values. It is clear from the Figure \ref{fig:order_SSD} that the schemes preserve the temporal order of high-order schemes.
    \item Figure \ref{fig:SSD_area_scheme} depicits the temporal evolution of area loss for the three schemes in terms of SSD problem, measured at different values of $r$. The figure demonstrates that as $r$ value increases from 3 to 6, the area loss for the different schemes tend to stabilitize, with the BDF1-CSAV scheme consistently maintaining its superiority in area conservation property.
    \item In Figure \ref{fig:relate_ener_SSD}, the three schemes are depicted for varying surface energy densities. Additionally, combining with Figure \ref{fig:SSD_area_scheme}, it is noted that the scheme with lower area loss also demonstrates less decline in the original energy.
    \item Figure \ref{fig:mesh_SSD} illustrates the temporal evolution of the mesh ratio in SSD, showing that for isotropic surface energy densities, the mesh ratio consistently converges to 1 over time. In contrast, for anisotropic surface energy densities, the mesh ratios of all three schemes converges to a constant C $\approx 1.98$.
    \item Figure \ref{fig:CSAV_area_SSD} illustrates that for SSD of open curves with different surface energy densities, the area loss decreases as the value of $r$ increases. Figure \ref{fig:ener_dis_SSD} shows that in the evolution process of strong anisotropic SDF, $\Delta W^h(t)$ initially exhibits irregular behavior for different time steps. However, it eventually decreases as the time step decreases.
\end{itemize}

%% SSD 实验 %%

  % 1 时间阶测试 %
  \begin{figure}[!htp]
         \centering
    \includegraphics[width=0.45\textwidth]{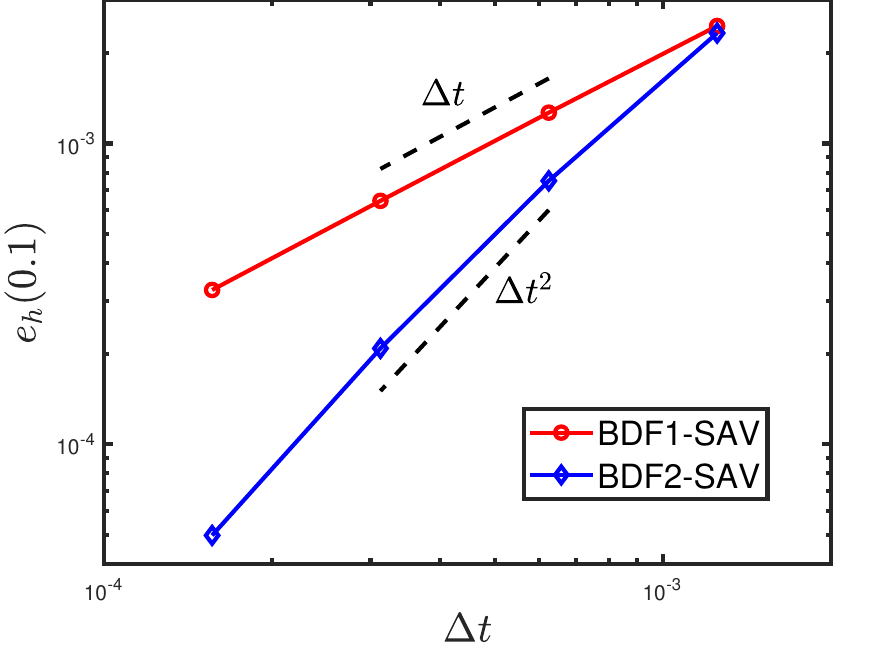}
    \includegraphics[width=0.45\textwidth]{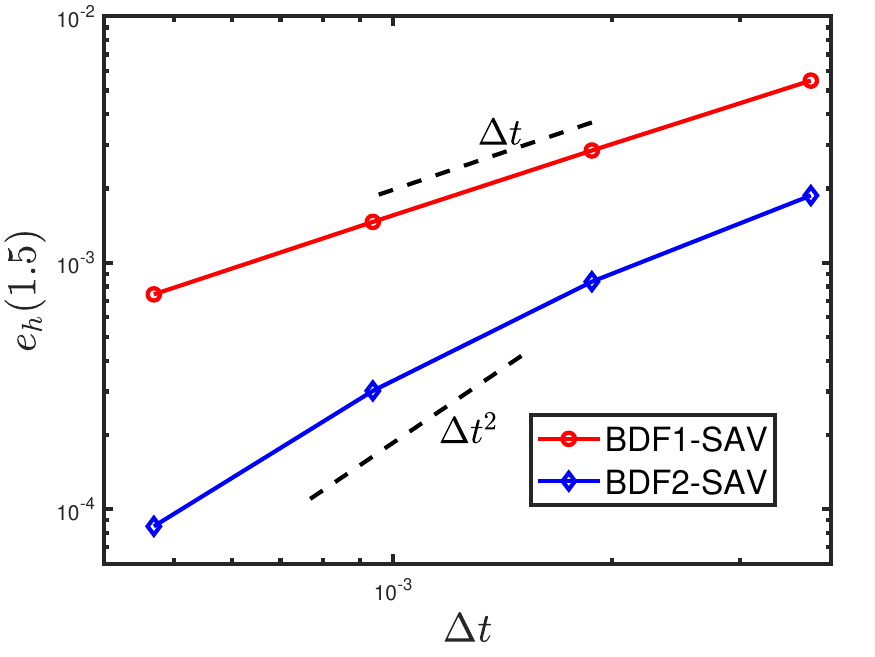}
    \includegraphics[width=0.45\textwidth]{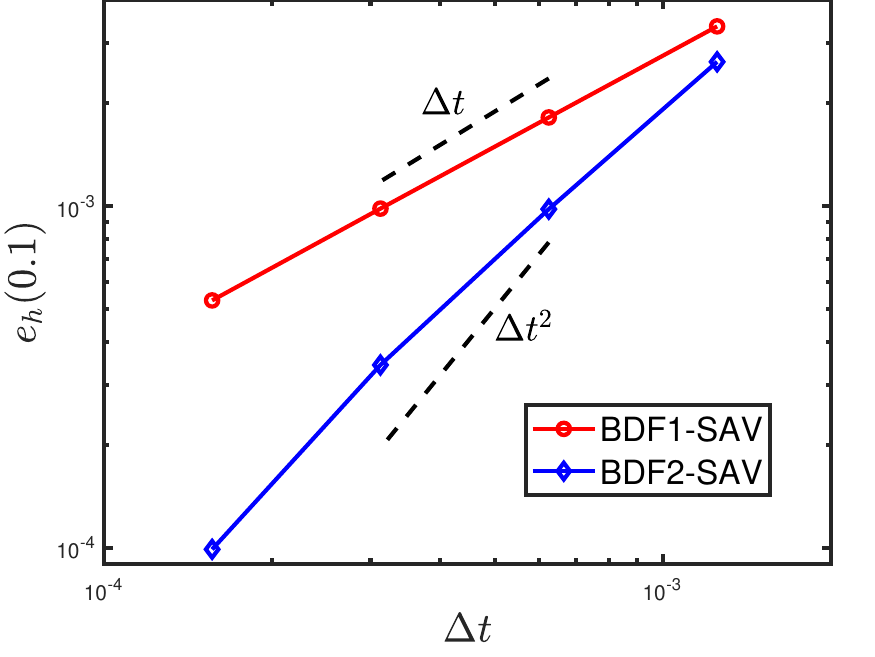}
    \includegraphics[width=0.45\textwidth]{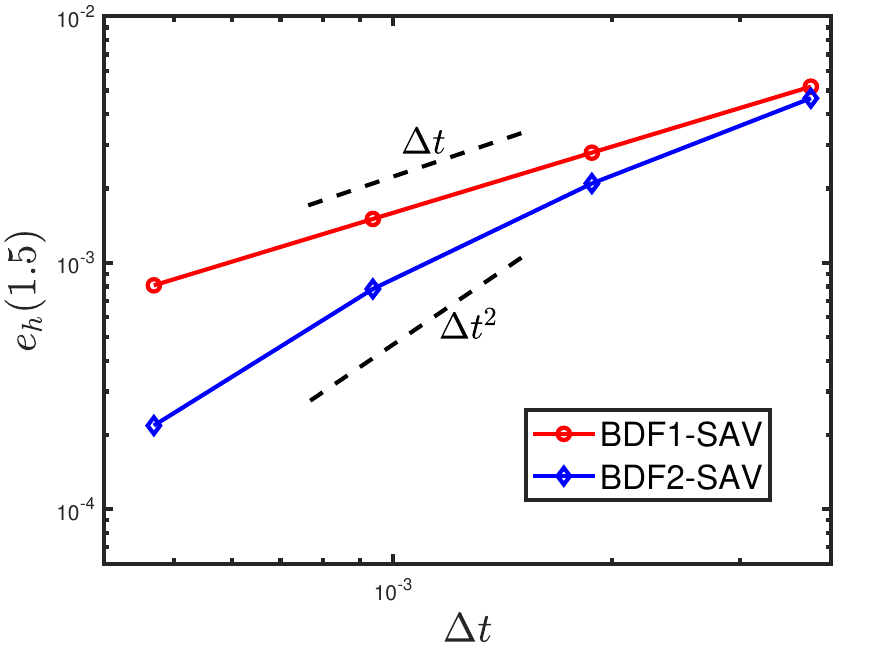}
    \caption{Convergence rates of BDF1-SAV and BDF2-SAV at times $T = 0.1$, $1.5$ with different surface energy densities: $\gamma(\theta)\equiv 1$ and $\gamma(\theta)=1+0.05\cos(4\theta)$, $\sigma = \cos(\frac{3}{4}\pi)$.}
    \label{fig:order_SSD}
    \end{figure}
    % 1 时间阶测试 %end

    % 2 面积损耗 %
    \begin{figure}[!htp]
         \centering
    \includegraphics[width=0.45\textwidth]{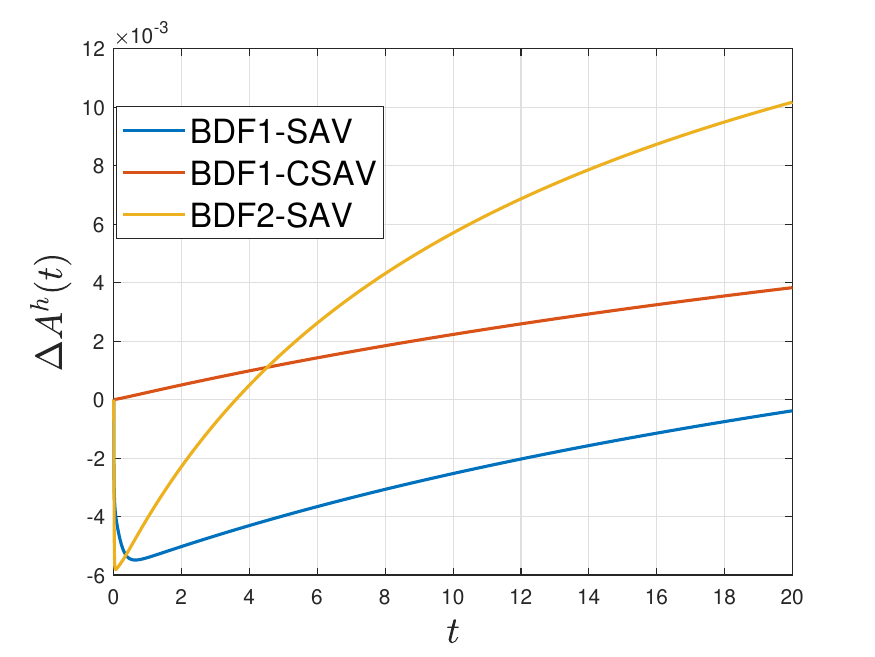}\quad
    \includegraphics[width=0.45\textwidth]{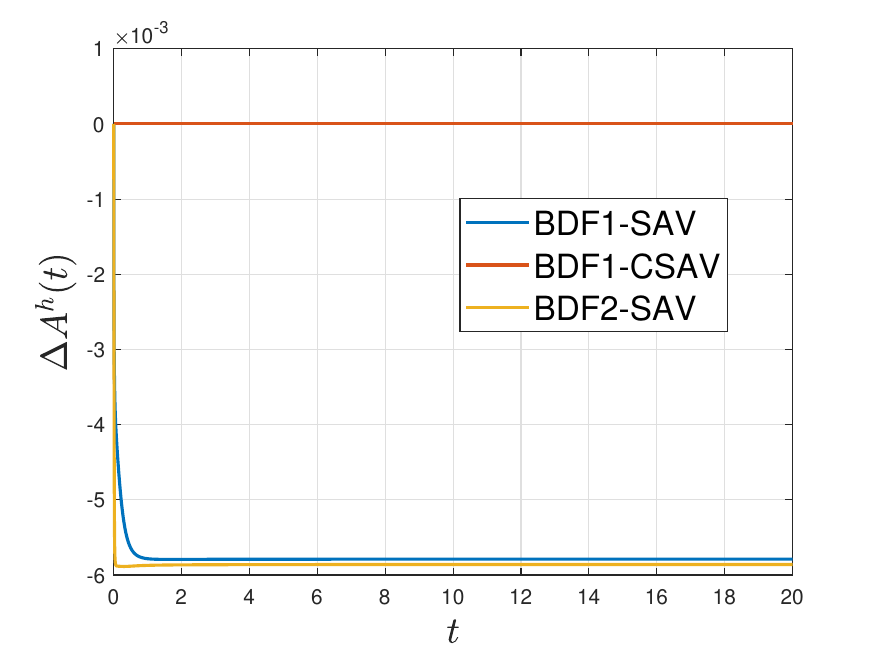}
    \includegraphics[width=0.45\textwidth]{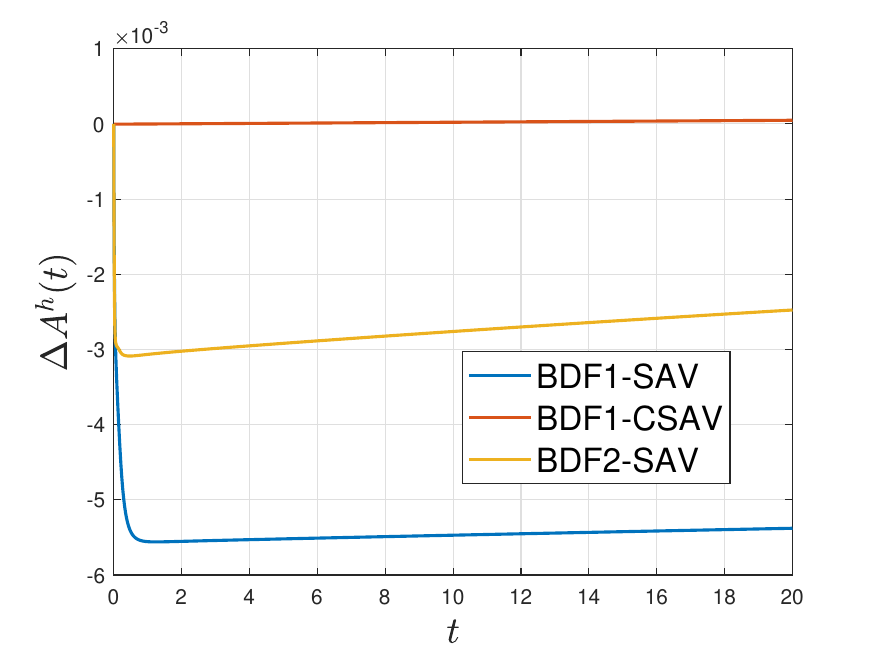}\quad
    \includegraphics[width=0.45\textwidth]{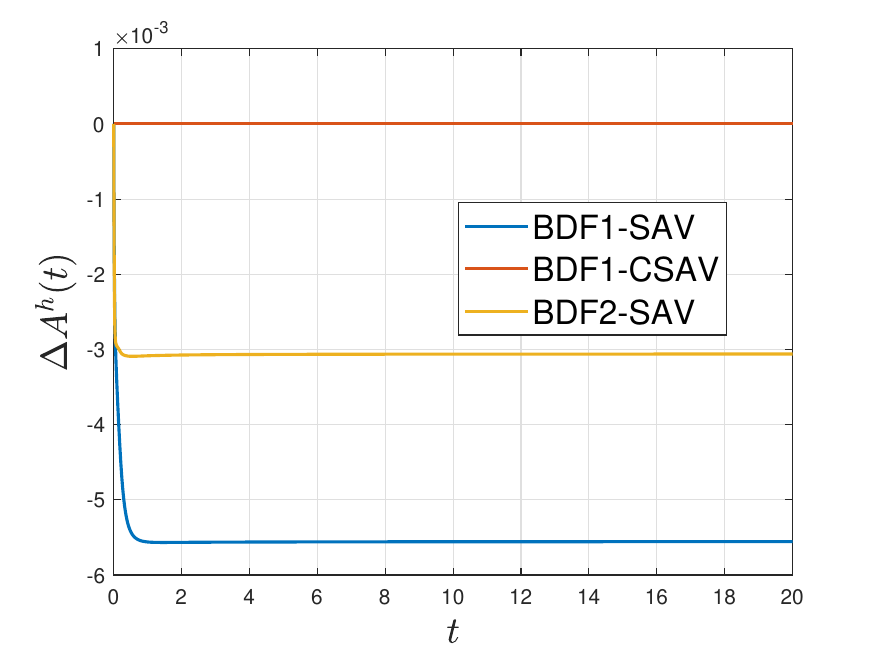}
    \includegraphics[width=0.45\textwidth]{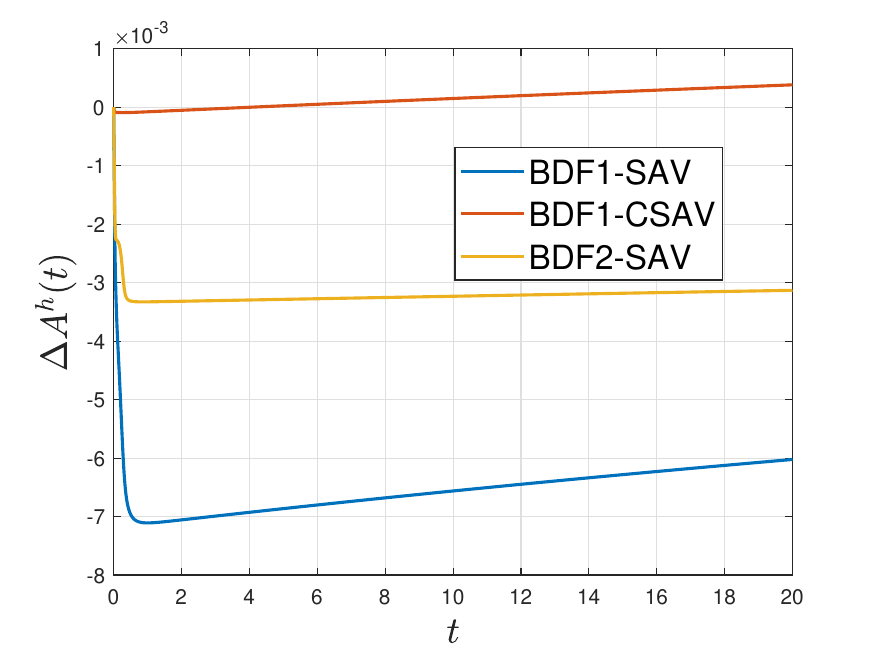}\quad
    \includegraphics[width=0.45\textwidth]{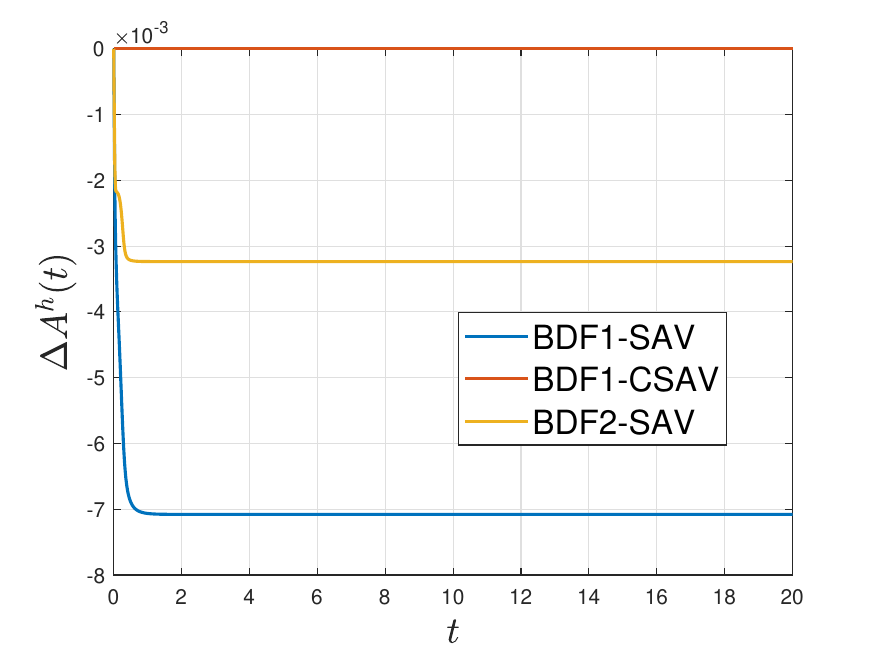}
    \caption{The relative area loss $\Delta A^h(t)$ for the three schemes at $r = 3$ and $r = 6$ under surface energy $\gamma(\theta) = 1+ \beta \cos(4\theta)$: $\beta =0,\frac{1}{20},\frac{1}{10}$. Parameters are chosen as $N = 80$, $\Delta t = \frac{1}{160}$, $\sigma = \cos(\frac{3}{4}\pi)$.}
    \label{fig:SSD_area_scheme}
\end{figure}
    % 2面积损耗 % end

   % 3能量相对性检验 %
\begin{figure}[!htp]
         \centering
    \includegraphics[width=0.45\textwidth]{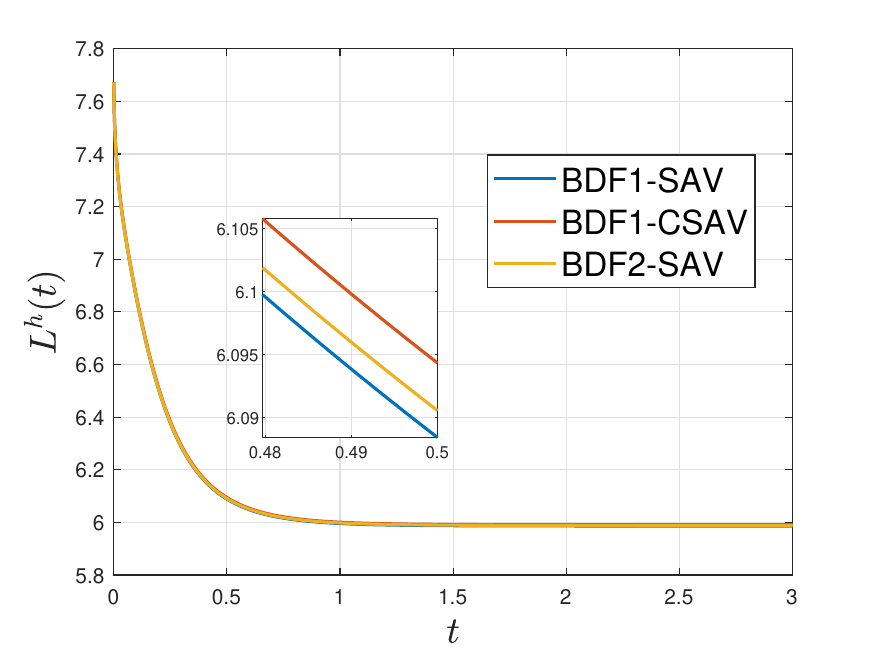}
    \includegraphics[width=0.45\textwidth]{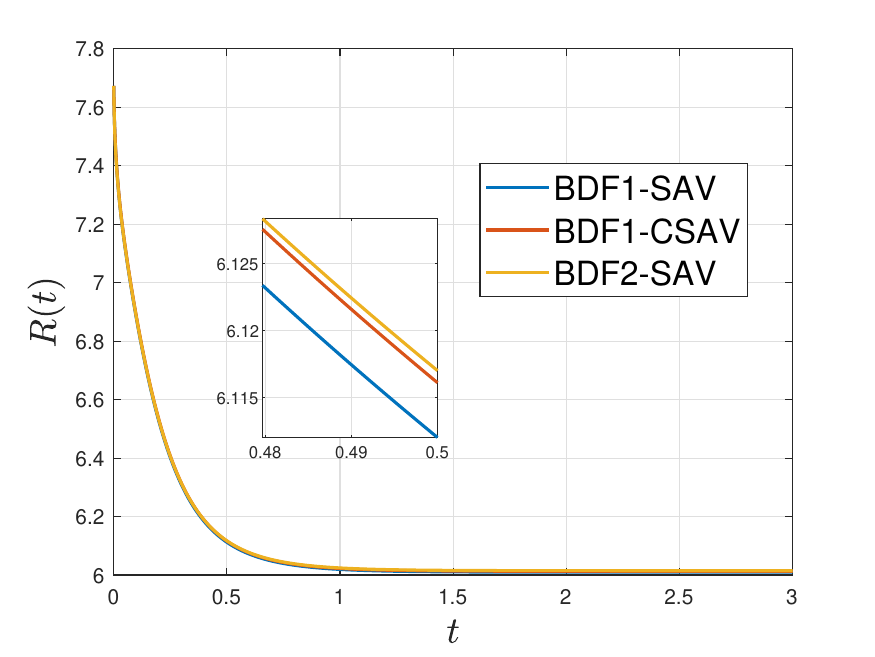}
    \includegraphics[width=0.45\textwidth]{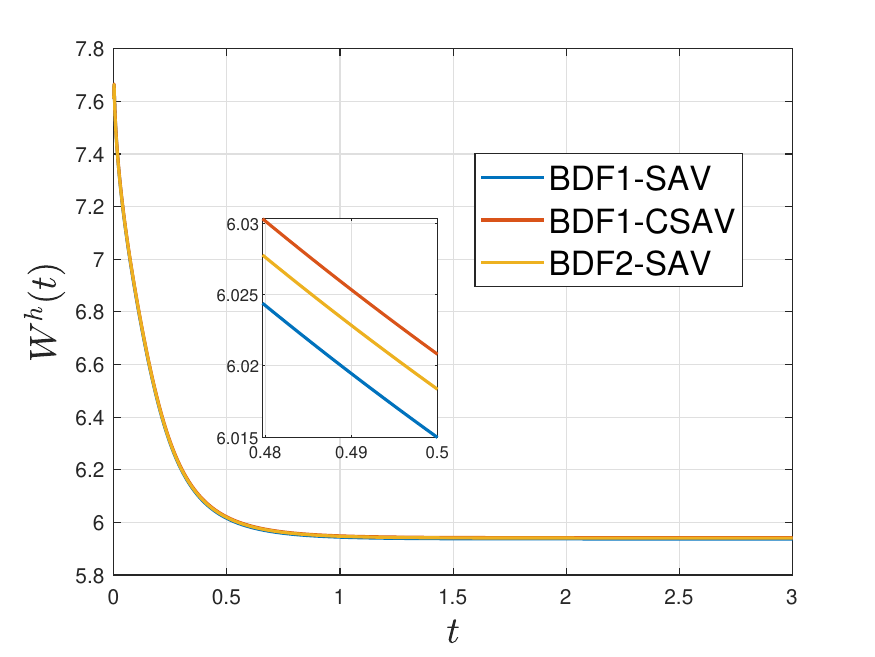}\quad
    \includegraphics[width=0.45\textwidth]{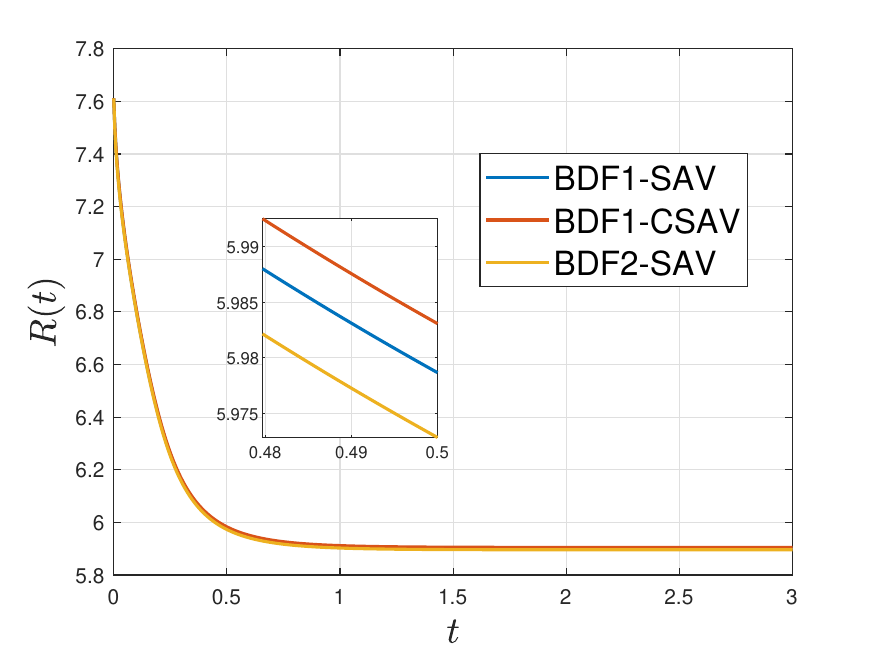}
   \includegraphics[width=0.45\textwidth]{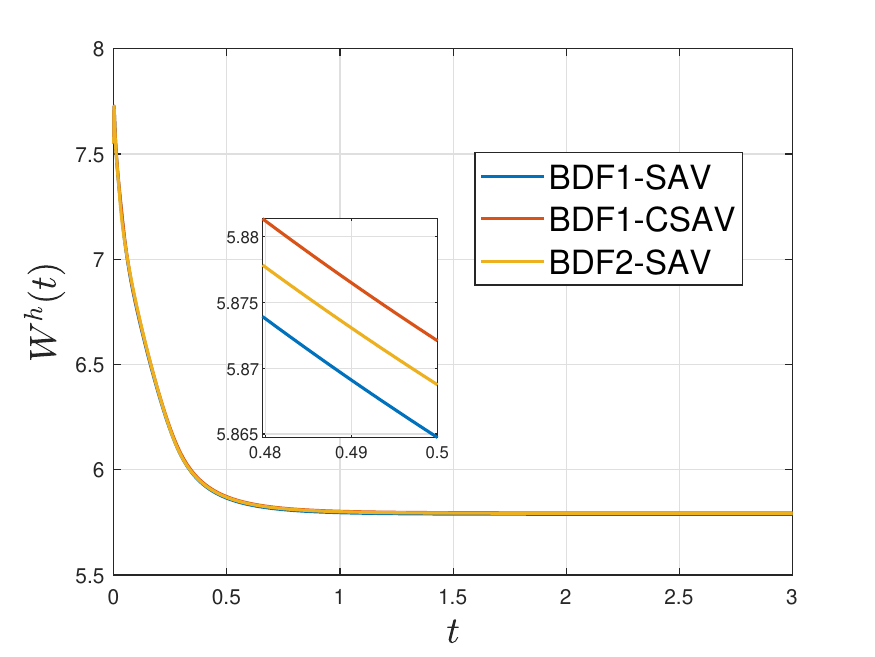}\quad
    \includegraphics[width=0.45\textwidth]{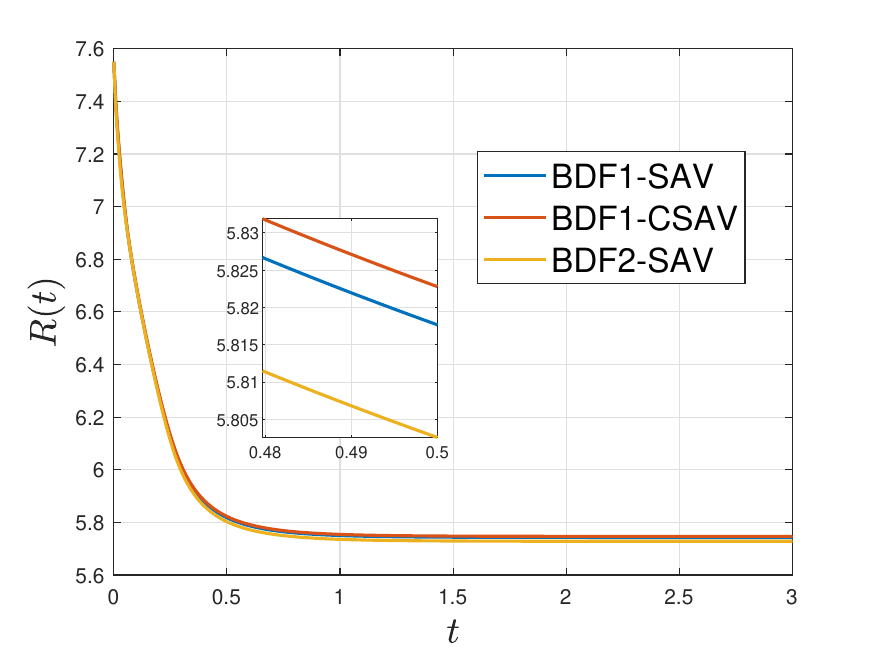}
    \caption{The original energy and modified energy for the three schemes under surface energy $\gamma(\theta) = 1+ \beta \cos(4\theta)$: $\beta =0,\frac{1}{20},\frac{1}{10}$. Parameters are chosen as $N = 640$, $\Delta t = \frac{1}{640}$, $r = 6$, $\sigma = \cos(\frac{3}{4}\pi)$.}
    \label{fig:relate_ener_SSD}
\end{figure}
  % 3能量相对性检验 % end
  
  % 4 网格比 %
\begin{figure}[!htp]
         \centering
    \includegraphics[width=0.45\textwidth]{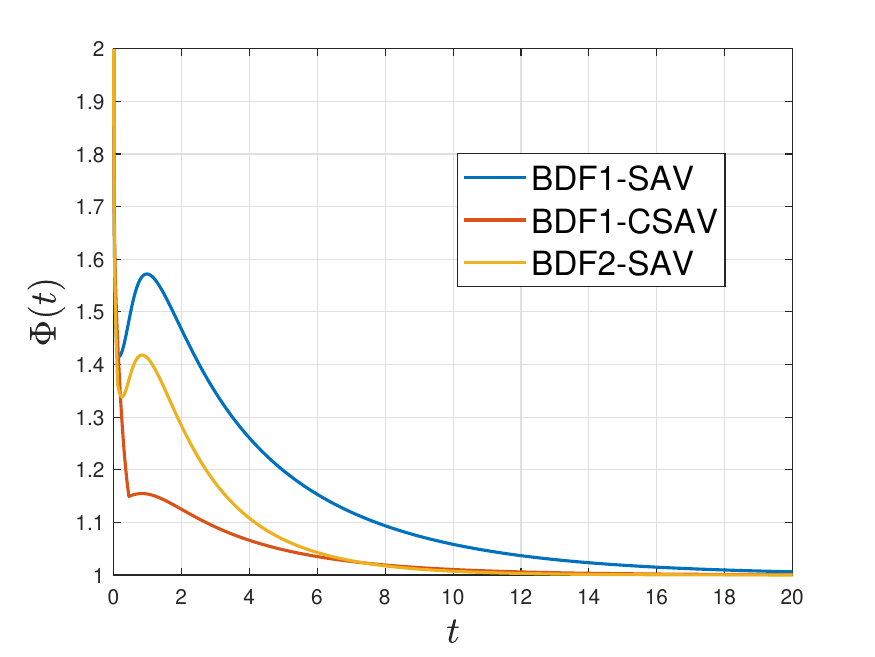}
    \includegraphics[width=0.45\textwidth]{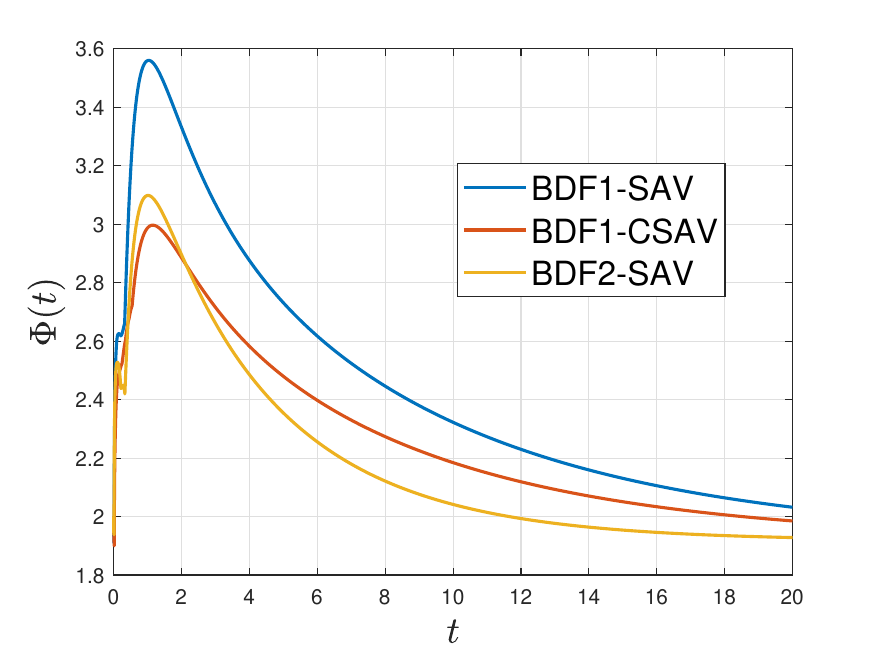}
    
    \caption{Temporal evolution of the mesh ratio $\Psi(t)$ for the three schemes with different surface energy $\gamma(\theta)=1+\beta \cos(4\theta)$: $\beta = 0, \frac{1}{20}.$ Parameters are chosen as $N=2^{-7}$, $\Delta t = 10^{-3}$, $r = 3$, $\sigma = \cos(\frac{3}{4}\pi)$.}
    \label{fig:mesh_SSD}
    \end{figure}
    % 4 网格比 % end

    % 5 面积 r %
    \begin{figure}[!htp]
         \centering
    \includegraphics[width=0.33\textwidth]{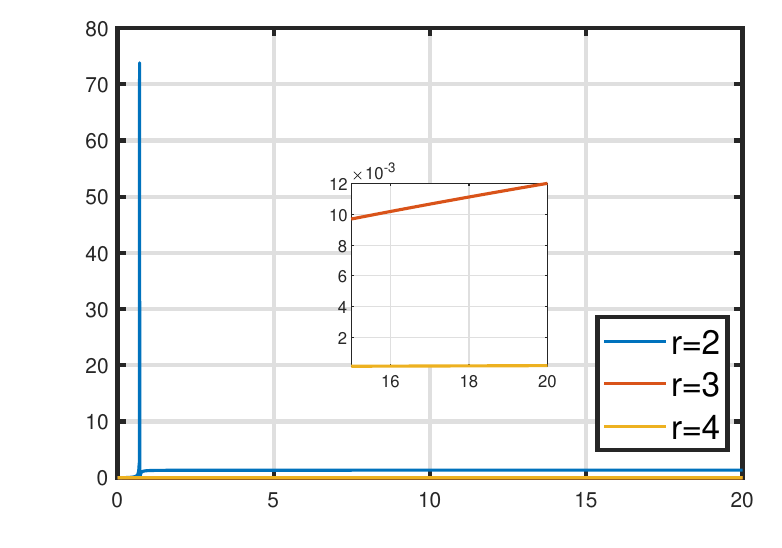}
    \includegraphics[width=0.33\textwidth]{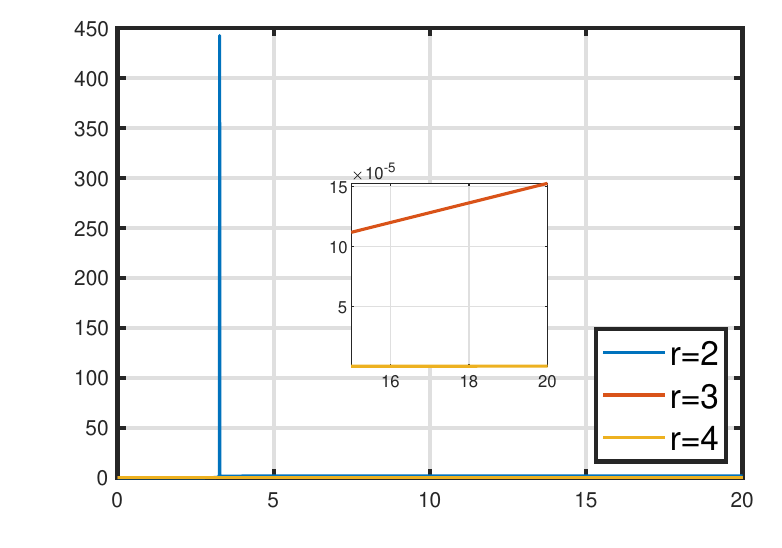}
    \includegraphics[width=0.33\textwidth]{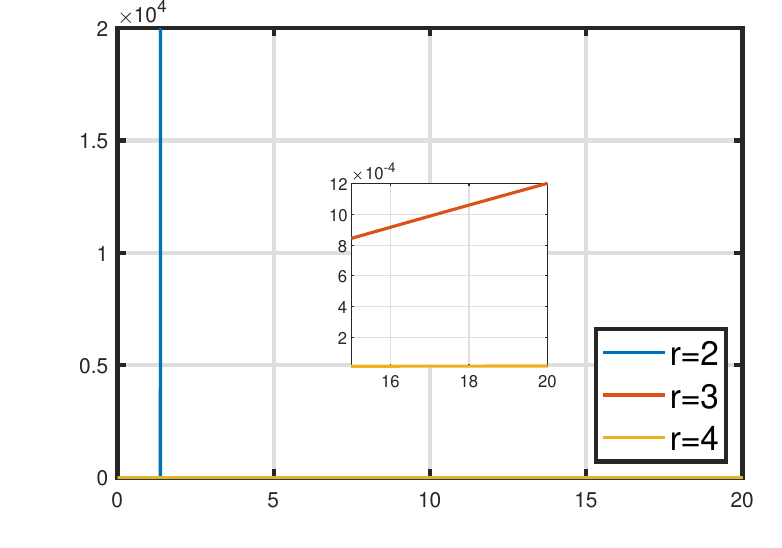}
    \caption{The area loss $ \left|A^h(\Vec{X}^m)-A^h(\Vec{X}^0) \right| $ for the BDF1-CSAV at r = 2,3,4 under surface energy $\gamma(\theta) = 1+ \beta \cos(4\theta)$: $\beta =0,\frac{1}{20},\frac{1}{10}$. Parameters are chosen as $N = 80$, $\Delta t = \frac{1}{160}$, $\sigma = \cos(\frac{3}{4}\pi)$.}
    \label{fig:CSAV_area_SSD}
\end{figure}
    % 5 面积 r % end
    
    % 6 能量delta t %
    \begin{figure}[!htp]
         \centering
    \includegraphics[width=0.33\textwidth]{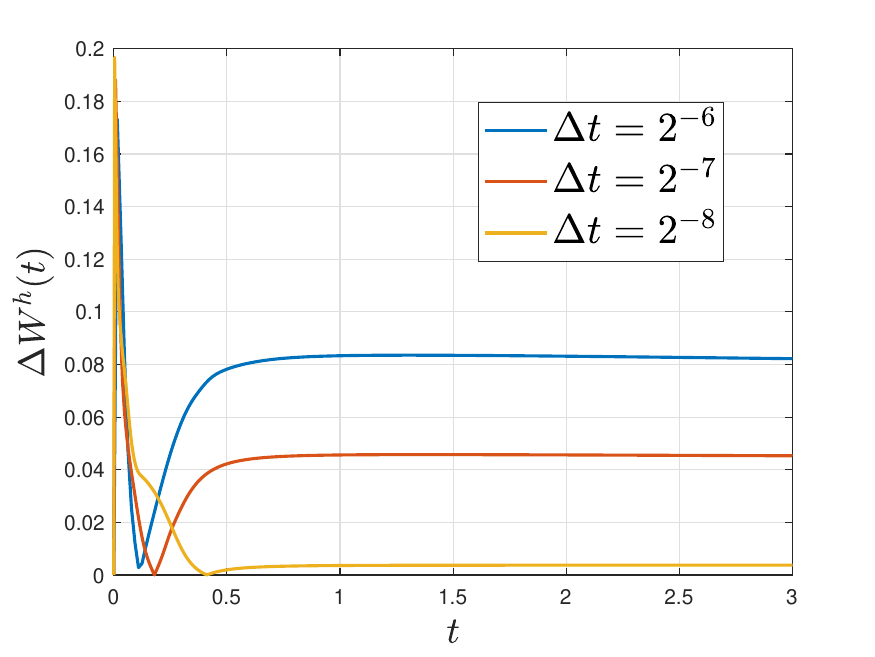}
    \includegraphics[width=0.33\textwidth]{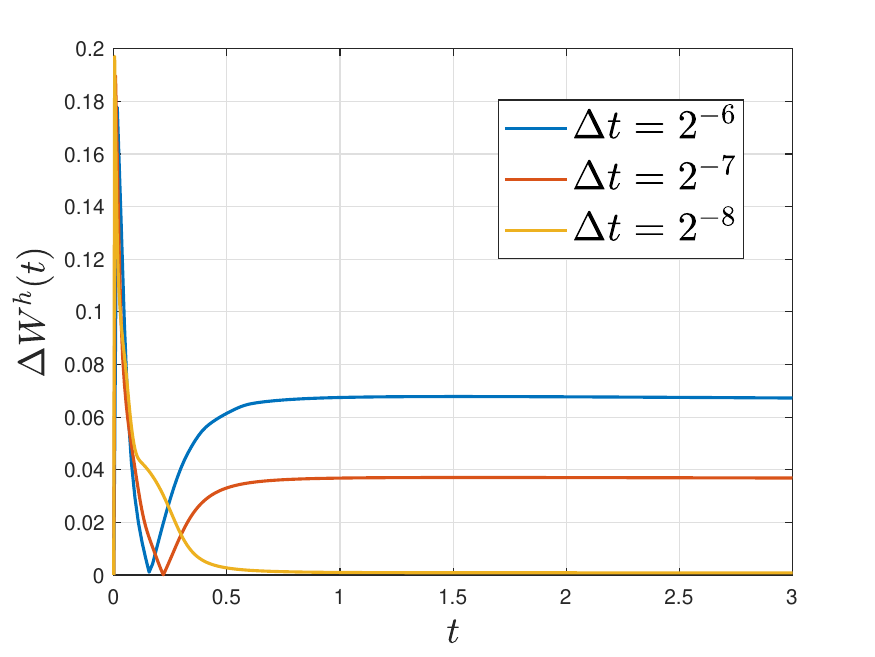}
    \includegraphics[width=0.33\textwidth]{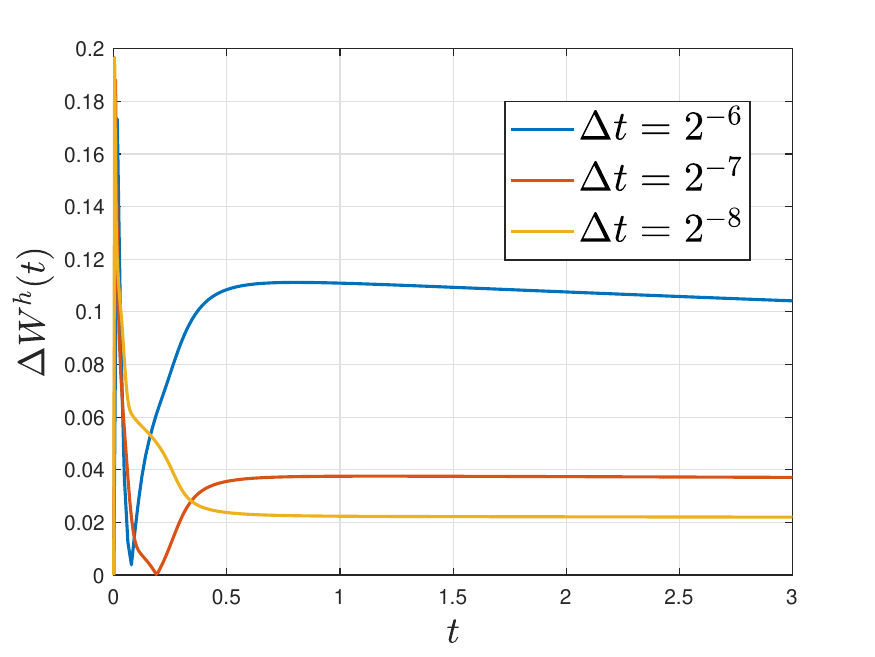}
    
    \caption{Temporal evolution of $\Delta W^h(t)$ for the three schemes with $\gamma(\theta)=1+\frac{1}{10}\cos(4\theta)$. Parameters are chosen as $N=64$, $r=3$, $\sigma = \cos(\frac{3}{4}\pi)$.}
    \label{fig:ener_dis_SSD}
    \end{figure}
    % 6 能量delta t %end

%% end 固态去湿实验 %%

\subsection{Applications to morphological evolution.}
Finally, we depict the morphological evolution of SDF and SSD for three schemes under different surface energy densities, starting with a rectangular initial shape. Figure \ref{fig:SDF_evo} illustrates the morphological evolution of SDF, clearly showing that the equilibrium shapes for different schemes are consistent. Moreover, equilibrium shapes for isotropic surface energy and anisotropic surface energy are circular and elliptical, respectively. Meanwhile, Figure \ref{fig:SSD_evo} depicts the morphological evolution of SSD.
%% 演化图 %%
  % 表面扩散 %

  \begin{figure}[!htp]
         \centering
    \includegraphics[width=0.33\textwidth]{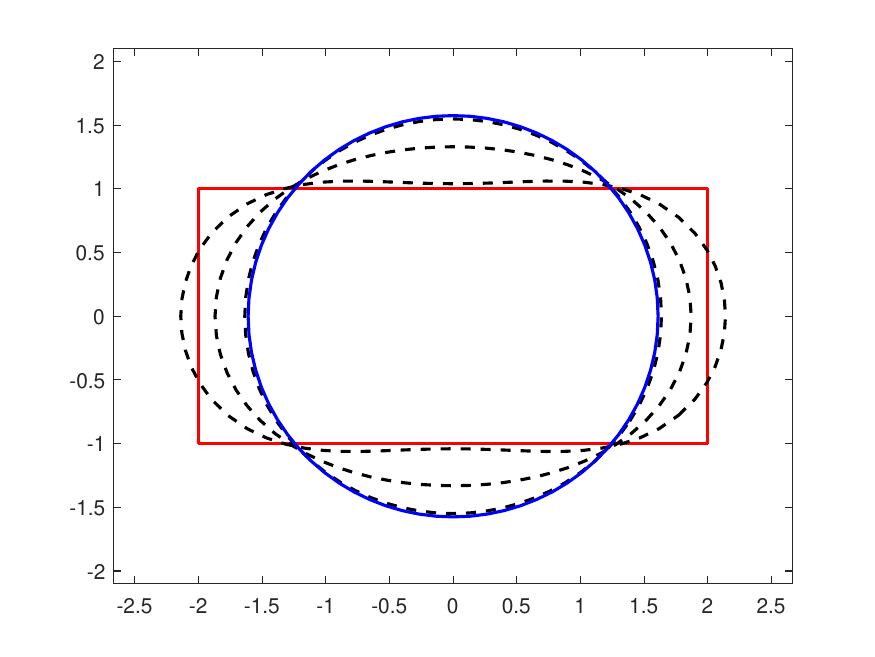}
    \includegraphics[width=0.33\textwidth]{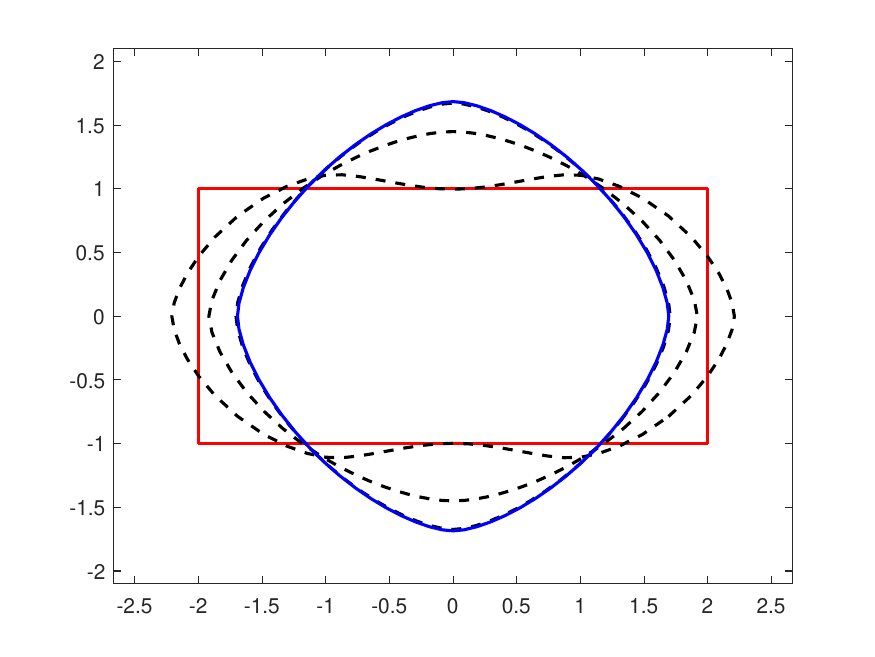}
    \includegraphics[width=0.33\textwidth]{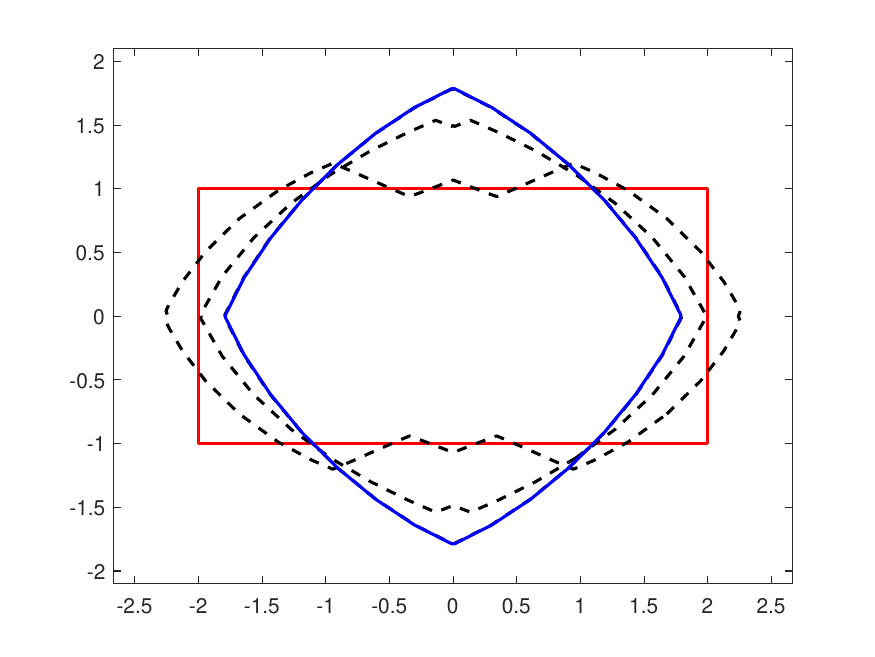}
    \includegraphics[width=0.33\textwidth]{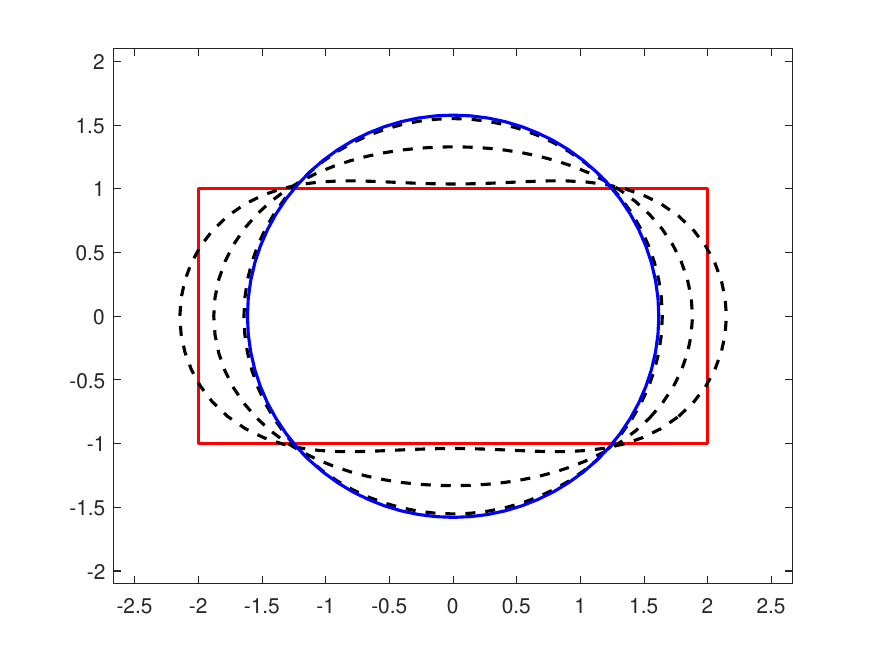}
    \includegraphics[width=0.33\textwidth]{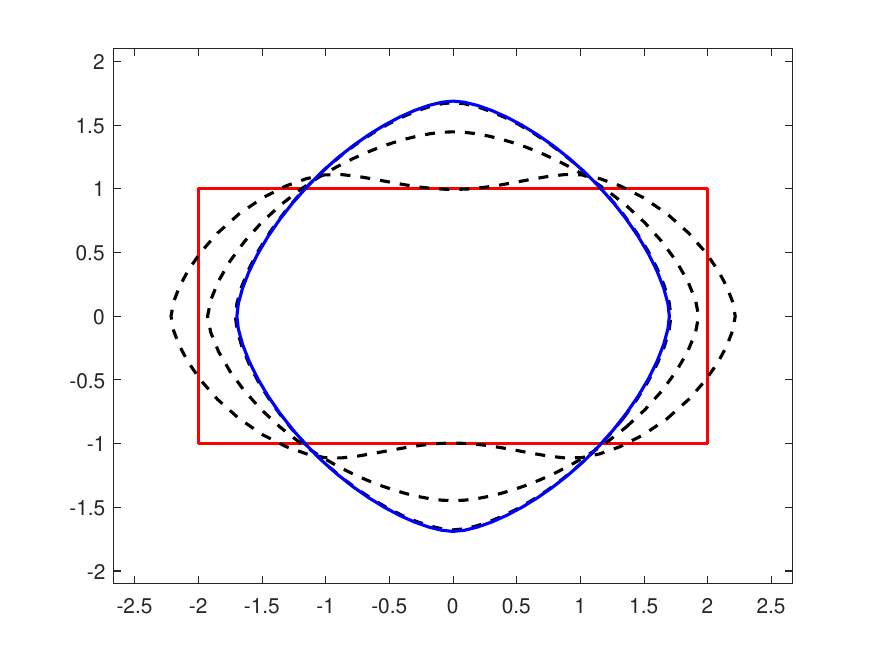}
    \includegraphics[width=0.33\textwidth]{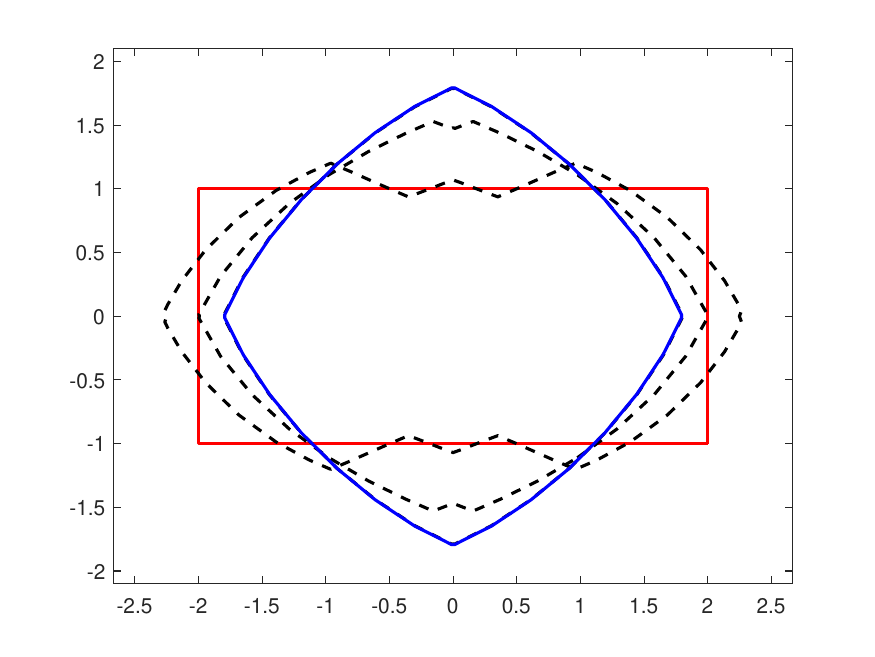}
    \includegraphics[width=0.33\textwidth]{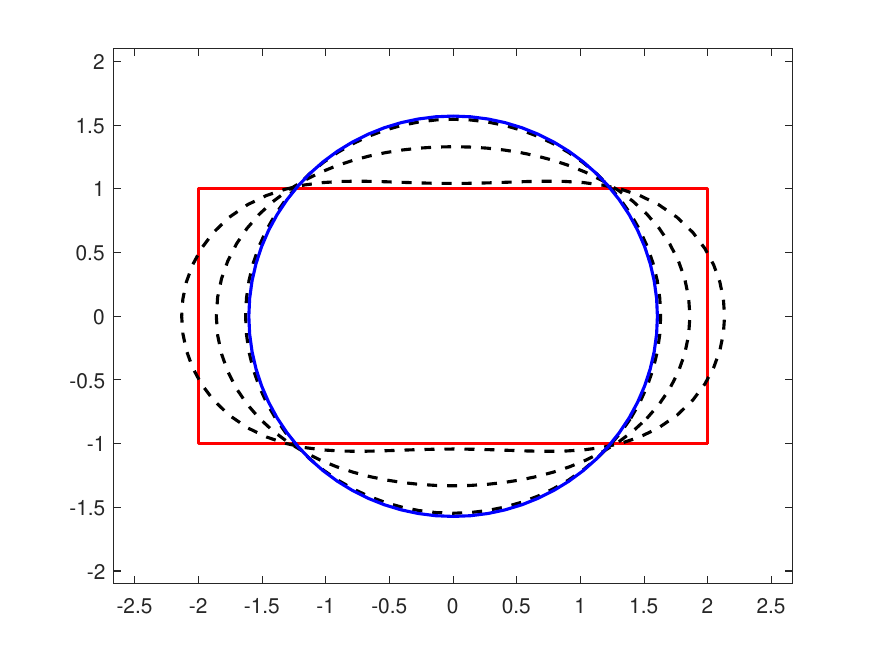}
    \includegraphics[width=0.33\textwidth]{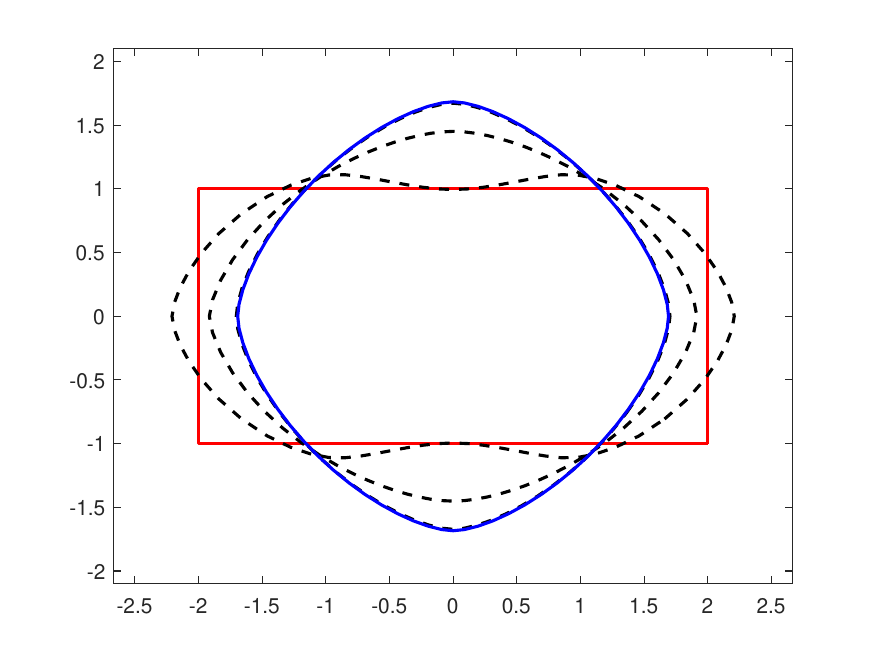}
    \includegraphics[width=0.33\textwidth]{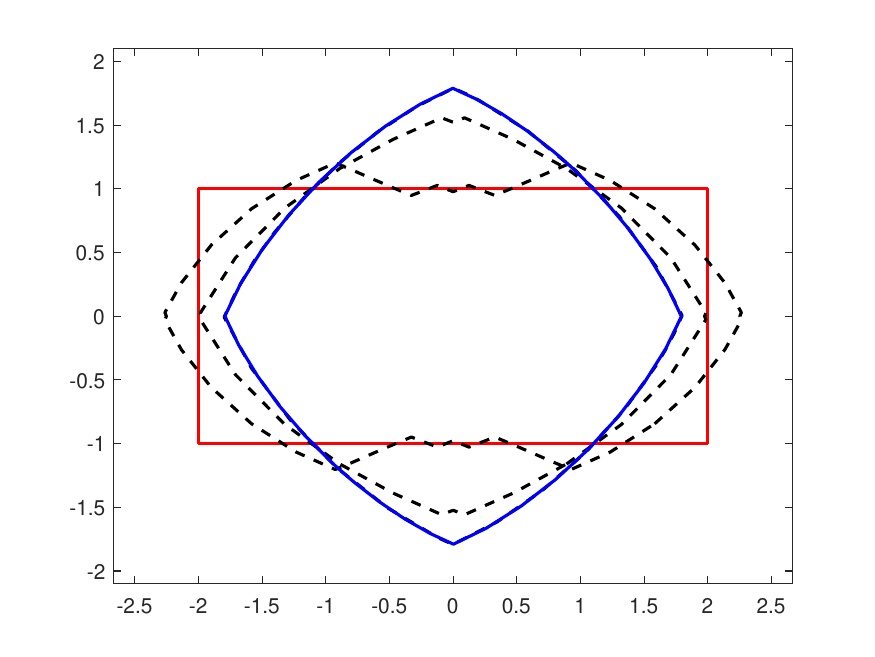}
    \caption{The morphological evolution of SDF for the three schemes under different surface energy densities $\gamma(\theta)=1+\beta \cos(4\theta)$, $\beta =0$, $\frac{1}{20}$, $\frac{1}{10}$. Other parameters are chosen as $N = 72$, $\Delta t = 10^{-3}$, $r = 3$.
}
    \label{fig:SDF_evo}
    \end{figure}
  % 表面扩散 %
  % 固态去湿 %
  \begin{figure}[!htp]
         \centering
    \includegraphics[width=0.33\textwidth]{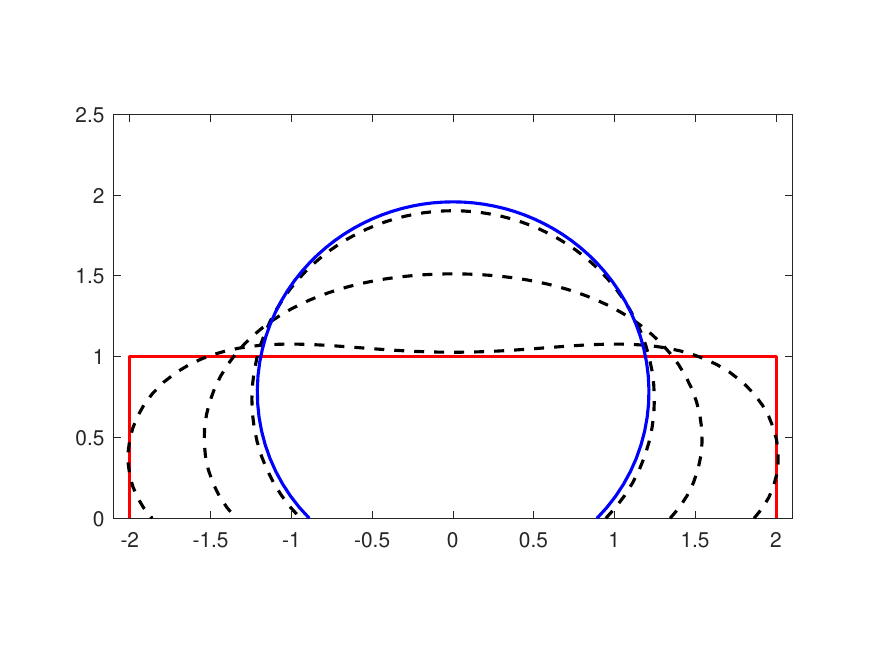}
    \includegraphics[width=0.33\textwidth]{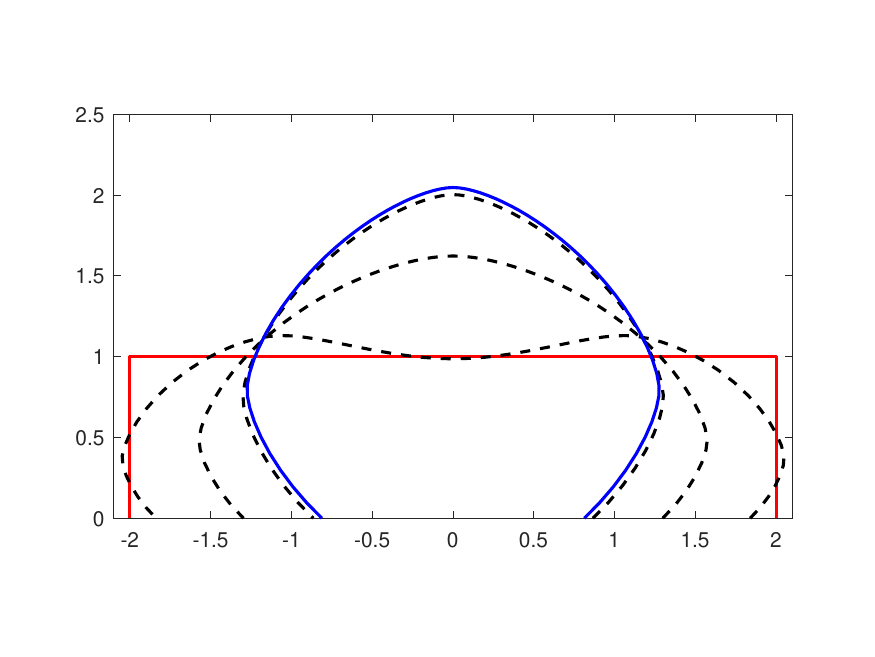}
    \includegraphics[width=0.33\textwidth]{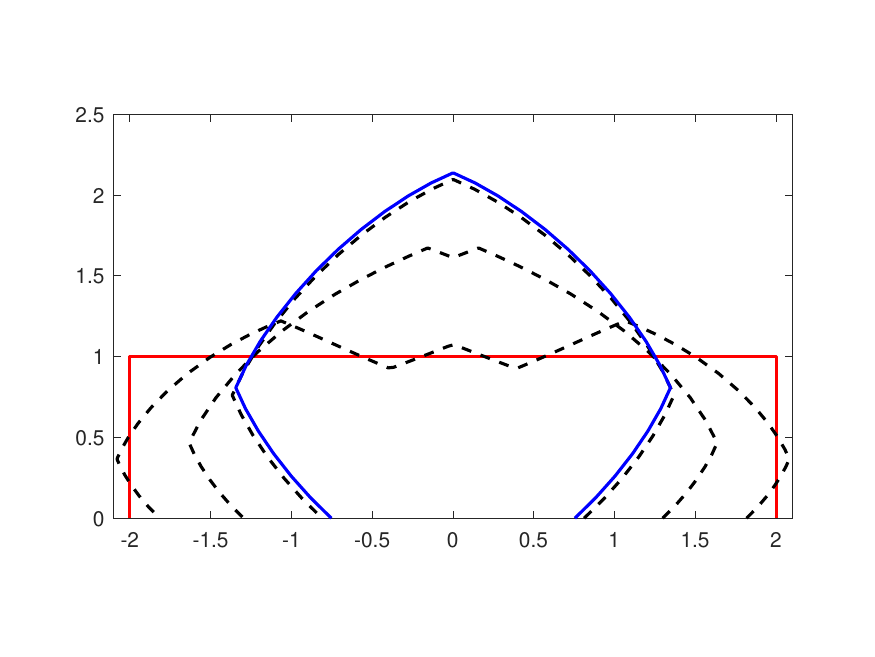}
    \includegraphics[width=0.33\textwidth]{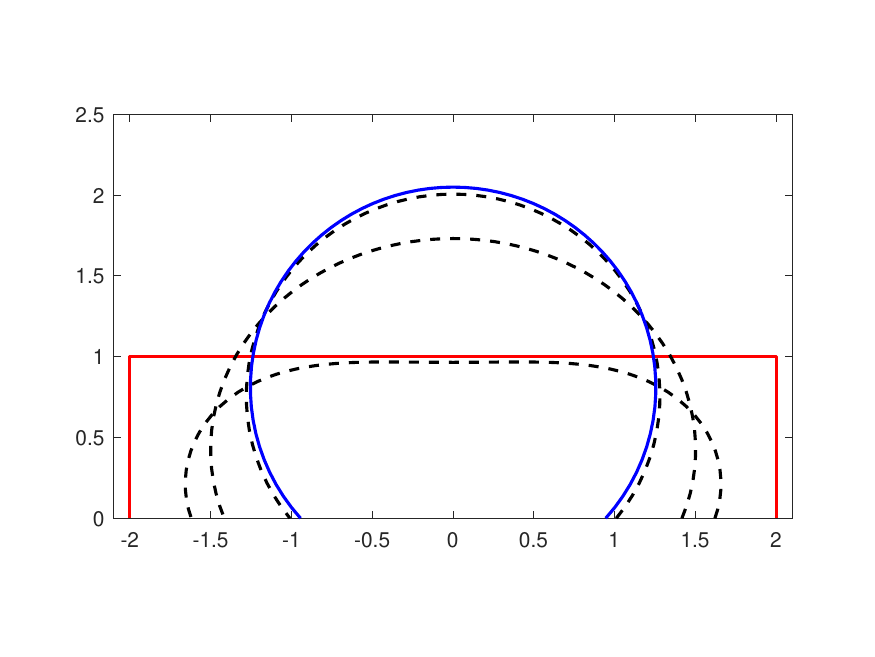}
    \includegraphics[width=0.33\textwidth]{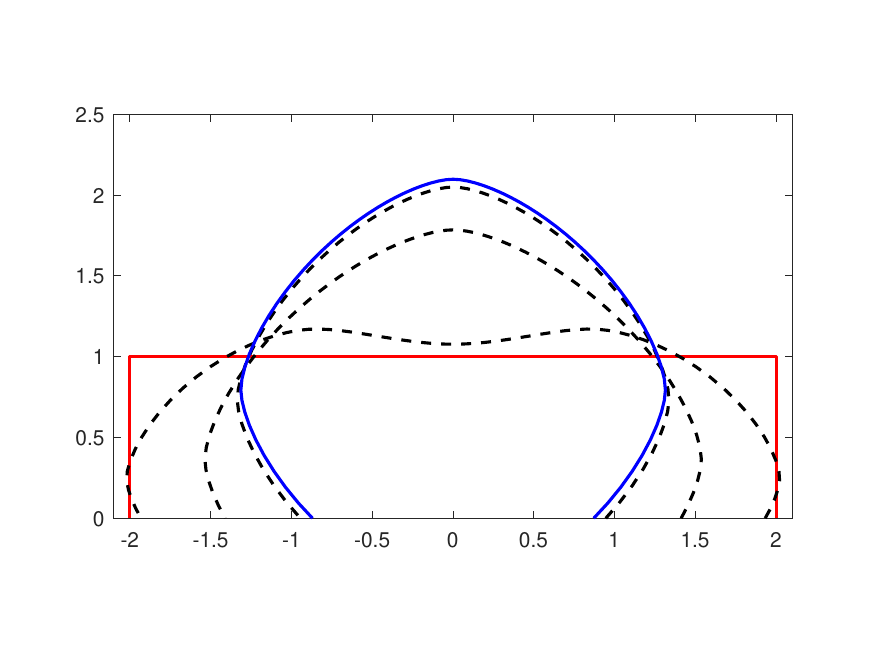}
    \includegraphics[width=0.33\textwidth]{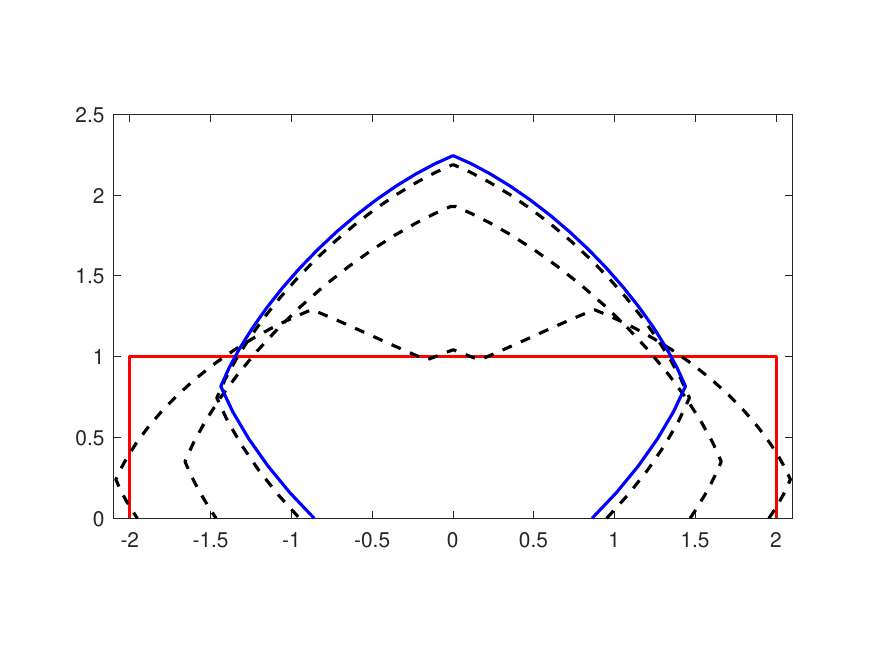}
    \includegraphics[width=0.33\textwidth]{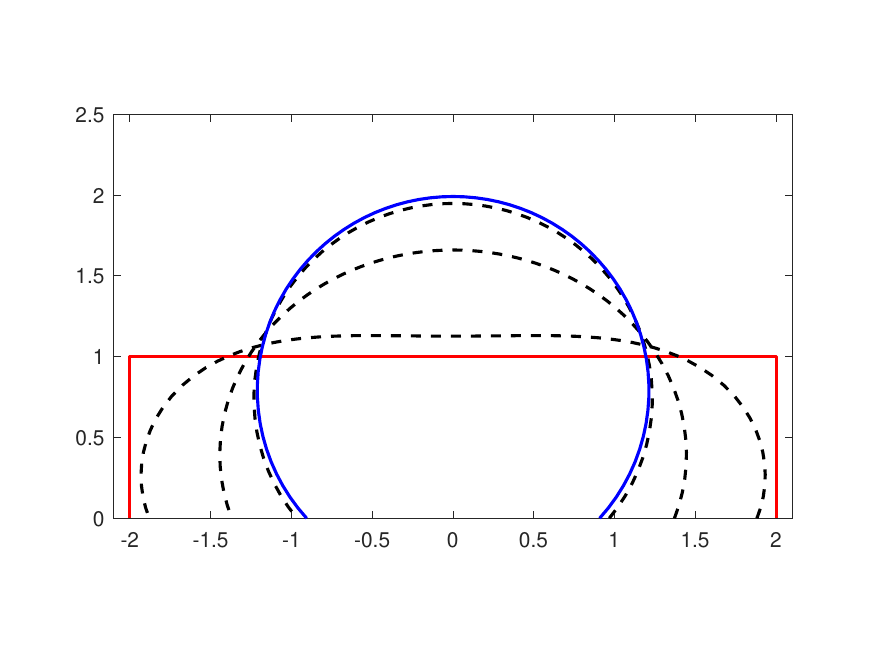}
    \includegraphics[width=0.33\textwidth]{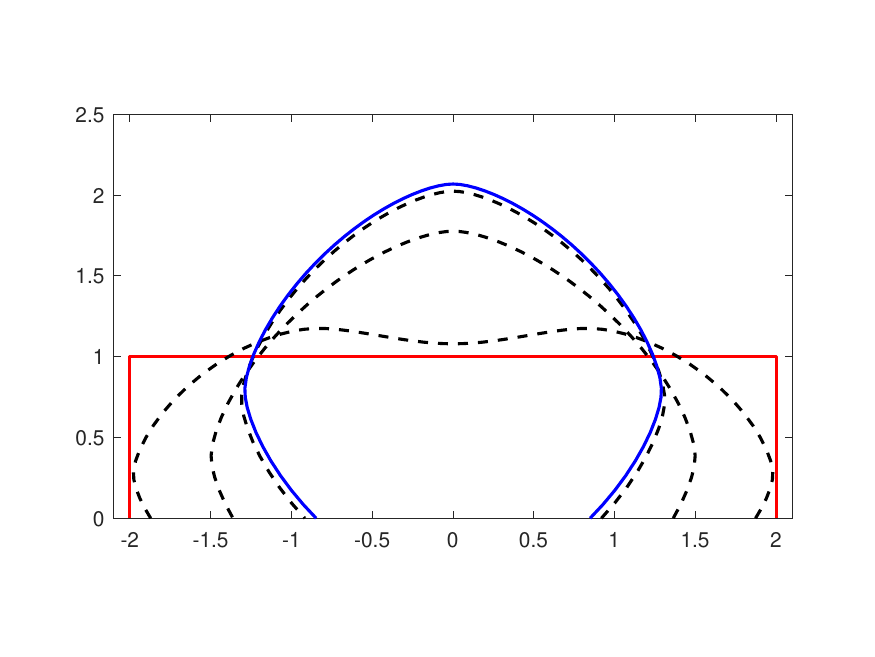}
    \includegraphics[width=0.33\textwidth]{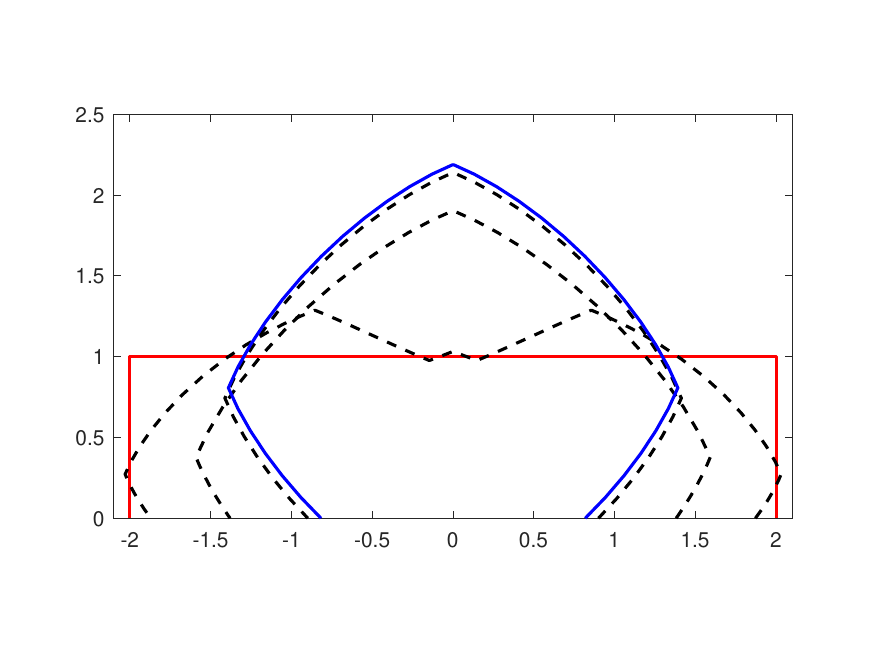}
    \caption{The morphological evolution of SSD for the three schemes under different surface energy densities $\gamma(\theta)=1+\beta \cos(4\theta)$, $\beta =0$, $\frac{1}{20}$, $\frac{1}{10}$. Other parameters are chosen as $N = 72$, $\Delta t = 10^{-3}$, $r = 3$, $\sigma = \cos(\frac{3}{4}\pi)$.}
    \label{fig:SSD_evo}
    \end{figure}

  % 固态去湿 %
%% 演化图 %%

%%chapter4 数值实验
% chapter5 结论
\section{Conclusions} \label{sec:conclusion}
Based on the SAV approach, we established the energy-stable parametric finite approximations for the SDF and SSD. The boundedness of the discrete original energy and the stability of the modified energy for the schemes were proven, as well as the approximately area-conservation of the BDF1-CSAV scheme. Finally, our numerical results demonstrated advantages of the proposed schemes. 
In the near future, we will employ the energy-stable methods  proposed in this works to other curvature flows, such as mean curvature flow and Willmore flow in two-dimensional and three-dimensional spaces.
\bibliographystyle{elsarticle-num}
\bibliography{references}
\end{document}